\documentclass[notitlepage,11pt,reqno]{amsart}
\usepackage[foot]{amsaddr}
\usepackage{amssymb,nicefrac,bm,upgreek,mathtools,verbatim}
\usepackage[final]{hyperref}
\usepackage[mathscr]{eucal}
\usepackage{dsfont}
\usepackage[normalem]{ulem}
\usepackage{amsopn,esint}
\usepackage[T1]{fontenc}
\usepackage{appendix}
\usepackage[shortlabels]{enumitem}
\usepackage[T1]{fontenc}
\usepackage[utf8]{inputenc}
\usepackage[utf8]{inputenc}
\usepackage{apptools}

\AtAppendix{\counterwithin{lemma}{section}} 

\usepackage[margin=.97in]{geometry}
\newcommand{\stkout}[1]{\ifmmode\text{\sout{\ensuremath{#1}}}\else\sout{#1}\fi}

\newtheorem{lemma}{Lemma}[section]
\newtheorem{theorem}{Theorem}[section]
\newtheorem{proposition}{Proposition}[section]

\theoremstyle{definition}
\newtheorem{definition}{Definition}[section]
\newtheorem{assumption}{Assumption}[section]

\newtheorem{remark}{Remark}[section]
\numberwithin{theorem}{section}
\numberwithin{equation}{section}

\hypersetup{
  colorlinks=true,
  citecolor=dred,
  linkcolor=mblue,
  urlcolor = blue,
  anchorcolor = blue,
  frenchlinks=false,
  pdfborder={0 0 0},
  naturalnames=false,
  hypertexnames=false,
  breaklinks}
\usepackage[capitalize,nameinlink]{cleveref}

\usepackage[abbrev,msc-links,nobysame,citation-order]{amsrefs}

\crefname{section}{Section}{Sections}
\crefname{appsec}{Appendix}{Appendices}
\crefname{subsection}{Section}{Sections}
\crefname{condition}{Condition}{Conditions}
\crefname{hypothesis}{Hypothesis}{Conditions}
\crefname{assumption}{Assumption}{Assumptions}
\crefname{lemma}{Lemma}{Lemmas}
\crefname{fact}{Fact}{Facts}

\Crefname{figure}{Figure}{Figures}

\crefformat{equation}{\textup{#2(#1)#3}}
\crefrangeformat{equation}{\textup{#3(#1)#4--#5(#2)#6}}
\crefmultiformat{equation}{\textup{#2(#1)#3}}{ and \textup{#2(#1)#3}}
{, \textup{#2(#1)#3}}{, and \textup{#2(#1)#3}}
\crefrangemultiformat{equation}{\textup{#3(#1)#4--#5(#2)#6}}%
{ and \textup{#3(#1)#4--#5(#2)#6}}{, \textup{#3(#1)#4--#5(#2)#6}}%
{, and \textup{#3(#1)#4--#5(#2)#6}}

\Crefformat{equation}{#2Equation~\textup{(#1)}#3}
\Crefrangeformat{equation}{Equations~\textup{#3(#1)#4--#5(#2)#6}}
\Crefmultiformat{equation}{Equations~\textup{#2(#1)#3}}{ and \textup{#2(#1)#3}}
{, \textup{#2(#1)#3}}{, and \textup{#2(#1)#3}}
\Crefrangemultiformat{equation}{Equations~\textup{#3(#1)#4--#5(#2)#6}}%
{ and \textup{#3(#1)#4--#5(#2)#6}}{, \textup{#3(#1)#4--#5(#2)#6}}%
{, and \textup{#3(#1)#4--#5(#2)#6}}

\crefdefaultlabelformat{#2\textup{#1}#3}
\newcommand{\vertiii}[1]{{\left\vert\kern-0.25ex\left\vert\kern-0.25ex\left\vert #1
    \right\vert\kern-0.25ex\right\vert\kern-0.25ex\right\vert}}

\newcommand{\rhoa}{\rho_a}
\newcommand{\rhob}{\rho_b}

\newcommand{\mcf}{\mathcal{F}}

\newcommand{\aaa}{{\mathcal{A}}}  
\newcommand{\mcc}{\mathcal{C}}

\newcommand{\mcp}{{\mathcal{P}}}
\newcommand{\mcq}{{\mathcal{Q}}}

\newcommand{\la}{\lambda}
\newcommand{\La}{\Lambda}

\newcommand{\RRN}{\mathds{R}^{N}}
\newcommand{\RR}{\mathds{R}}
\newcommand{\NN}{\mathds{N}}

\newcommand{\Rn}{{\mathds{R}^{n}}}

\DeclareMathOperator{\dv}{div}

\newcommand{\apprle}{\lesssim}
\newcommand{\tz}{\tilde{z}}

\def\Yint#1{\mathchoice
	{\YYint\displaystyle\textstyle{#1}}%
	{\YYint\textstyle\scriptstyle{#1}}%
	{\YYint\scriptstyle\scriptscriptstyle{#1}}%
	{\YYint\scriptscriptstyle\scriptscriptstyle{#1}}%
	\!\iint}
\def\YYint#1#2#3{{\setbox0=\hbox{$#1{#2#3}{\iint}$}
		\vcenter{\hbox{$#2#3$}}\kern-.50\wd0}}
\def\longdash{-\mkern-9.5mu-} 
\def\tiltlongdash{\rotatebox[origin=c]{18}{$\longdash$}}
\def\fiint{\Yint\tiltlongdash}

\def\XXint#1#2#3{{\setbox0=\hbox{$#1{#2#3}{\int}$}
		\vcenter{\hbox{$#2#3$}}\kern-.50\wd0}}

\usepackage{color}
\definecolor{dmagenta}{rgb}{.5,.1,.3}
\definecolor{dblue}{rgb}{.0,.0,.5}
\definecolor{mblue}{rgb}{.0,.4,.7}
\definecolor{ddblue}{rgb}{.0,.0,.4}
\definecolor{dred}{rgb}{.9,.0,.0}
\definecolor{dgreen}{rgb}{.0,.5,.0}
\definecolor{Eeom}{rgb}{.0,.0,.5}
\definecolor{dbrown}{rgb}{.7,.0,.0}

\allowdisplaybreaks

\newcommand{\ttl}{}
\begin{document}
\title[Gradient higher integrability for parabolic multi phase]
{\ttl}

\begin{center}{\bf  \Large Gradient higher integrability for degenerate/ singular parabolic multi-phase problems }\medskip
\bigskip

{\small  Abhrojyoti Sen\footnote{sen@math.uni-frankurt.de}
}
 \bigskip

{\small
$^1$Goethe-Universit\"{a}t Frankfurt, Institut f\"{u}r Mathematik, Robert-Mayer-Str. 10,\\
D-60629 Frankfurt, Germany
 } 




\begin{abstract} 
This article establishes an interior gradient higher integrability result for weak solutions to parabolic multi-phase problems. The prototype equation for the parabolic multi-phase problem of $p$-Laplace type is given by
\[
u_t - \operatorname{div} \left(|\nabla u|^{p-2} \nabla u + a(z) |\nabla u|^{q-2} \nabla u + b(z) |\nabla u|^{s-2} \nabla u \right) = 0,
\]
where $\frac{2n}{n+2} < p \leq q \leq s < \infty$, and the coefficients $a(z)$ and $b(z)$ are non-negative Hölder continuous functions on $\Omega_T = \Omega \times (0, T)$, with $\Omega \subset \mathbb{R}^n$. We introduce a novel intrinsic scaling to address the problem in both the degenerate regime ($p \geq 2$) and the singular regime $\left(\frac{2n}{n+2} < p < 2\right),$ providing a unified framework.
Our approach involves proving uniform parabolic Sobolev-Poincaré inequalities, which are key to establishing reverse Hölder type inequalities, along with covering lemmas for the $p$, $(p,q)$, $(p,s)$, and $(p,q,s)$-phases without distinguishing between the regimes of $p$, $q$, and $s$. In the end, we also discuss the gradient higher integrability for general parabolic multi-phase problem involving a finite number of phases.
\end{abstract}
\end{center}

\maketitle
\tableofcontents

\section{Introduction and main result}
\subsection{Overview} In this paper, we  are interested in studying the gradient  higher integrability for weak solutions to 
\begin{align}\label{main_eqn}
	u_t- \dv \aaa \left(z,\nabla u\right)=0\,\,\,\, \text{in}\,\,\,\,\Omega_T,
\end{align}
where $z=(x,t) \in \Omega_T:=\Omega\times (0, T)$ denotes a space-time cylinder, with $\Omega \subset \RR^n$ being a bounded domain for $n\geq 2$  and $\aaa$ satisfies the conditions outlined in \cref{definition_aaa} (see \cref{def n not}). Here, the function $u: \Omega_T \to \RR^N$ with $N\geq 1,$ where $\nabla u$ represents the spatial gradient of $u$ and $u_t$  denotes the time derivative. The local higher integrability of the gradient for weak solutions $u$ to the elliptic $p$-Laplace equation is a well-established result; it shows that if the weak solution $u$ is in $W^{1,p}_{loc}(\Omega),$ then $u$ belongs to $W^{1, p+\varepsilon}_{loc}(\Omega)$ for some $\varepsilon >0$ (see \cite{ME}). This phenomenon was first observed by Gehring \cite{GH} in the context of the Jacobian of a quasi-conformal mapping. The proof for elliptic equations and systems is based on the energy estimate- specifically, an estimate that has the structure of a reverse Poincar\'{e} inequality- and a Sobolev-Poincar\'{e} inequality that yields certain reverse H\"{o}lder type inequalities for $\nabla u.$ The final proof of higher integrability for $\nabla u$ follows from an application of Gehring's type lemma.
As it is well known by now, although the proof for the higher integrability of $\nabla u$ where $u$ is a weak solution to the parabolic $p$-Laplace equation broadly follows a similar strategy as in the elliptic case, there are certain genuine difficulties. One of the main difficulties is when $p\neq 2,$ the parabolic system exhibits a nonhomogeneous scaling, that is, a constant multiple (except $0$ and $1$) of a solution to the parabolic system no longer remains a solution. However, a scale-invariant reverse H\"{o}lder type inequality is crucial for establishing higher integrability. To overcome this difficulty, DiBenedetto-Friedman \cite{MR1230384, FD84, FD85, FD185} introduced scaled parabolic cylinders, known as intrinsic cylinders, whose scaling parameter depends on the local behaviour of the solution. The subsequent step in the proof involves demonstrating the existence of these intrinsic cylinders on a suitable superlevel set using a stopping time argument. The final step in establishing higher integrability combines these results with reverse H\"{o}lder inequalities derived for the intrinsic cylinders. The first result establishing the gradient higher integrability of weak solutions to the parabolic $p$-Laplace equation in the range $\frac{2n}{n+2}< p< \infty$ is due to Kinnunen-Lewis \cite{MR1749438}. Also, see \cite{KL very weak} for the gradient higher integrability of the very weak solutions.

In the following paragraphs, we will review some existing regularity results for elliptic and parabolic double-phase and multi-phase problems. The elliptic double-phase system reads as:

\begin{align}\label{elliptic double phase}
-\dv \left(|\nabla u|^{p-2}\nabla u+a(x)|\nabla u|^{q-2}\nabla u\right)=0\,\,\,\, \text{in}\,\,\,\, \Omega,
\end{align}
where $a(x)\in C^{\alpha}(\Omega)$ for some $\alpha \in (0,1].$ Esposito et. al. \cite{MR2076158} were the first to establish the gradient higher integrability of weak solutions to \eqref{elliptic double phase} under the conditions:

\begin{align*}
1<p<q< p+\frac{\alpha p}{n}
\end{align*}
with an additional assumption that $\nabla u \in L^q(\Omega).$ The related regularity results, including Harnack's inequality, H\"{o}lder estimates and Calder\'{o}n-Zygmund type estimates can be found in \cite{MR3348922, MR3447716, MR3985927, MR3294408}. Moreover, Colombo-Mingione \cite{CM15} established the gradient H\"{o}lder regularity under a sharp condition on $p$ and $q,$ i.e.,
\begin{align*}
    q\leq p+\alpha
\end{align*}
which is independent of $n.$

From a more technical point of view, a typical approach to addressing double phase problems is to first consider a fixed ball $B_R$ along with the condition 
\begin{align}\label{elliptic p-phase condition}
    \inf_{x\in B_R}\frac{a(x)}{R^{\alpha}} \leq M
\end{align}
for some threshold quantity $M.$ If the condition \eqref{elliptic p-phase condition} is satisfied, it indicates that the coefficient $a(x)$ remains sufficiently small in $B_R,$ placing us in the $p$-phase. Conversely, if \eqref{elliptic p-phase condition} does not hold, specifically if,
\begin{align*}
  \inf_{x\in B_R}\frac{a(x)}{R^{\alpha}}> M,   
\end{align*}
then the influence of $q$-growth in \eqref{elliptic double phase} dominates and we are in the $(p,q)$-phase.

The elliptic multi-phase equation reads as follows:
\begin{align*}
-\operatorname*{div}\left(|Du|^{p-2}Du + \sum_{i=1}^m a_i(x)|Du|^{p_i-2}Du\right)
 =-\operatorname*{div}\left(|F|^{p-2}F +\sum_{i=1}^m a_i(x)|F|^{p_i-2}F\right)
\end{align*}
in $\Omega$, where $1<p<p_1\leq \cdots \leq p_m$ and $0 \leq a_i(\cdot)\in C^{0,\alpha_i}(\overline{\Omega})$ with $\alpha_i \in (0,1]$ for all $i\in\{1,\cdots, m\}$.
Similarly to the assumptions made in the double-phase case, we impose the following condition:
$$
\frac{p_i}{p}\leq 1+\frac{\alpha_i}{n} \qquad \text{for all }i\in \{1,\cdots, m\}.
$$ 
De Filippis-Oh \cite{FO19} established that the gradient of a local minimizer of the multi-phase variational problem is H\"{o}lder continuous and satisfies intrinsic Morrey decay. Additionally, De Filippis \cite{DF21} proved Calderón-Zygmund type estimates, while Baasandorj-Byun-Oh \cite{BBO21} investigated gradient estimates in the Orlicz-type multi-phase equations. Moreover, Fang-R\u{a}dulescu-Zhang-Zhang \cite{FRZZ22} provided gradient estimates within the framework of Campanato spaces.

Although a substantial body of literature exists on the regularity theory for elliptic double-phase problems, the investigation of parabolic double-phase problems remains relatively sparse. Recent work by Chlebicka et al. \cite{MR3985549} and Singer \cite{MR3532237} has established the existence of weak solutions for parabolic double-phase problems. The first regularity result concerning the parabolic double phase problem is the local gradient higher integrability of weak solutions to the following equation:
\begin{align*}
    u_t-\operatorname*{div}\left(|Du|^{p-2}Du +  a(z)|Du|^{q-2}Du\right) =-\operatorname*{div}\left(|F|^{p-2}F +a(z)|F|^{q-2}F\right)
\end{align*} 
due to Kim-Kinnunen-Moring \cite{KKM}. Their result is established for the case $2\leq p\leq q<\infty.$ On the other hand, Kim-S\"{a}rki\"{o} \cite{KS2024} addressed the case where $\frac{2n}{n+2}<p<2.$ It is worth mentioning that H\"{a}st\"{o}-Ok \cite{MR4302665} studied the gradient higher integrability for parabolic $(p,q)$-growth problems with constant coefficients, specifically when $a(z)=a_0>0.$  In their work, Kim-Kinnunen-Moring \cite{KKM} introduced appropriate alternatives and two types of intrinsic cylinders. For the $p$-phase, they defined $p$-intrinsic cylinders as:
\begin{align*}
    Q^{\la}_{\varrho}(z_0)=B_{\varrho}(x_0)\times (t_0-\la^{2-p}\varrho^2, t_0+\la^{2-p}\varrho^2)
\end{align*}
and for $(p, q)$-phase, they considered the $(p,q)$-intrinsic cylinders given by: 
\begin{align*}
    G^{\la}_{\varrho}(z_0)=B_{\varrho}(x_0)\times \left(t_0-\frac{\la^2}{\la^p+a(z_0)\la^q}\varrho^2, t_0+\frac{\la^2}{\la^p+a(z_0)\la^q}\varrho^2\right)
\end{align*}
are considered for the $(p,q)$-phase. On the other hand, Kim-S\"{a}rki\"{o} \cite{KS2024} introduced different intrinsic cylinders for $p$-phase and $(p,q)$-phase:
\begin{align*}
    Q^{\la}_{\varrho}(z_0)=B_{\la^{\frac{p-2}{2}}\varrho}(x_0)\times (t_0-\varrho^2, t_0+\varrho^2)
\end{align*}
and 
\begin{align*}
    G^{\la}_{\varrho}(z_0)=B_{\la^{\frac{p-2}{2}}\varrho}(x_0)\times \left(t_0-\frac{\la^p}{\la^p+a(z_0)\la^q}\varrho^2, t_0+\frac{\la^p}{\la^p+a(z_0)\la^q}\varrho^2\right),
\end{align*}
respectively. A notable characteristic of these intrinsic cylinders is that, for the case $p\geq 2$, both types of intrinsic cylinder shrink in the time scale as $\la \to \infty.$ In contrast, for the case $p<2,$ the $p$-intrinsic cylinders shrink in the spatial scale as $\lambda \to \infty.$ Furthermore, in this latter scenario, since $p\leq q,$ the $(p,q)$-intrinsic cylinders shrink in both time and space scales as $\lambda$ approaches infinity.

The primary objective of this paper is to establish a gradient higher integrability result for bounded weak solutions to \eqref{main_eqn}, which is modeled after the following equation:

\begin{align*}
    u_t-\operatorname{div} \left(|\nabla u|^{p-2} \nabla u + a(z) |\nabla u|^{q-2} \nabla u + b(z) |\nabla u|^{s-2} \nabla u \right) = 0,\,\,\,\text{in}\,\,\, \Omega_T.
\end{align*}

In a recent paper by Kim-Oh \cite{KO2024}, the authors established gradient higher integrability for the case $2\leq p\leq q< \infty$ (the degenerate case). In this work, we focus on the singular regime, specifically for $\frac{2n}{n+2}<p\leq 2.$ We further assume that the non-negative coefficients $a(\cdot),b(\cdot)$ are $\alpha$- and $\beta$-H\"{o}lder continuous, respectively, for some $\alpha, \beta \in (0, 1]$ under conditions $q\leq p+\frac{2\alpha}{n+2}$ and $s\leq p+\frac{2\beta}{n+2}.$ In \cite{KO2024}, four types of intrinsic cylinders and alternatives were introduced to achieve the goal of gradient higher integrability. In this study, we propose the use of $\mu$-modified intrinsic cylinders (see \ref{not5}, \cref{Notation}) to facilitate our analysis. From the definition of $\mu$ provided in  \ref{not4}, \cref{Notation}, specifically,
$\frac{n}{n+2}<\mu\leq \min\left\{1, \frac{p}{2}\right\},$
it is evident that the cylinders shrink in both time and spatial scales simultaneously as $\la \to \infty.$ The introduction of these cylinders offers the significant advantage of enabling a unified treatment of both the singular case $\left(\frac{2n}{n+2}<p<2\right)$ and the degenerate ($p\geq 2$) case. The motivation for utilizing such intrinsic cylinders is inspired by the work of Adimurthi-Byun-Oh \cite{adi}, which presents a unified approach to achieving boundary gradient higher integrability for very weak solutions to parabolic systems with $p(z)$-growth \cite{MR2779582}.  Notably, similar results were previously established by B\"{o}gelein-Li \cite{BL14} for the case $p(z)\geq 2$ and by Li \cite{Li17} for the case $\frac{2n}{n+2}<p(z)\leq 2,$ but handled separately.
\vspace{.5cm}
\subsection{Definitions and Notations}\label{def n not}
In this section, we introduce some basic definitions and notations. We start with the definition of the parabolic metric.
\begin{definition}[Parabolic metric]
Given any two points $z_1=(x_1, t_1)$ and $z_2=(x_2, t_2) \in \RR^{n+1},$ we define the parabolic metric $d$ as 
\begin{align*}
d(z_1, z_2):=\max\left\{|x_1-x_2|, \sqrt{|t_1-t_2|}\right\}.
\end{align*}
\end{definition}
\vspace{.3cm}
A suitable adaptation of the above definition in the intrinsically scaled cylinders centered at $z_0\in \RR^{n+1}$ is as follows.
\begin{definition}[Scaled parabolic metric]\label{para metric}
Fix $z_0 \in \RR^{n+1}.$ Given any two points $z_1=(x_1, t_1)$ and $z_2=(x_2, t_2) \in \RR^{n+1}$ with $K\geq 1, \lambda\geq 1, \mu>0,$ we define the parabolic metric $d^{\la,\mu}_{z_0}$ as

\begin{align*}
\displaystyle
    d^{\la, \mu}_{z_0}(z_1,z_2 )=\left\{\begin{array}{l}
    \max\left\{\la^{1-\mu}|x_1-x_2|, \sqrt{\lambda^{p-2\mu}|t_1-t_2|}\right\},\\ 
    \qquad\qquad\qquad \text{if}\,\,\, K\lambda^p\geq a(z_0)\lambda^q \,\,\, \text{and}\,\,\,K\lambda^p \geq b(z_0)\lambda^s,\\
    \max\left\{\la^{1-\mu}|x_1-x_2|, \sqrt{(\lambda^p+a(z_0)\lambda^q)\lambda^{-2\mu}|t_1-t_2|}\right\},\\
    \qquad\qquad\qquad \text{if}\,\,\, K\lambda^p< a(z_0)\lambda^q \,\,\, \text{and}\,\,\,K\lambda^p\geq b(z_0)\lambda^s,\\
    \max\left\{\la^{1-\mu}|x_1-x_2|, \sqrt{(\lambda^p+b(z_0)\lambda^s)\lambda^{-2\mu}|t_1-t_2|}\right\},\\
    \qquad\qquad\qquad \text{if}\,\,\, K\lambda^p\geq a(z_0)\lambda^q \,\,\, \text{and}\,\,\,K\lambda^p< b(z_0)\lambda^s,\\
    \max\left\{\la^{1-\mu}|x_1-x_2|, \sqrt{\Lambda\lambda^{-2\mu}|t_1-t_2|}\right\},\\
    \qquad\qquad\qquad \text{if}\,\,\, K\lambda^p< a(z_0)\lambda^q \,\,\, \text{and}\,\,\,K\lambda^p< b(z_0)\lambda^s,
    \end{array}\right.
\end{align*}
where $\La:=\la^p+a(z_0)\la^q+b(z_0)\la^s.$   
\end{definition}
We fix the following radii.
\begin{definition}\label{choice of radii}
For any $\varrho>0,$ we fix two radii $\rho_a$ and $\rhob$ such that 
\begin{align*}
   0<\varrho\leq \rhoa< \rhob \leq 4\varrho\,\,\,\,\, \text{and}\,\,\,\, 2\varrho \leq \rhob \leq 4\varrho \,\,\,\text{with}\,\,\,0<h_0< \frac{\rhob^2-\rhoa^2}{4}
\end{align*}
be very small number.
\end{definition}
We set the following structural assumptions on $\aaa$ in \eqref{main_eqn}.
\begin{assumption}
	\label{definition_aaa}
	We assume $\aaa(z, \xi):\Omega_T\times \RR^{Nn} \to \RR^{Nn}$ with $N\geq 1$ is a Carath\'{e}odory vector field satisfying the following structural assumptions:
 \vspace{.3cm}
	\begin{itemize}
		\item[\normalfont{\textbf{A1.}}] \label{A1} 
  $\langle \aaa (z, \xi), \xi \rangle \geq \mcc_0 \left(|\xi|^p+a(z)|\xi|^q+b(z)|\xi|^s\right)$ and $|\aaa (z, \xi)| \leq \mcc_1\left(|\xi|^{p-1}+a(z)|\xi|^{q-1}+b(z)|\xi|^{s-1}\right)$ holds  for every $z\in \Omega_T$ and $\xi \in \RR^{Nn}$. Here, $\mcc_0, \mcc_1$ are positive constants with $\mcc_0 \leq \mcc_1 < \infty.$
		\vspace{.3cm}
		\item [\normalfont{\textbf{A2.}}]\label{A2}
  The coefficients $a:\Omega_T \to \RR^+$ is in $C^{\alpha, \frac{\alpha}{2}}(\Omega_T)$ for some $\alpha \in (0,1]$ and $b:\Omega_T \to \RR^+$ is in $C^{\beta, \frac{\beta}{2}}(\Omega_T)$ for some $\beta \in (0,1]$  and we impose the restriction 
		\begin{equation}\label{def_pq}
          \begin{aligned}
			\frac{2n}{n+2}<p \leq q \leq  p + \min\left\{\alpha\left(\frac{p}{2}- \frac{n}{n+2}\right),\frac{2\alpha}{n+2}\right\}<\infty,\\
            \frac{2n}{n+2}<p \leq s \leq  p + \min\left\{\beta\left(\frac{p}{2}- \frac{n}{n+2}\right),\frac{2\beta}{n+2}\right\}<\infty.
            \end{aligned}
		\end{equation}
		In particular, this means that $a, b \in L^{\infty}(\Omega_T)$ and \begin{align*}
		&|a(x_1,t_1)-a(x_2,t_2)|\leq [a]_{\alpha}\left(|x_1-x_2|^{\alpha}+|t_1-t_2|^{\frac{\alpha}{2}}\right), \\
  &|b(x_1,t_1)-b(x_2,t_2)|\leq [b]_{\beta}\left(|x_2-x_2|^{\beta}+|t_1-t_2|^{\frac{\beta}{2}}\right),
  \end{align*} for every $(x_1,t_1), (x_2,t_2) \in \Omega_T.$
	\end{itemize}
\end{assumption}
Let us first recall the definition of weak solution to \cref{main_eqn}:
\begin{definition}\label{weak solution}
	A function $u : \Omega\times (0, T) \to \RRN$ with 
	\begin{align*}
		u \in C_{loc}\left(0,T; L^2_{loc}\left(\Omega, \RRN\right)\right)\cap L^s_{loc}\left(0, T; W^{1, s}_{loc}\left(\Omega, \RRN\right)\right)
	\end{align*}
	is said to be a weak solution to \eqref{main_eqn} if for every compact subset $K\subset \Omega$ and for every subinterval $[t_1, t_2]\subset (0, T]$
	\begin{align*}
		\int_{t_1}^{t_2}\int_K -u\cdot \varphi_t \,dx\,dt+\int_{K}u\varphi dx\Big|_{t_1}^{t_2}+\int_{t_1}^{t_2} \int_K \aaa(z, \nabla u)\cdot \nabla \varphi \,dx \,dt=0
	\end{align*}
	holds for all nonnegative test function $\varphi \in C^{\infty}_0 (K \times(t_1, t_2), \RRN).$ 
\end{definition}
\vspace{.3cm}
\subsubsection{Notations}\label{Notation}
Here we collect the standard notation that will be used throughout the paper:
\begin{enumerate}[label=(\roman*),series=theoremconditions]
\item \label{not1} We shall denote $n$ to be the space dimension and by $z=(x,t)$ to be  a point in $ \RR^n\times (0,T)$. Moreover, we will always assume $n \geq 2$.   
\item\label{not2} A ball with center $x_0 \in \Rn$ and radius $r$ is denoted as
 \begin{align*}
     B_r(x_0)=\left\{x \in \Rn | \,\,\, |x-x_0|<r\right\}.
 \end{align*}
\item\label{general cylinder} In general, a parabolic cylinder centered at $z_0=(x_0, t_0)$ is denoted as 
\begin{align*}
    Q_{r, \varrho}(z_0)=B_r(x_0)\times \left(t_0-\varrho, t_0+\varrho\right):=B_r(x_0)\times l_{\varrho}(t_0).
\end{align*}
\item \label{not4} We shall fix a universal constant $\frac{n}{n+2} < \mu \leq \min\left\{1,\frac{p}{2}\right\}$ which is possible since we always assume $p > \frac{2n}{n+2}$ holds. 
\item\label{not5} We use four types of cylinders for $p,$ $(p, q),$ $(p,s)$ and $(p,q, s)$-phases. For $\la\geq 1,$ the notation $\mcp^{\lambda}_{\varrho}(z_0)$ is used to denote $p$-intrinsic cylinders at $z_0=(x_0, t_0),$
\begin{align*}
\mcp^{\lambda}_{\varrho}(z_0):=B_{\lambda^{-1+\mu}\varrho}(x_0)\times \left(t_0-\frac{\lambda^{2\mu}\varrho^2}{\lambda^p}, t_0+\frac{\lambda^{2\mu}\varrho^2}{\lambda^p}\right)
=:B^\la_{\varrho}(x_0)\times I^{\la, p}_{\varrho}(t_0).
\end{align*}
$\mcq^{\lambda}_{\varrho}(z_0)$ is used to denote $(p,q)$-intrinsic cylinders at $z_0=(x_0, t_0),$
\begin{align*}
\mcq^{\lambda}_{\varrho}(z_0):=B_{\lambda^{-1+\mu}\varrho}(x_0)\times \left(t_0-\frac{\lambda^{2\mu}\varrho^2}{\lambda^p+a(z_0)\lambda^q}, t_0+\frac{\lambda^{2\mu}\varrho^2}{\lambda^p+a(z_0)\lambda^q}\right)
=:B^\la_{\varrho}(x_0)\times I^{\la, (p,q)}_{\varrho}(t_0).    
\end{align*}
$\mathcal{S}^{\la}_{\varrho}(z_0)$ is used to denote $(p,s)$-intrinsic cylinders at $z_0=(x_0, t_0),$
\begin{align*}
\mathcal{S}^{\lambda}_{\varrho}(z_0):=B_{\lambda^{-1+\mu}\varrho}(x_0)\times \left(t_0-\frac{\lambda^{2\mu}\varrho^2}{\lambda^p+b(z_0)\lambda^s}, t_0+\frac{\lambda^{2\mu}\varrho^2}{\lambda^p+b(z_0)\lambda^s}\right)
=:B^\la_{\varrho}(x_0)\times I^{\la, (p,s)}_{\varrho}(t_0),    
\end{align*}
and $\mathcal{G}^{\la}_{\varrho}(z_0)$ is used to denote $(p,q,s)$-intrinsic cylinders at $z_0=(x_0, t_0),$
\begin{align*}
\mathcal{G}^{\lambda}_{\varrho}(z_0)&:=B_{\lambda^{-1+\mu}\varrho}(x_0)\times \left(t_0-\frac{\lambda^{2\mu}\varrho^2}{\lambda^p+a(z_0)\lambda^q+b(z_0)\la^s}, t_0+\frac{\lambda^{2\mu}\varrho^2}{\lambda^p+a(z_0)\lambda^q+b(z_0)\la^s}\right)\\
&=:B^\la_{\varrho}(x_0)\times I^{\la, (p,q, s)}_{\varrho}(t_0).    
\end{align*}
\item \label{not5.5} In the case $\la =1$, we denote 
\[
\mcp_\varrho(z_0) := B_{\varrho}(x_0) \times (t_0 - \varrho^2,t_0+\varrho^2) =: B_\varrho(x_0) \times I_\varrho(t_0).
\]
\item\label{not6} We shall use the notation $H(z,|\xi|)  := | \xi|^{p} + a(z) | \xi|^q+b(z)|\xi|^s$ for any $\xi \in \Rn.$
	\item\label{not7} Integration with respect to either space or time only will be denoted by a single integral $\int$ whereas integration on $\Omega\times (0, T)$ will be denoted by a double integral $\iint$. 
	\item\label{not9} The notation $a \lesssim b$ is shorthand for $a\leq c b$ where $c$ is a universal constant which depends on ``\texttt{data}$"$ where 
 \begin{align*}
     \text{\texttt{data}}:=\left(n, N, p, q, s, \mcc_0, \mcc_1, \alpha, \beta, [a]_{\alpha}, [b]_{\beta}, ||a||_{L^{\infty}(\Omega_T)}, ||b||_{L^{\infty}(\Omega_T)}, ||H(z, |\nabla u|)||_{L^1(\Omega_T)}\right).
 \end{align*}
 \item \label{integral avarage} We shall denote
 \begin{align*}
     (f)_{A}:= \frac{1}{|A|}\iint_A f dz= \fiint_{A} f dz
 \end{align*}
 as the integral average of $f$ over a measurable set $A \subset \Omega_T$ with $0<|A|<\infty.$
\end{enumerate}
\subsection{Main result} Now we state our main result.
\begin{theorem}
	\label{main_thm}
	Let $u$ be a weak solution of \eqref{main_eqn} with \cref{definition_aaa} in force and  $\mu$ be the constant defined in \cref{Notation} \ref{not4}. Then there exist positive constants $\varepsilon_0 = \varepsilon_0(\textnormal{\texttt{data}})\in (0,1)$ and $c=c(\textnormal{\texttt{data}})\geq 1$ such that the following holds:
	\begin{align*}
		\fiint_{\mcp_{r}(z_0)} H(z, |\nabla u|)^{1+\varepsilon}\, dz \leq c\left(\fiint_{\mcp_{2r}(z_0)}H(z, |\nabla u|)\,dz+1\right)^{1+\frac{\varepsilon s}{(n+2)\mu-n}},  
	\end{align*}
	for every $\mcp_{2r}(z_0)\subset \Omega \times (0, T)$ and any $\varepsilon \in (0, \varepsilon_0)$ where $\mcp_{(\cdot)}(z_0)$ is defined in \cref{Notation} \ref{not5.5}.
\end{theorem}
\begin{remark}\label{remark1.1}
Note that in \cref{weak solution}, we assume that $|\nabla u|\in L^s(\Omega_T).$ However, a more natural assumption would be to consider $|\nabla u|\in L^p(\Omega_T)$ alongside the condition

\begin{align*}
    \iint_{\Omega_T}H(z, |\nabla u|)\,dz=\iint_{\Omega_T}\left(|\nabla u|^p+a(z)|\nabla u|^q+b(z)|\nabla u|^s\right) \,dz < \infty.
\end{align*}
In an upcoming work \cite{KOS2024}, we will employ a parabolic Lipschitz truncation method for the system \eqref{main_eqn} under the assumption $2\leq p\leq q<\infty.$ The Lipschitz truncation method for the singular case, including in the double-phase problem, remains an open question and could be explored using this new intrinsic scaling in future work.
\end{remark}
\begin{remark}
It is important to note that \cref{main_eqn} is homogeneous. The techniques developed in this work can also be applied to equations with a non-homogeneous term in divergence form. However, to simplify the exposition, we restrict our focus to the homogeneous equation.
\end{remark}
\subsubsection{Some comments on the proof} Before proceeding, we provide an outline of the proof of the main theorem. As mentioned earlier, the strategy to establish the higher integrability of gradients for parabolic equations typically involves two main steps: first, we prove reverse Hölder-type inequalities on intrinsic cylinders, and in the second step, we show the existence of such intrinsic cylinders in certain superlevel sets. The final proof can be completed by using a Vitali covering argument alongside Fubini's theorem. In the following, we describe each of these steps in more detail within our context.
\begin{itemize}
    \item[(i)] To analyze the multi-phase phenomenon described by the equation, we consider four types of intrinsic cylinders and appropriate alternatives, namely $p, (p,q), (p,s),$ $(p,q,s)$-phases, as introduced in \cite{KO2024}. We establish reverse Hölder inequalities for each type of intrinsic cylinder. A fundamental component in deriving these reverse Hölder inequalities is the parabolic Sobolev-Poincaré-type inequality, which we first need to establish for each phase. Our parabolic Sobolev-Poincar\'{e} inequalities hold for the range $\frac{2n}{n+2}<p\leq q\leq s<\infty.$  The primary tools for proving these inequalities include basic applications of Hölder's inequality and Young's inequality. Interestingly, our findings suggest that the parabolic Sobolev-Poincaré inequalities for the singular case can be viewed as a perturbation by $``\lambda"$ of the inequalities for the degenerate case (see \cref{p_intrinsic poincare 2} and \cref{pq_intrinsic poincare}).  This perspective gives a streamlined approach to several estimates, including those in the double-phase case.   Furthermore, we introduce a parameter $\mu$ to modify our intrinsic cylinders, allowing for a unified handling of the problem. In particular, when $\mu=1,$ we recover the intrinsic cylinders defined in \cite{KO2024}, and for $\mu=\frac{p}{2},$ we obtain similar intrinsic cylinders as those defined in \cite{KS2024} for the double-phase problem.
\item[(ii)] The next objective is to establish the existence of such intrinsic cylinders in a superlavel set and devise a precise covering argument that is applicable for all $\frac{2n}{n+2}<p\leq q\leq s<\infty.$ We derive a Vitali constant $\kappa$ that depends on the parameter $\mu.$ Unlike \cite{KKM, KS2024}, our  Vitali covering argument, requires addressing sixteen distinct cases. This is required to encompass all possible configurations of intrinsic cylinders, as outlined in \cite{KO2024}. To ensure that the exposition is self-contained and precise, we provide detailed and rigorous proofs throughout the discussion.
\end{itemize}
\subsection{Plan of the paper} The article is structured as follows: In \cref{sec2}, we record some known results.  \cref{sec3} is dedicated to proving the reverse H\"{o}lder inequalities for the $p, (p,q), (p,s), (p,q,s)$-phases. In \cref{sec4}, we establish our gradient higher integrability result, which is stated in \cref{main_thm}. In \cref{general multiphase sec}, we outline a blueprint for achieving the analogue of \cref{main_thm} for \eqref{general multiphase}.
Lastly, in \cref{appendix}, we provide a sketch of the proofs for \cref{ps_intrinsic poincare} and \cref{pqs_intrinsic poincare}.

\section{Auxiliary results} \label{sec2}
In this section, we collect some auxiliary lemmas which will be used throughout the paper.
The first result is the Gagliardo-Nirenberg inequality (see \cite[Lemma 4.1]{KKM} or \cite[Chapter I]{MR1230384}).
\begin{lemma}
	\label{g_n}
	Let $\varrho >0$ and $B_\varrho(x_0)\subset \mathbb{R}^n$ be a ball. For constants $\sigma,\xi ,r\in[1,\infty)$ and $\vartheta\in(0,1)$ satisfying $-\frac{n}{\sigma}\leq \vartheta(1-\frac{n}{\xi})-(1-\vartheta)\frac{n}{r}$  and for any $h\in W^{1,\xi}(B_\varrho(x_0))$,  there holds
	$$
	\fint_{B_\varrho(x_0)}\frac{\left| h\right|^\sigma}{\varrho^\sigma}\ dx\apprle_{(n,\xi)}\left(\fint_{B_\varrho(x_0)}\frac{\left| h\right|^{\xi}}{\varrho^\xi}+|\nabla h|^\xi\ dx\right)^\frac{\vartheta\sigma}{\xi}\left(\fint_{B_\varrho(x_0)}\frac{\left| h\right|^r}{\varrho^r}\ dx\right)^\frac{(1-\vartheta)\sigma}{r}.
	$$
\end{lemma}
Next, we record the standard iteration lemma (see \cite[Lemma 4.2]{KKM}):
\begin{lemma}
	\label{iter_lemma}
	Let $0< r< R<\infty$ be given and $h : [r, R] \to \RR$ be a non-negative and bounded function. Furthermore, let $\theta \in (0,1)$ and $A,B,\gamma_1,\gamma_2 \geq 0$ be fixed constants and 
	suppose that
	$$
	h(\varrho_1) \leq \theta h(\varrho_2) + \frac{A}{(\varrho_2-\varrho_1)^{\gamma_1}} + \frac{B}{(\varrho_2-\varrho_1)^{\gamma_2}},
	$$
	holds for all $r \leq \varrho_1 < \varrho_2 \leq R$, then the following conclusion holds:
	$$
	h(r) \apprle_{(\theta,\gamma_1,\gamma_2)} \frac{A}{(R-r)^{\gamma_1}} + \frac{B}{(R-r)^{\gamma_2}}.
	$$
\end{lemma} 
Below, we recall the energy estimate derived in \cite[Lemma 3.1]{KO2024}. Note that, Kim-Oh derived the energy estimate for $2\leq p\leq q<\infty.$ By using a Lipschitz truncation method as in \cite{KKS} the same energy estimate can be proved for the singular case, $\frac{2n}{n+2}<p\leq 2$ (see \cref{remark1.1}).

\vspace{.3cm}
To state the energy estimate, let us define appropriate cutoff functions. Assume $R_2, S_2>0$ with $R_1 \in [\frac{R_2}{2}, R_2)$ and $S_1 \in [\frac{S_2}{2^2}, S_2).$ Furthermore, let $0<h_0< \frac{S_2-S_1}{4}$ and $\eta \in C^{\infty}_c (B_{R_2}(x_0))$ and $\zeta \in C^{\infty}_c (l_{S_2-h_0}(t_0)).$ Moreover, $\eta$ and $\zeta$ satisfy the following:
\begin{align*}
&(i)\,\,\,\,\ 0\leq \eta \leq 1, \,\,\,\, \eta\equiv 1 \,\,\, \text{in}\,\,\, B_{R_1}(x_0),\,\,\,\, \text{and}\,\,\,\, ||\nabla \eta||_{\infty}\leq \frac{c}{R_2-R_1},\\
&(ii)\,\,\,\,\ 0\leq \zeta \leq 1, \,\,\,\, \zeta\equiv 1 \,\,\, \text{in}\,\,\, l_{S_1}(t_0),\,\,\,\, \text{and}\,\,\,\, ||\partial_t \zeta||_{\infty}\leq \frac{c}{S_2-S_1}.
\end{align*}
Then the energy estimate reads as follows. 
\begin{lemma}\label{general caccipoli} 
Let $u$ be a weak solution to \eqref{main_eqn}. Then there exists a constant $c=c(n, p, q,s, \mcc_0, \mcc_1)$ such that
\begin{align*}
    &\sup_{t \in l_{S_2}(t_0)}\fint_{B_{R_2}(x_0)}\frac{\left|u-(u)_{Q_{R_2, S_2}}\right|^2}{S_2} \eta^s(x)\zeta^2(t) dx+ \fiint_{Q_{R_2, S_2}(z_0)}H(z, |\nabla u|)\eta^s(x)
	\zeta^2(t)\,dx\,dt\\
 &\apprle
 \fiint_{Q_{R_2, S_2}(z_0)} H\left(z,\left|\frac{u-(u)_{Q_{R_2, S_2}}}{R_2-R_1}\right|\right) \zeta^2(t)\,dx\,dt
	 +  \fiint_{Q_{R_2, S_2}(z_0)}\frac{\left|u-(u)_{Q_{R_2, S_2}}\right|^2}{S_2-S_1}\eta^s(x)\zeta(t)\,dx\,dt
\end{align*}
holds for every $Q_{R_2, S_2}(z_0)\subset \Omega_T.$
\end{lemma}
The next result is the parabolic Poincar\'{e} inequality which is derived in \cite[Lemma 3.3]{KO2024}.
\begin{lemma}\label{parabolic poincare for double phase}
    Let $u$ be a weak solution to \eqref{main_eqn}. Then there exists a constant $c=c(\textnormal{\texttt{data}})$ such that

    \begin{align*}
        \fiint_{Q_{r, \varrho}(z_0)}\left|\frac{u-(u)_{Q_{r, \varrho}(z_0)}}{r}\right|^{\theta m}\, dz \leq c \fiint_{Q_{r, \varrho}(z_0)}|\nabla u|^{\theta m} dz+ c\left(\frac{\varrho}{r^2}\fiint_{Q_{r, \varrho}(z_0)}\tilde{H}(z, |\nabla u|)\, dz\right)^{\theta m}
    \end{align*}
    for any $Q_{r, \varrho}(z_0)\subset \Omega_T$ defined in \cref{Notation} \ref{general cylinder} with $m\in (1, s]$ and $\theta \in (\frac{1}{m}, 1],$ where $$\tilde{H}(z, |\nabla u|):=|\nabla u|^{p-1}+a(z)|\nabla u|^{q-1}+b(z)|\nabla u|^{s-1}.$$
\end{lemma}
\section{Reverse H\"{o}lder inequalities}\label{sec3} 
In this section, we will establish the reverse H\"{o}lder inequalities for the $p, (p,q), (p,s), (p,q,s)$-phases.
From this point onward, we will set
\vspace{.3cm}
\begin{align}\label{new_k}
    K:=2+\frac{(2 [a]_{\alpha})^{\frac{n+2}{\alpha}}}{|B_1|}\iint_{\Omega_T}H(z, |\nabla u|)\,dz+\frac{(2 [b]_{\beta})^{\frac{n+2}{\beta}}}{|B_1|}\iint_{\Omega_T}H(z, |\nabla u|)\,dz.
\end{align}
\subsection{Reverse H\"older inequality for $p$-phase}
We start with defining the $p$-intrinsic cylinders. 
\begin{assumption}
Recalling \cref{choice of radii}, in the $p$-phase,  we assume that the following is satisfied:
\begin{enumerate}[label=(\roman*),series=theoremconditions]
    \item [\textbf{p1.}] \label{assmp1a} $K\lambda^p \geq a(z_0)\lambda^q,\,\,\, K\la^p\geq b(z_0)\la^s.$ \\
    \item [\textbf{p2.}] \label{assmp1b} $\fiint_{\mcp^{\lambda}_{\rho}(z_0)}(|\nabla u|^p+a(z)|\nabla u|^q+b(z)|\nabla u|^s) dz < \lambda^p\,\,\, \text{holds for all}\,\,\, \rho \in (\rho_a, \rho_b].$ \\
    \item [\textbf{p3.}] \label{assmp1c} $\fiint_{\mcp^{\lambda}_{\rho_a}(z_0)}(|\nabla u|^p+a(z)|\nabla u|^q+b(z)|\nabla u|^s) dz =\lambda^p.$ 
\end{enumerate}
\end{assumption}
We shall start by proving uniform parabolic Sobolev Poincar\'{e} inequalities in $p$-intrinsic cylinders for any $\frac{2n}{n+2}<p\leq q\leq s<\infty$. The following lemma is the first step towards this.
\begin{lemma}\label{p_intrinsic poincare 1}
Let $u$ be a weak solution to \eqref{main_eqn} and let the \cref{assmp1a} be in force. Then for any $\theta\in \left(\max\left\{\frac{s-1}{p}, \frac{1}{p}\right\}, 1\right],$ there exists a constant $c=c(\textnormal{\texttt{data}})$ such that the following estimates,
\begin{enumerate}[label=(\roman*),series=theoremconditions]
\item \label{1_lem7.3} 
\begin{align*}
\fiint_{\mcp^\la_{\rhob}(z_0)}\left|\frac{u-(u)_{\mcp^{\la}_{\rhob}(z_0)}}{\la^{-1+\mu}\rhob}\right|^{\theta p} dz &\leq c \fiint_{\mcp^{\la}_{\rhob}(z_0)}\left|\nabla u\right|^{\theta p}dz + c \la^{(2-p)\theta p} \left(\fiint_{\mcp^{\la}_{\rhob}(z_0)}|\nabla u|^{\theta p} dz\right)^{p-1}\\
&+c \la^{(2-q)\theta p} \left(\fiint_{\mcp^{\la}_{\rhob}(z_0)}|\nabla u|^{\theta p} dz\right)^{q-1}+c \la^{(2-s)\theta p} \left(\fiint_{\mcp^{\la}_{\rhob}(z_0)}|\nabla u|^{\theta p} dz\right)^{s-1},
\end{align*}
\item \label{2_lem7.3}
\begin{align*}
    \fiint_{\mcp^{\la}_{\rhob}(z_0)}a^{\theta}(z)\left|\frac{u-(u)_{\mcp^{\la}_{\rhob}(z_0)}}{\la^{-1+\mu}\rhob}\right|^{\theta q} dz &\leq c\la^{(p-q)\theta} \Bigg[ \fiint_{\mcp^{\la}_{\rhob}(z_0)}\left|\nabla u\right|^{\theta q}dz + \la^{(2-p)\theta q} \left(\fiint_{\mcp^{\la}_{\rhob}(z_0)}|\nabla u|^{\theta q} dz\right)^{p-1}\\
&+ \la^{(2-q)\theta q} \left(\fiint_{\mcp^{\la}_{\rhob}(z_0)}|\nabla u|^{\theta q} dz\right)^{q-1}\\
&+\la^{(2-s)\theta q} \left(\fiint_{\mcp^{\la}_{\rhob}(z_0)}|\nabla u|^{\theta q} dz\right)^{s-1}\Bigg],
\end{align*}
\item \label{3_lem7.3}
\begin{align*}
    \fiint_{\mcp^{\la}_{\rhob}(z_0)}a^{\theta}(z)\left|\frac{u-(u)_{\mcp^{\la}_{\rhob}(z_0)}}{\la^{-1+\mu}\rhob}\right|^{\theta q} dz &\leq c\la^{(p-q)\theta} \Bigg[ \fiint_{\mcp^{\la}_{\rhob}(z_0)}\left|\nabla u\right|^{\theta q}dz +  \la^{(2-p)\theta q} \left(\fiint_{\mcp^{\la}_{\rhob}(z_0)}|\nabla u|^{\theta p} dz\right)^{\frac{q(p-1)}{p}}\\
&+ \la^{(2-p)\theta q }\la^{(p-q)\theta}\left(\fiint_{\mcp^{\la}_{\rhob}(z_0)}\inf_{w\in \mcp^{\la}_{\rhob}(z_0)}a(w)^{\theta}|\nabla u|^{\theta q} dz\right)^{q-1}\\
&+ \la^{(2-p)\theta q}\la^{(p-s)\frac{\theta q}{s}}\left(\fiint_{\mcp^{\la}_{\rhob}(z_0)}\inf_{w\in \mcp^{\la}_{\rhob}(z_0)}b(w)^{\theta}|\nabla u|^{\theta s}\right)^{\frac{q(s-1)}{s}}\Bigg],
\end{align*}
\item \label{4_lem7.3}
\begin{align*}
    \fiint_{\mcp^{\la}_{\rhob}(z_0)}b^{\theta}(z)\left|\frac{u-(u)_{\mcp^{\la}_{\rhob}(z_0)}}{\la^{-1+\mu}\rhob}\right|^{\theta s} dz &\leq c\la^{(p-s)\theta} \Bigg[ \fiint_{\mcp^{\la}_{\rhob}(z_0)}\left|\nabla u\right|^{\theta s}dz + \la^{(2-p)\theta s} \left(\fiint_{\mcp^{\la}_{\rhob}(z_0)}|\nabla u|^{\theta s} dz\right)^{p-1}\\
&+ \la^{(2-q)\theta s} \left(\fiint_{\mcp^{\la}_{\rhob}(z_0)}|\nabla u|^{\theta s} dz\right)^{q-1}\\
&+ \la^{(2-s)\theta s}\left(\fiint_{\mcp^{\la}_{\rhob}(z_0)}|\nabla u|^{\theta s} dz\right)^{s-1}\Bigg],
\end{align*}
\item \label{5_lem7.3}
\begin{align*}
    \fiint_{\mcp^{\la}_{\rhob}(z_0)}b^{\theta}(z)\left|\frac{u-(u)_{\mcp^{\la}_{\rhob}(z_0)}}{\la^{-1+\mu}\rhob}\right|^{\theta s} dz &\leq c\la^{(p-s)\theta} \Bigg[ \fiint_{\mcp^{\la}_{\rhob}(z_0)}\left|\nabla u\right|^{\theta s}dz +  \la^{(2-p)\theta s} \left(\fiint_{\mcp^{\la}_{\rhob}(z_0)}|\nabla u|^{\theta p} dz\right)^{\frac{s(p-1)}{p}}\\
&+ \la^{(2-p)\theta s }\la^{(p-q)\frac{\theta s}{q}}\left(\fiint_{\mcp^{\la}_{\rhob}(z_0)}\inf_{w\in \mcp^{\la}_{\rhob}(z_0)}a(w)^{\theta}|\nabla u|^{\theta q} dz\right)^{\frac{s(q-1)}{q}}\\
&+ \la^{(2-p)\theta s }\la^{(p-s)\theta}\left(\fiint_{\mcp^{\la}_{\rhob}(z_0)}\inf_{w\in \mcp^{\la}_{\rhob}(z_0)}b(w)^{\theta}|\nabla u|^{\theta s} dz\right)^{s-1}\Bigg],
\end{align*}
\end{enumerate}
hold whenever $\mcp^{\la}_{\rhob}(z_0) \subset \Omega_T.$
\end{lemma}
\begin{proof} We split the proof into six steps.\\
\vspace{.2cm}
\noindent {\bf Step 1.} In this step, we prove the following estimate. \textbf{Claim:}
\begin{align}\label{claim_lem7.3}
  &\fiint_{\mcp^{\la}_{\rhob}(z_0)}\left(|\nabla u|^{p-1}+a(z)|\nabla u|^{q-1}+b(z)|\nabla u|^{s-1}\right) dz \leq \fiint_{\mcp^{\la}_{\rhob}(z_0)}|\nabla u|^{p-1}dz + c\la^{p-q} \fiint_{\mcp^{\la}_{\rhob}(z_0)}|\nabla u|^{q-1}dz \nonumber\\
   &+c\la^{p-s} \fiint_{\mcp^{\la}_{\rhob}(z_0)}|\nabla u|^{s-1}dz + c \left(\fiint_{\mcp^{\la}_{\rhob}(z_0)}|\nabla u|^{\theta p} dz\right)^{\frac{p-1}{\theta p}}.
\end{align}
Note that we only need to estimate the term involving $a(z)$ and $b(z)$ from the left hand side of \eqref{claim_lem7.3}. We shall begin with the term involving $a(z)$ and the estimation of the term involving $b(z)$ will be similar.
\begin{align*}
    \fiint_{\mcp^{\la}_{\rhob}(z_0)}a(z)|\nabla u|^{q-1}dz &\leq \fiint_{\mcp^{\la}_{\rhob}(z_0)}|a(z)-a(z_0)||\nabla u|^{q-1} dz + \fiint_{\mcp^{\la}_{\rhob}(z_0)}a(z_0)|\nabla u|^{q-1}dz\\
    & \leq\underbrace{c[a]_{\alpha} \max\left\{\la^{-1+\mu}\rhob, \la^{\frac{2\mu-p}{2}}\rhob\right\}^{\alpha}\fiint_{\mcp^{\la}_{\rhob}(z_0)}|\nabla u|^{q-1}dz}_{I}\\
    &+ K \la^{p-q}\fiint_{\mcp^{\la}_{\rhob}(z_0)}|\nabla u|^{q-1}dz.
\end{align*}
In the first inequality, we used H\"older regularity of $a(\cdot)$ and in the second inequality we used \textbf{p1} of \cref{assmp1a}. 

\noindent \textbf {Case $\max\left\{\la^{-1+\mu}\rhob, \la^{\frac{2\mu-p}{2}}\rhob\right\}^{\alpha}=\left(\la^{-1+\mu}\rhob\right)^{\alpha}$:} In this case, using H\"older's inequality we have
\begin{align}\label{EQUATION7.3}
I=c[a]_{\alpha}\left(\la^{-1+\mu}\rhob\right)^{\alpha}\fiint_{\mcp^{\la}_{\rhob}(z_0)}|\nabla u|^{q-1}dz\leq c[a]_{\alpha}\left(\la^{-1+\mu}\rhob\right)^{\alpha}\left(\fiint_{\mcp^{\la}_{\rhob}(z_0)}|\nabla u|^{\theta p}dz\right)^{\frac{q-1}{\theta p}}.
\end{align}
Note that from \textbf{p3} of \cref{assmp1c}, we have 
\begin{align*}
    \la^p < \frac{K \la^{n-(n+2)\mu} \la^p}{(2[a]_{\alpha})^{\frac{n+2}{\alpha}}\rhoa^{n+2}}.
\end{align*}
This immediately implies that
\begin{align}\label{Estimates on raddi with k}
    \rhoa^{\alpha}\la^{\alpha\left(\mu-\frac{n}{n+2}\right)}\leq \frac{K^{\frac{\alpha}{n+2}}}{2[a]_{\alpha}}\leq \frac{K}{2[a]_{\alpha}}.
\end{align}
Using the range of $q$ from \eqref{def_pq}, i.e.,
\begin{align*}
    q\leq p+\frac{2\alpha}{n+2},
\end{align*}
we have
\begin{align*}
   \rhoa^{\alpha}\la^{\left(\mu-\frac{n}{n+2}\right)\alpha}\la^q \leq \frac{K}{2[a]_{\alpha}}\la^p \la^{\frac{2\alpha}{n+2}} \implies \left(\la^{-1+\mu}\rhoa\right)^{\alpha} \leq \frac{K}{2[a]_{\alpha}}\la^{p-q}.
\end{align*}
Moreover, since $\rhob$ and $\rhoa$ are comparable, we deduce that
\begin{align}\label{EQUATION7.4}
    \left(\la^{-1+\mu}\rhob\right)^{\alpha}\leq 2^{\alpha}\frac{K}{2[a]_{\alpha}}\la^{p-q}.
\end{align}
Thus using \eqref{EQUATION7.4} in \eqref{EQUATION7.3} and for some $\upgamma>0$ to be chosen later, we find
\begin{align*}
    I &\leq c \la^{p-q} \left(\fiint_{\mcp^{\la}_{\rhob}(z_0)}|\nabla u|^{\theta p}dz\right)^{\frac{q-1}{\theta p}}=c\la^{p-q}\left(\fiint_{\mcp^{\la}_{\rhob}(z_0)}|\nabla u|^{\theta p} dz\right)^{\frac{\upgamma}{\theta p}}\left(\fiint_{\mcp^{\la}_{\rhob}(z_0)}|\nabla u|^{\theta p} dz\right)^{\frac{q-1-\upgamma}{\theta p}}\\
   &\overset{\text{H\"{o}lder's inequality}}{\leq} c\la^{p-q}\left(\fiint_{\mcp^{\la}_{\rhob}(z_0)}|\nabla u|^p dz\right)^{\frac{\upgamma}{p}} \left(\fiint_{\mcp^{\la}_{\rhob}(z_0)}|\nabla u|^{\theta p} dz\right)^{\frac{q-1-\upgamma}{\theta p}}.
\end{align*}
Now using the assumption \textbf{p2} of \cref{assmp1a} on $p$-intrinsic cylinders, we further have
\begin{align*}
    I \leq c \la^{p-q}\la^{\upgamma}\left(\fiint_{\mcp^{\la}_{\rhob}(z_0)}|\nabla u|^{\theta p} dz\right)^{\frac{q-1-\gamma}{\theta p}}.
\end{align*}
At this point, we choose $\upgamma=q-p>0,$ and finally find
\begin{align}\label{EQUATION7.5}
I=c[a]_{\alpha}\left(\la^{-1+\mu}\rhob\right)^{\alpha}\fiint_{\mcp^{\la}_{\rhob}(z_0)}|\nabla u|^{q-1}dz \leq c \left(\fiint_{\mcp^{\la}_{\rhob}(z_0)}|\nabla u|^{\theta p} dz\right)^{\frac{p-1}{\theta p}}.
\end{align}

\noindent \textbf{Case $\max\left\{\la^{-1+\mu}\rhob, \la^{\frac{2\mu-p}{2}}\rhob\right\}^{\alpha}=\left(\la^{\frac{2\mu-p}{2}}\rhob\right)^{\alpha}$ :} In this case,  proceeding as above and using the range of $q$ form \eqref{def_pq}, i.e.,
\begin{align*}
    q \leq p+\alpha \left(\frac{p}{2}-\frac{n}{n+2}\right),
\end{align*}
we arrive at the same conclusion as \eqref{EQUATION7.4}. Hence using this we again obtain \eqref{EQUATION7.5}.  By similar arguments presented above and using the appropriate range of $s$ for different cases, one can obtain
\begin{align}\label{EQQQ3.7}
    \fiint_{\mcp^{\la}_{\rhob}(z_0)}b(z)|\nabla u|^{s-1}\,dz \leq c \left(\fiint_{\mcp^{\la}_{\rhob}(z_0)}|\nabla u|^{\theta p}\, dz\right)^{\frac{p-1}{\theta p}}+K\la^{p-s}\fiint_{\mcp^{\la}_{\rhob}(z_0)}|\nabla u|^{s-1}\,dz
\end{align}
The proof of the claim \eqref{claim_lem7.3} is finished combining \eqref{EQUATION7.5}-\eqref{EQQQ3.7}.

\noindent {\bf Step 2.} Now choosing $m=p$ and applying the estimate from the previous step in \cref{parabolic poincare for double phase}, we obtain
\begin{align}\label{EQUATION7.7}
 \fiint_{\mcp^{\la}_{\rhob}(z_0)}\left|\frac{u-(u)_{\mcp^{\la}_{\rhob}(z_0)}}{\la^{-1+\mu}\rhob}\right|^{\theta p}\,dz &\leq c \fiint_{\mcp^{\la}_{\rhob}(z_0)}|\nabla u|^{\theta p} dz+ c\left(\la^{2-p}\fiint_{\mcp^{\la}_{\rhob}(z_0)}\tilde{H}(z, \nabla u)\, dz\right)^{\theta p}\nonumber \\
 &\leq c \fiint_{\mcp^{\la}_{\rhob}(z_0)}|\nabla u|^{\theta p} dz + c \la^{(2-p)\theta p}\left(\fiint_{\mcp^{\la}_{\rhob}(z_0)}|\nabla u|^{p-1}dz\right)^{\theta p}\nonumber\\
 &+ c \la^{(2-q)\theta p}\left(\fiint_{\mcp^{\la}_{\rhob}(z_0)}|\nabla u|^{q-1}dz\right)^{\theta p} + c \la^{(2-s)\theta p}\left(\fiint_{\mcp^{\la}_{\rhob}(z_0)}|\nabla u|^{s-1}dz\right)^{\theta p}\nonumber\\
 &+ c \la^{(2-p)\theta p}\left(\fiint_{\mcp^{\la}_{\rhob}(z_0)}|\nabla u|^{\theta p}dz\right)^{p-1}.
\end{align}
Note that an application of H\"older's inequality gives
\begin{align*}
 \left(\fiint_{\mcp^{\la}_{\rhob}(z_0)}|\nabla u|^{p-1}dz\right)^{\theta p} \apprle \left(\fiint_{\mcp^{\la}_{\rhob}(z_0)}|\nabla u|^{\theta p}dz\right)^{p-1}\,\,\, \text{for}\,\,\,\,\, \frac{p-1}{p} <\theta \leq 1 ,  
\end{align*}
\begin{align*}
 \left(\fiint_{\mcp^{\la}_{\rhob}(z_0)}|\nabla u|^{q-1}dz\right)^{\theta p} \apprle \left(\fiint_{\mcp^{\la}_{\rhob}(z_0)}|\nabla u|^{\theta p}dz\right)^{q-1}\,\,\, \text{for}\,\,\,\,\, \frac{q-1}{p} <\theta \leq 1, 
\end{align*}
and 
\begin{align*}
 \left(\fiint_{\mcp^{\la}_{\rhob}(z_0)}|\nabla u|^{s-1}dz\right)^{\theta p} \apprle \left(\fiint_{\mcp^{\la}_{\rhob}(z_0)}|\nabla u|^{\theta p}dz\right)^{s-1}\,\,\, \text{for}\,\,\,\,\, \frac{s-1}{p} <\theta \leq 1 .  
\end{align*}
Substituting the above estimates in \eqref{EQUATION7.7}, we get \ref{1_lem7.3}.

\vspace{.3cm}
\noindent {\bf Step 3.} In this step we prove \ref{2_lem7.3}.
\begin{align*}
     \fiint_{\mcp^{\la}_{\rhob}(z_0)}a(z)^{\theta}\left|\frac{u-(u)_{\mcp^{\la}_{\rhob}(z_0)}}{\la^{-1+\mu}\rhob}\right|^{\theta q} dz &\leq  \fiint_{\mcp^{\la}_{\rhob}(z_0)}|a(z)-a(z_0)|^{\theta}\left|\frac{u-(u)_{\mcp^{\la}_{\rhob}(z_0)}}{\la^{-1+\mu}\rhob}\right|^{\theta q} dz\\
     &+ \fiint_{\mcp^{\la}_{\rhob}(z_0)}a(z_0)^{\theta}\left|\frac{u-(u)_{\mcp^{\la}_{\rhob}(z_0)}}{\la^{-1+\mu}\rhob}\right|^{\theta q} dz.
\end{align*}

\vspace{.3cm}
Using H\"older regularity of $a(z)$ and \textbf{p1} of \cref{assmp1a}, we get
\begin{align}\label{EQUATION7.8}
    \fiint_{\mcp^{\la}_{\rhob}(z_0)}a^{\theta}(z)\left|\frac{u-(u)_{\mcp^{\la}_{\rhob}(z_0)}}{\la^{-1+\mu}\rhob}\right|^{\theta q} dz &\leq \max\left\{\la^{-1+\mu}\rhob, \la^{\frac{2\mu-p}{2}}\rhob\right\}^{\alpha \theta}\fiint_{\mcp^{\la}_{\rhob}(z_0)}\left|\frac{u-(u)_{\mcp^{\la}_{\rhob}(z_0)}}{\la^{-1+\mu}\rhob}\right|^{\theta q} dz\nonumber\\
    &+K \la^{(p-q)\theta}\fiint_{\mcp^{\la}_{\rhob}(z_0)}\left|\frac{u-(u)_{\mcp^{\la}_{\rhob}(z_0)}}{\la^{-1+\mu}\rhob}\right|^{\theta q} dz.
\end{align}

\vspace{.3cm}
Moreover, from {\bf Step 1} we have 
\begin{align}\label{EQUATION7.9}
    \max\left\{\la^{-1+\mu}\rhob, \la^{\frac{2\mu-p}{2}}\rhob\right\}^{\alpha \theta}\leq c \la^{(p-q)\theta}\,\,\,\, \text{for}\,\,\,\, \frac{2n}{n+2}<p< \infty.
\end{align}

\vspace{.3cm}
Now choosing $m=q$ in \cref{parabolic poincare for double phase} and using the estimate from {\bf Step 1} we get
\begin{align}\
&\fiint_{\mcp^{\la}_{\rhob}(z_0)}\left|\frac{u-(u)_{\mcp^{\la}_{\rhob}(z_0)}}{\la^{-1+\mu}\rhob}\right|^{\theta q} dz\nonumber\\
&\leq c\fiint_{\mcp^{\la}_{\rhob}(z_0)}|\nabla u|^{\theta q} dz+ c\left(\la^{2-p}\fiint_{\mcp^{\la}_{\rhob}(z_0)}\left(|\nabla u|^{p-1}+a(z)|\nabla u|^{q-1}+b(z)|\nabla u|^{s-1}\right) dz\right)^{\theta q}\label{EQUATION3.9}\\
&\leq c \fiint_{\mcp^{\la}_{\rhob}(z_0)}|\nabla u|^{\theta q} dz + c \la^{(2-p)\theta q}\left(\fiint_{\mcp^{\la}_{\rhob}(z_0)}|\nabla u|^{p-1}dz\right)^{\theta q}\nonumber\\
 &+ c \la^{(2-q)\theta q}\left(\fiint_{\mcp^{\la}_{\rhob}(z_0)}|\nabla u|^{q-1}dz\right)^{\theta q} +c \la^{(2-s)\theta q}\left(\fiint_{\mcp^{\la}_{\rhob}(z_0)}|\nabla u|^{s-1}dz\right)^{\theta q}\nonumber\\
 &+ c \la^{(2-p)\theta q}\left(\fiint_{\mcp^{\la}_{\rhob}(z_0)}|\nabla u|^{\theta p}dz\right)^{\frac{(p-1)q}{p}}\label{EQUATION7.10}.
\end{align}
Using the H\"older inequality in each term on the right hand side of the above inequality, we get
\begin{align*}
\left(\fiint_{\mcp^{\la}_{\rhob}(z_0)}|\nabla u|^{p-1}dz\right)^{\theta q} \apprle \left(\fiint_{\mcp^{\la}_{\rhob}(z_0)}|\nabla u|^{\theta q} dz\right)^{p-1}\,\,\,\, \text{for}\,\,\,\,\,\, \frac{p-1}{q}< \theta\leq 1,    
\end{align*}
\begin{align*}
 \left(\fiint_{\mcp^{\la}_{\rhob}(z_0)}|\nabla u|^{q-1}dz\right)^{\theta q} \apprle \left(\fiint_{\mcp^{\la}_{\rhob}(z_0)}|\nabla u|^{\theta q}\right)^{q-1} \,\,\,\, \text{for}\,\,\,\,\,\, \frac{q-1}{q}< \theta\leq 1,    
\end{align*}
\begin{align*}
 \left(\fiint_{\mcp^{\la}_{\rhob}(z_0)}|\nabla u|^{s-1}dz\right)^{\theta q} \apprle \left(\fiint_{\mcp^{\la}_{\rhob}(z_0)}|\nabla u|^{\theta q}\right)^{s-1} \,\,\,\, \text{for}\,\,\,\,\,\, \frac{s-1}{q}< \theta\leq 1,    
\end{align*}
and 
\begin{align*}
\left(\fiint_{\mcp^{\la}_{\rhob}(z_0)}|\nabla u|^{\theta p}dz\right)^{\frac{(p-1)q}{p}} \apprle \left(\fiint_{\mcp^{\la}_{\rhob}(z_0)}|\nabla u|^{\theta q}dz\right)^{p-1}.
\end{align*}
Note that the range of $\theta$ in the above two inequalities are admissible since $\theta > \frac{s-1}{p}\geq \frac{p-1}{q},$ $\theta> \frac{q-1}{p}\geq \frac{q-1}{q},$ and $\theta > \frac{s-1}{p}\geq \frac{s-1}{q}.$ Now applying the above estimates in \eqref{EQUATION7.10} we obtain
\begin{align}
 \fiint_{\mcp^{\la}_{\rhob}(z_0)}\left|\frac{u-(u)_{\mcp^{\la}_{\rhob}(z_0)}}{\la^{-1+\mu}\rhob}\right|^{\theta q} dz &\leq c \fiint_{\mcp^{\la}_{\rhob}(z_0)}\left|\nabla u\right|^{\theta q}dz + c \la^{(2-p)\theta q} \left(\fiint_{\mcp^{\la}_{\rhob}(z_0)}|\nabla u|^{\theta q} dz\right)^{p-1}\nonumber\\
&+c \la^{(2-q)\theta q} \left(\fiint_{\mcp^{\la}_{\rhob}(z_0)}|\nabla u|^{\theta q} dz\right)^{q-1}+c \la^{(2-s)\theta q} \left(\fiint_{\mcp^{\la}_{\rhob}(z_0)}|\nabla u|^{\theta q} dz\right)^{s-1}\label{EQUATION3.11}.  
\end{align}
Lastly, using \eqref{EQUATION7.9} and the above estimate in \eqref{EQUATION7.8}, we prove \ref{2_lem7.3}. 

\vspace{.5cm}
\noindent {\bf Step 4.} In this step we prove \ref{3_lem7.3}. First, we perform an estimate similar to \textbf{Step 1.} Indeed, we estimate
\begin{align*}
    \fiint_{\mcp^{\la}_{\rhob}(z_0)}\tilde{H}(z, \nabla u) dz &\leq \fiint_{\mcp^{\la}_{\rhob}(z_0)}|\nabla u|^{p-1} dz + \underbrace{\fiint_{\mcp^{\la}_{\rhob}(z_0)}|a(z)-\inf_{w\in \mcp^{\la}_{\rhob}(z_0)}a(w)||\nabla u|^{q-1} dz}_{I}\\
    &+\underbrace{\fiint_{\mcp^{\la}_{\rhob}(z_0)}|b(z)-\inf_{w\in \mcp^{\la}_{\rhob}(z_0)}b(w)||\nabla u|^{s-1} dz}_{II}+ \fiint_{\mcp^{\la}_{\rhob}(z_0)} \inf_{w\in \mcp^{\la}_{\rhob}(z_0)}a(w) |\nabla u|^{q-1} dz\\
    &+\fiint_{\mcp^{\la}_{\rhob}(z_0)} \inf_{w\in \mcp^{\la}_{\rhob}(z_0)}b(w) |\nabla u|^{s-1} dz.
\end{align*}

\vspace{.4cm}
As in \textbf{Step 1.}, we can estimate $I$  and $II$ as
\begin{align*}
    I \leq c\left(\fiint_{\mcp^{\la}_{\rhob}(z_0)}|\nabla u|^{\theta p} dz\right)^{\frac{p-1}{\theta p}}\,\,\, \text{and}\,\,\, \,\,II\leq c\left(\fiint_{\mcp^{\la}_{\rhob}(z_0)}|\nabla u|^{\theta p} dz\right)^{\frac{p-1}{\theta p}}.
\end{align*}
Hence, we have
\begin{align*}
    &\left(\la^{2-p}\fiint_{\mcp^{\la}_{\rhob}(z_0)}\left(|\nabla u|^{p-1}+a(z)|\nabla u|^{q-1}+b(z)|\nabla u|^{s-1}\right)dz\right)^{\theta q}\\
    &\leq \la^{(2-p)\theta q} \left(\fiint_{\mcp^{\la}_{\rhob}(z_0)}|\nabla u|^{p-1} dz\right)^{\theta q}+ \la^{(2-p)\theta q} \left(\fiint_{\mcp^{\la}_{\rhob}(z_0)} \inf_{w\in \mcp^{\la}_{\rhob}(z_0)}a(w) |\nabla u|^{q-1} dz\right)^{\theta q}\\
    &+ c\la^{(2-p)\theta q} \left(\fiint_{\mcp^{\la}_{\rhob}(z_0)}|\nabla u|^{\theta p} dz\right)^{\frac{(p-1)q}{p}}+c\la^{(2-p)\theta q}\left(\fiint_{\mcp^{\la}_{\rhob}(z_0)}\inf_{w\in \mcp^{\la}_{\rhob}(z_0)}b(w)|\nabla u|^{s-1}\,dz\right)^{\theta q}\\
    &\leq c\la^{(2-p)\theta q}\left(\fiint_{\mcp^{\la}_{\rhob}(z_0)}|\nabla u|^{\theta p} dz\right)^{\frac{q(p-1)}{p}}+c \la^{(2-p)\theta q}\left(\fiint_{\mcp^{\la}_{\rhob}(z_0)}\inf_{w\in \mcp^{\la}_{\rhob}(z_0)}a(w)^{\frac{\theta q}{q-1}}|\nabla u|^{\theta q} dz\right)^{q-1} \\
    &+\la^{(2-p)\theta q}\inf_{w\in \mcp^{\la}_{\rhob}(z_0)}b(w)^{\frac{\theta q}{s}}\left(\fiint_{\mcp^{\la}_{\rhob}(z_0)}\inf_{w\in \mcp^{\la}_{\rhob}(z_0)}b(w)^{\theta}|\nabla u|^{\theta s} dz\right)^{\frac{q(s-1)}{s}}\\
    &\leq c \la^{(2-p)\theta q}\left(\fiint_{\mcp^{\la}_{\rhob}(z_0)}|\nabla u|^{\theta p} dz\right)^{\frac{q(p-1)}{p}}+ \la^{(2-p)\theta q}\inf_{w\in \mcp^{\la}_{\rhob}(z_0)}a(w)^{\theta}\left(\fiint_{\mcp^{\la}_{\rhob}(z_0)}\inf_{w\in \mcp^{\la}_{\rhob}(z_0)}a(w)^{\theta}|\nabla u|^{\theta q} dz\right)^{q-1}\\
    &+\la^{(2-p)\theta q}\inf_{w\in \mcp^{\la}_{\rhob}(z_0)}b(w)^{\frac{\theta q}{s}}\left(\fiint_{\mcp^{\la}_{\rhob}(z_0)}\inf_{w\in \mcp^{\la}_{\rhob}(z_0)}b(w)^{\theta}|\nabla u|^{\theta s} dz\right)^{\frac{q(s-1)}{s}}.
\end{align*}
Therefore, applying the above inequality in \eqref{EQUATION3.9}, we obtain
\begin{align}
 \fiint_{\mcp^{\la}_{\rhob}(z_0)}\left|\frac{u-(u)_{\mcp^{\la}_{\rhob}(z_0)}}{\la^{-1+\mu}\rhob}\right|^{\theta q} dz &\leq c\fiint_{\mcp^{\la}_{\rhob}(z_0)}|\nabla u|^{\theta q} dz + c \la^{(2-p)\theta q}\left(\fiint_{\mcp^{\la}_{\rhob}(z_0)}|\nabla u|^{\theta p}\right)^{\frac{q(p-1)}{p}}\nonumber\\
 &+ c \la^{(2-p)\theta q}\la^{(p-q)\theta}\left(\fiint_{\mcp^{\la}_{\rhob}(z_0)}\inf_{w\in \mcp^{\la}_{\rhob}(z_0)}a(w)^{\theta}|\nabla u|^{\theta q}\right)^{q-1}\nonumber\\
 &+ c \la^{(2-p)\theta q}\la^{(p-s)\frac{\theta q}{s}}\left(\fiint_{\mcp^{\la}_{\rhob}(z_0)}\inf_{w\in \mcp^{\la}_{\rhob}(z_0)}b(w)^{\theta}|\nabla u|^{\theta s}\right)^{\frac{q(s-1)}{s}}\label{EQUATION3.12}.
\end{align}
Using the above estimate in \eqref{EQUATION7.8}, we get \ref{3_lem7.3}.

\vspace{.5cm}
\noindent {\bf Step 5.} In this step, we prove \ref{4_lem7.3}. Since the proof can be deduced from {\bf Step 3.}, here we omit the details. Indeed, we have 
\begin{align*}
    \fiint_{\mcp^{\la}_{\rhob}(z_0)}b(z)^{\theta}\left|\frac{u-(u)_{\mcp^{\la}_{\rhob}(z_0)}}{\la^{-1+\mu}\rhob}\right|^{\theta s}\, dz \apprle \la^{(p-s)\theta}\fiint_{\mcp^{\la}_{\rhob}(z_0)}\left|\frac{u-(u)_{\mcp^{\la}_{\rhob}(z_0)}}{\la^{-1+\mu}\rhob}\right|^{\theta s}
\end{align*}
 and following the same steps as previous we obtain 
 \begin{align}
 \fiint_{\mcp^{\la}_{\rhob}(z_0)}\left|\frac{u-(u)_{\mcp^{\la}_{\rhob}(z_0)}}{\la^{-1+\mu}\rhob}\right|^{\theta s} &\leq \fiint_{\mcp^{\la}_{\rhob}(z_0)}\left|\nabla u\right|^{\theta s}dz + \la^{(2-p)\theta s} \left(\fiint_{\mcp^{\la}_{\rhob}(z_0)}|\nabla u|^{\theta s} dz\right)^{p-1}\nonumber\\
&+ \la^{(2-q)\theta s} \left(\fiint_{\mcp^{\la}_{\rhob}(z_0)}|\nabla u|^{\theta s} dz\right)^{q-1}+ \la^{(2-s)\theta s}\left(\fiint_{\mcp^{\la}_{\rhob}(z_0)}|\nabla u|^{\theta s} dz\right)^{s-1},\label{EQQQQ3.15}   
 \end{align}
 and hence \ref{4_lem7.3}.

\vspace{.5cm}
\noindent{\bf Step 6.} In this step we provide the proof of \ref{5_lem7.3} which is similar to the arguments in {\bf Step 4.} Hence, we omit the details. Indeed, in this case we obtain
\begin{align}
\fiint_{\mcp^{\la}_{\rhob}(z_0)}\left|\frac{u-(u)_{\mcp^{\la}_{\rhob}(z_0)}}{\la^{-1+\mu}\rhob}\right|^{\theta s} dz &\leq c\fiint_{\mcp^{\la}_{\rhob}(z_0)}\left|\nabla u\right|^{\theta s}dz +  \la^{(2-p)\theta s} \left(\fiint_{\mcp^{\la}_{\rhob}(z_0)}|\nabla u|^{\theta p} dz\right)^{\frac{s(p-1)}{p}}\nonumber\\
&+ \la^{(2-p)\theta s }\la^{(p-q)\frac{\theta s}{q}}\left(\fiint_{\mcp^{\la}_{\rhob}(z_0)}\inf_{w\in \mcp^{\la}_{\rhob}(z_0)}a(w)^{\theta}|\nabla u|^{\theta q} dz\right)^{\frac{s(q-1)}{q}}\nonumber\\
&+ \la^{(2-p)\theta s }\la^{(p-s)\theta}\left(\fiint_{\mcp^{\la}_{\rhob}(z_0)}\inf_{w\in \mcp^{\la}_{\rhob}(z_0)}b(w)^{\theta}|\nabla u|^{\theta s} dz\right)^{s-1}\label{EQQQQ3.16}.   
\end{align}
This completes the proof.
\end{proof}
The next lemma shows that we can further estimate the right-hand sides of \ref{1_lem7.3},\ref{2_lem7.3}, \ref{3_lem7.3}, \ref{4_lem7.3} and \ref{5_lem7.3} of \cref{p_intrinsic poincare 1}.
\begin{lemma}\label{p_intrinsic poincare 2}
Let $u$ be a weak solution to \eqref{main_eqn} and let the \cref{assmp1a} be in force. Then for any $\theta\in \left(\max\left\{\frac{s-1}{p}, \frac{1}{p}\right\}, 1\right]$ and $\varepsilon \in (0,1),$ there exists a constant $c=c(\textnormal{\texttt{data}})$ such that the following estimates
\begin{enumerate}[label=(\roman*),series=theoremconditions]
\item \label{1_lem7.4} 
\begin{align*}
\fiint_{\mcp^\la_{\rhob}(z_0)}\left|\frac{u-(u)_{\mcp^{\la}_{\rhob}(z_0)}}{\la^{-1+\mu}\rhob}\right|^{\theta p} dz \leq c \fiint_{\mcp^{\la}_{\rhob}(z_0)}\left|\nabla u\right|^{\theta p}dz +\varepsilon \la^{\theta p},
\end{align*}
\item \label{2_lem7.4}
\begin{align*}
    \fiint_{\mcp^{\la}_{\rhob}(z_0)}\left|\frac{u-(u)_{\mcp^{\la}_{\rhob}(z_0)}}{\la^{-1+\mu}\rhob}\right|^{\theta q} dz &\leq c \fiint_{\mcp^{\la}_{\rhob}(z_0)}\left|\nabla u\right|^{\theta q}dz+ c\la^{(q-p)\theta}\left(\fiint_{\mcp^{\la}_{\rhob}(z_0)}a^{\theta}(z)|\nabla u|^{\theta q}dz\right)\\
    &+c \la^{(2-p)\theta q}\la^{(p-s)\frac{\theta q}{s}}\left(\fiint_{\mcp^{\la}_{\rhob}(z_0)}\inf_{w\in \mcp^{\la}_{\rhob}(z_0)}b(w)^{\theta}|\nabla u|^{\theta s}\,dz\right)^{\frac{q(s-1)}{s}}+ \varepsilon \la^{\theta q},
\end{align*}
\item \label{3_lem7.4}
\begin{align*}
    \fiint_{\mcp^{\la}_{\rhob}(z_0)}\left|\frac{u-(u)_{\mcp^{\la}_{\rhob}(z_0)}}{\la^{-1+\mu}\rhob}\right|^{\theta s} dz &\leq c \fiint_{\mcp^{\la}_{\rhob}(z_0)}\left|\nabla u\right|^{\theta s}dz+ c\la^{(s-p)\theta}\left(\fiint_{\mcp^{\la}_{\rhob}(z_0)}b^{\theta}(z)|\nabla u|^{\theta s}dz\right)\\
    &+c \la^{(2-p)\theta s}\la^{(p-q)\frac{\theta s}{q}}\left(\fiint_{\mcp^{\la}_{\rhob}(z_0)}\inf_{w\in \mcp^{\la}_{\rhob}(z_0)}a(w)^{\theta}|\nabla u|^{\theta q}\,dz\right)^{\frac{s(q-1)}{q}}+ \varepsilon \la^{\theta s},
\end{align*}
\end{enumerate}
hold whenever $\mcp^{\la}_{\rhob}(z_0)\subset \Omega_T.$
\end{lemma}
\begin{proof}
We start with the proof of \ref{1_lem7.4}. Let us consider the following two cases.

\vspace{.3cm}
\noindent {\bf Case $p\geq2$ :} This case is a simple application of H\"older's inequality and \cref{assmp1b}. Indeed, the second term in \ref{1_lem7.3} of \cref{p_intrinsic poincare 1} can be estimated as
\begin{align}\label{EQUATION7.11}
    \la^{(2-p)\theta p} \left(\fiint_{\mcp^{\la}_{\rhob}(z_0)}|\nabla u|^{\theta p} dz\right)^{p-1}&=\la^{(2-p)\theta p}\left(\fiint_{\mcp^{\la}_{\rhob}(z_0)}|\nabla u|^{\theta p} dz\right)^{p-2}\left(\fiint_{\mcp^{\la}_{\rhob}(z_0)}|\nabla u|^{\theta p} dz\right)\nonumber\\
    &\leq \la^{(2-p)\theta p} \left(\fiint_{\mcp^{\la}_{\rhob}(z_0)}|\nabla u|^{p} dz\right)^{(p-2)\theta}\left(\fiint_{\mcp^{\la}_{\rhob}(z_0)}|\nabla u|^{\theta p} dz\right)\nonumber\\
    &\leq \la^{(2-p)\theta p}\la^{(p-2)\theta p}\left(\fiint_{\mcp^{\la}_{\rhob}(z_0)}|\nabla u|^{\theta p} dz\right). 
\end{align}
The third and the fourth term in \ref{1_lem7.3} of \cref{p_intrinsic poincare 1} can be estimated similarly. Thus we obtained
\begin{align*}
    \fiint_{\mcp^\la_{\rhob}(z_0)}\left|\frac{u-(u)_{\mcp^{\la}_{\rhob}(z_0)}}{\la^{-1+\mu}\rhob}\right|^{\theta p} dz \leq c \left(\fiint_{\mcp^{\la}_{\rhob}(z_0)}|\nabla u|^{\theta p} dz\right). 
\end{align*}
\noindent {\bf Case $p<2$ :} In this case, we use Young's inequality with $\left(\frac{1}{2-p}, \frac{1}{p-1}\right).$ Note that, since $1<p<2,$ $\frac{1}{2-p}>1$ and $\frac{1}{p-1}>1.$ Applying Young's inequality in the second term in \ref{1_lem7.3} of \cref{p_intrinsic poincare 1}, we get
\begin{align}\label{EQUATION7.12}
  \la^{(2-p)\theta p} \left(\fiint_{\mcp^{\la}_{\rhob}(z_0)}|\nabla u|^{\theta p} dz\right)^{p-1} \leq \varepsilon \la^{\theta p}  + c(\varepsilon) \left(\fiint_{\mcp^{\la}_{\rhob}(z_0)}|\nabla u|^{\theta p} dz\right),
\end{align}
and similarly from the third and fourth term in \ref{1_lem7.3} of \cref{p_intrinsic poincare 1} we get
\begin{align}\label{EQUATION7.13}
  \la^{(2-q)\theta p} \left(\fiint_{\mcp^{\la}_{\rhob}(z_0)}|\nabla u|^{\theta p} dz\right)^{q-1} \leq \varepsilon \la^{\theta p}  + c(\varepsilon) \left(\fiint_{\mcp^{\la}_{\rhob}(z_0)}|\nabla u|^{\theta p} dz\right),  
\end{align}
and 
\begin{align}\label{EQUATIONNN7.13}
  \la^{(2-s)\theta p} \left(\fiint_{\mcp^{\la}_{\rhob}(z_0)}|\nabla u|^{\theta p} dz\right)^{s-1} \leq \varepsilon \la^{\theta p}  + c(\varepsilon) \left(\fiint_{\mcp^{\la}_{\rhob}(z_0)}|\nabla u|^{\theta p} dz\right).   
\end{align}
Note that here we assume that $s, q<2.$ In the case of $p<2$ and $s, q>2,$ we can estimate the above term as previous. Combining the estimates, \eqref{EQUATION7.11}, \eqref{EQUATION7.12}, \eqref{EQUATION7.13} and \eqref{EQUATIONNN7.13}  we obtain \ref{1_lem7.4}.

\noindent The proof of \ref{2_lem7.4} is similar.

\noindent{\bf Case $p\geq 2$ :} In this case, we use H\"older's inequality. The second term in right hand side in the expression \eqref{EQUATION3.12} can be estimated as
\begin{align*}
   \la^{(2-p)\theta q} \left(\fiint_{\mcp^{\la}_{\rhob}(z_0)}|\nabla u|^{\theta p} dz\right)^{\frac{q(p-1)}{p}}&=\la^{(2-p)\theta q} \left(\fiint_{\mcp^{\la}_{\rhob}(z_0)}|\nabla u|^{\theta p} dz\right)^{\frac{q(p-2)}{p}} \left(\fiint_{\mcp^{\la}_{\rhob}(z_0)}|\nabla u|^{\theta p} dz\right)^{\frac{q}{p}}\nonumber\\
   &\leq\la^{(2-p)\theta q} \left(\fiint_{\mcp^{\la}_{\rhob}(z_0)}|\nabla u|^{p} dz\right)^{\frac{\theta q(p-2)}{p}} \left(\fiint_{\mcp^{\la}_{\rhob}(z_0)}|\nabla u|^{\theta q} dz\right)\nonumber\\
   &\leq \la^{(2-p)\theta q} \la^{(p-2)\theta q}\left(\fiint_{\mcp^{\la}_{\rhob}(z_0)}|\nabla u|^{\theta q} dz\right).
\end{align*}
The third term in the right hand side of \eqref{EQUATION3.12} can be estimated similarly as
\begin{align*}
 &\la^{(2-p)\theta q}\la^{(p-q)\theta}\left(\fiint_{\mcp^{\la}_{\rhob}(z_0)}\inf_{w\in \mcp^{\la}_{\rhob}(z_0)}a(w)^{\theta}|\nabla u|^{\theta q}dz\right)^{q-1} \\
 &\leq \la^{(2-p)\theta q}\la^{(p-q)\theta}\left(\fiint_{\mcp^{\la}_{\rhob}(z_0)}a(z)^{\theta}|\nabla u|^{\theta q}dz\right)^{q-2}\left(\fiint_{\mcp^{\la}_{\rhob}(z_0)}a(z)^{\theta}|\nabla u|^{\theta q}dz\right)\\
 &\leq \la^{(2-p)\theta q}\la^{(p-q)\theta}\left(\fiint_{\mcp^{\la}_{\rhob}(z_0)}a(z)|\nabla u|^{q}dz\right)^{\theta(q-2)}\left(\fiint_{\mcp^{\la}_{\rhob}(z_0)}a(z)^{\theta}|\nabla u|^{\theta q}dz\right)\\
 & \leq \la^{(2-p)\theta q}\la^{(p-q)\theta} \la^{\theta p(q-2)}\left(\fiint_{\mcp^{\la}_{\rhob}(z_0)}a(z)^{\theta}|\nabla u|^{\theta q}dz\right)=\la^{(q-p)\theta}\left(\fiint_{\mcp^{\la}_{\rhob}(z_0)}a(z)^{\theta}|\nabla u|^{\theta q}dz\right).
\end{align*}
Hence combining all the estimates above, in the case $p\geq 2,$ we get
\begin{align}\label{EQUATION3.16}
    \fiint_{\mcp^{\la}_{\rhob}(z_0)}\left|\frac{u-(u)_{\mcp^{\la}_{\rhob}(z_0)}}{\la^{-1+\mu}\rhob}\right|^{\theta q} dz &\leq c\left(\fiint_{\mcp^{\la}_{\rhob}(z_0)}|\nabla u|^{\theta q} dz\right)+ c\la^{(q-p)\theta}\left(\fiint_{\mcp^{\la}_{\rhob}(z_0)}a(z)^{\theta}|\nabla u|^{\theta q}dz\right)\nonumber\\
    &+c \la^{(2-p)\theta q}\la^{(p-s)\frac{\theta q}{s}}\left(\fiint_{\mcp^{\la}_{\rhob}(z_0)}\inf_{w\in \mcp^{\la}_{\rhob}(z_0)}b(w)^{\theta}|\nabla u|^{\theta s}\right)^{\frac{q(s-1)}{s}}.
\end{align}

\noindent{\bf Case $p<2$ :} In this case, we use Young's inequality with $(\frac{1}{2-p}, \frac{1}{p-1})$. We estimate the second term in the right hand side of \eqref{EQUATION3.11} as
\begin{align}\label{EQUATION7.14}
    c \la^{(2-p)\theta q} \left(\fiint_{\mcp^{\la}_{\rhob}(z_0)}|\nabla u|^{\theta q} dz\right)^{p-1} \leq \varepsilon \la^{\theta q}+c(\varepsilon) \left(\fiint_{\mcp^{\la}_{\rhob}(z_0)}|\nabla u|^{\theta q} dz\right).
\end{align}
Similarly, from the third and the fourth term in the right hand side of \eqref{EQUATION3.11}, we have
\begin{align}\label{EQUATION7.15}
    \la^{(2-q)\theta q} \left(\fiint_{\mcp^{\la}_{\rhob}(z_0)}|\nabla u|^{\theta q} dz\right)^{q-1} \leq \varepsilon \la^{\theta q}+c(\varepsilon ) \left(\fiint_{\mcp^{\la}_{\rhob}(z_0)}|\nabla u|^{\theta q} dz\right)
\end{align}
and 
\begin{align}\label{EQUATIONNN7.15}
    \la^{(2-s)\theta q} \left(\fiint_{\mcp^{\la}_{\rhob}(z_0)}|\nabla u|^{\theta q} dz\right)^{s-1} \leq \varepsilon \la^{\theta q}+c(\varepsilon ) \left(\fiint_{\mcp^{\la}_{\rhob}(z_0)}|\nabla u|^{\theta q} dz\right).
\end{align}
Note that here we assumed $s, q<2.$ The case $p<2$ and $s, q>2$ can occur and in that case, we use Young's inequality for $p<2$ case and H\"older's inequality for $s, q>2$ case as previous. Thus combining the estimates \eqref{EQUATION3.16}, \eqref{EQUATION7.14}, \eqref{EQUATION7.15} and \eqref{EQUATIONNN7.15} we get \ref{2_lem7.4}.
Next, we show \ref{3_lem7.4}.

\vspace{.3cm}
\noindent {\bf Case $p\geq 2$ :} As previous, the second term of \eqref{EQQQQ3.16} can be estimated as
\begin{align*}
 \la^{(2-p)\theta s} \left(\fiint_{\mcp^{\la}_{\rhob}(z_0)}|\nabla u|^{\theta p} dz\right)^{\frac{s(p-1)}{p}} \leq \left(\fiint_{\mcp^{\la}_{\rhob}(z_0)}|\nabla u|^{\theta s}\, dz\right).   
\end{align*}
On the other hand, the fourth term can be estimates as
\begin{align*}
\la^{(2-p)\theta s }\la^{(p-s)\theta}\left(\fiint_{\mcp^{\la}_{\rhob}(z_0)}\inf_{w\in \mcp^{\la}_{\rhob}(z_0)}b(w)^{\theta}|\nabla u|^{\theta s} dz\right)^{s-1} \leq \la^{(s-p)\theta}\left(\fiint_{\mcp^{\la}_{\rhob}(z_0)}b^{\theta}(z)|\nabla u|^{\theta s}\, dz\right).   
\end{align*}
Therefore, in this case we get
\begin{align}
\fiint_{\mcp^{\la}_{\rhob}(z_0)}\left|\frac{u-(u)_{\mcp^{\la}_{\rhob}(z_0)}}{\la^{-1+\mu}\rhob}\right|^{\theta s} dz  &\leq c\fiint_{\mcp^{\la}_{\rhob}(z_0)}\left|\nabla u\right|^{\theta s}dz + c \la^{(s-p)\theta}\left(\fiint_{\mcp^{\la}_{\rhob}(z_0)}b^{\theta}(z)|\nabla u|^{\theta s}\, dz\right) \nonumber\\ 
&+\la^{(2-p)\theta s }\la^{(p-q)\frac{\theta s}{q}}\left(\fiint_{\mcp^{\la}_{\rhob}(z_0)}\inf_{w\in \mcp^{\la}_{\rhob}(z_0)}a(w)^{\theta}|\nabla u|^{\theta q} dz\right)^{\frac{s(q-1)}{q}}.\label{EQQQQ3.25}
\end{align}

\noindent{\bf Case $p<2$ :} In this case, we estimate the right hand side of \eqref{EQQQQ3.15} and get
\begin{align}\label{EQQQQ3.26}
\fiint_{\mcp^{\la}_{\rhob}(z_0)}\left|\frac{u-(u)_{\mcp^{\la}_{\rhob}(z_0)}}{\la^{-1+\mu}\rhob}\right|^{\theta s} dz \leq \varepsilon \la^{\theta s}+ c(\varepsilon) \left(\fiint_{\mcp^{\la}_{\rhob}(z_0)}|\nabla u|^{\theta s}\, dz\right).
\end{align}
Combining the estimates \eqref{EQQQQ3.25}-\eqref{EQQQQ3.26}, we obtain \ref{3_lem7.4}.
This completes the proof.
\end{proof}
Now we state the energy estimate \cref{general caccipoli} in terms of $p$-intrinsic cylinders. Before that let us define appropriate cutoff functions. We choose $R_2=\la^{-1+\mu}\rhob,$ $R_1=\la^{-1+\mu}\rhoa,$ $S_2=\la^{2\mu-p}\rhob^2,$ and $S_1=\la^{2\mu-p}\rhoa^2$ in \cref{general caccipoli}. Furthermore, we have the restriction $\frac{\rhob}{2}\leq \rhoa < \rhob.$

\vspace{.3cm}
\noindent \textbf{Cut-off functions for $p$-phase.} Let $0 < \varrho \leq \rhoa < \rhob \leq 4\varrho$ with the restriction mentioned above and we consider the following cutoff functions:
\begin{align*}
    &\eta:=\eta(x) \in C_c^{\infty}(B_{\la^{-1+\mu}\rhob}(x_0)), \quad  \eta \equiv 1 \,\,\text{on}\, B_{\la^{-1+\mu}\rhoa}(x_0), \quad 0\leq \eta \leq 1 \,\,\,\,\,\text{and}\,\,\,\,\,\, |\nabla \eta| \apprle \frac{1}{\la^{-1+\mu}(\rhob-\rhoa)}\\
    &\zeta:=\zeta(t) \in C^{\infty}_c(t_0-\la^{2\mu-p}(\rhob-h_0)^2, t_0+\la^{2\mu-p}(\rhob-h_0)^2),  \qquad \zeta \equiv 1 \,\,\text{on}\,(t_0-\la^{2\mu-p}\rhoa^2,t_0+\la^{2\mu-p}\rhoa^2), \\
  & 0\leq \zeta \leq 1\,\,\, \text{and}\,\,\, |\partial_t \zeta| \apprle \frac{1}{\la^{2\mu-p}(\rhob-\rhoa)^2}.
\end{align*}
Then the energy estimate for $p$-phase reads as:
\begin{lemma}\label{scl_energy_p}
 Let $u$ be a weak solution of \eqref{main_eqn}. Then we have the following energy estimate:
\begin{align*}
	&\lambda^{p-2} \sup_{I^{\lambda,p}_{\rhoa}(t_0)}\fint_{B^{\lambda}_{\rhoa}(x_0)}\left|\frac{u-(u)_{\mcp^{\lambda}_{\rhoa}}}{\lambda^{-1+\mu}\rhoa}\right|^2 dx + \fiint_{\mcp^{\lambda}_{\rhoa}(z_0)}H\left(z,|\nabla u|\right)\,dz\\
	&\apprle  
	 \fiint_{\mcp^{\lambda}_{\rhob}(z_0)} H\left(z,\left|\frac{u-(u)_{\mcp^{\lambda}_{\rhob}}}{\lambda^{-1+\mu}(\rhob-\rhoa)}\right|\right)\,dz + \lambda^{p-2} \fiint_{\mcp^{\lambda}_{\rhob}(z_0)}\left|\frac{u-(u)_{\mcp^{\lambda}_{\rhob}}}{\lambda^{-1+\mu}(\rhob-\rhoa)}\right|^2 \,dz.
\end{align*}   
\end{lemma}
\begin{lemma}\label{LEMMA7.5}
Let $u$ be a weak solution to \eqref{main_eqn} and let \cref{scl_energy_p} hold with the \cref{assmp1a} in force. Then there exists a constant $c=c(\textnormal{\texttt{data}})$ such that the following estimate,
\begin{align*}
\lambda^{p-2} \sup_{I^{\lambda,p}_{2\varrho}(t_0)}\fint_{B^{\lambda}_{2\varrho}(x_0)}\left|\frac{u-(u)_{\mcp^{\la}_{2\varrho}}}{\lambda^{-1+\mu}2\varrho}\right|^2 dx \leq c \lambda^p
\end{align*}
holds whenever $\mcp^{\la}_{\rhob}(z_0)\subset \Omega_T.$
\end{lemma}
\begin{proof}
We will be estimating the terms on the right hand side of \cref{scl_energy_p}.  For any $2\varrho \leq \rho_1< \rho_2 \leq 4\varrho,$ from the energy estimate in \cref{scl_energy_p} we have
\begin{align}\label{EQUATION7.16}
   &\la^{p-2}\sup_{I^{\la, p}_{\rho_1}(t_0)}\fint_{B^{\la}_{\rho_1}(x_0)}\left|\frac{u-(u)_{\mcp^{\la}_{\rho_1}}}{\la^{-1+\mu}\rho_1}\right|^2 dx\nonumber\\
    &\leq \underbrace{c \left(\frac{\rho_2}{\rho_2-\rho_1}\right)^s\fiint_{\mcp^{\lambda}_{\rho_2}(z_0)} \left(\left|\frac{u-(u)_{\mcp^{\lambda}_{\rho_2}}}{\lambda^{-1+\mu}\rho_2}\right|^p + a(z)\left|\frac{u-(u)_{\mcp^{\lambda}_{\rho_2}}}{\lambda^{-1+\mu}\rho_2}\right|^q+b(z)\left|\frac{u-(u)_{\mcp^{\lambda}_{\rho_2}}}{\lambda^{-1+\mu}\rho_2}\right|^s\right)\,dz}_{I}\nonumber\\
    &+ \underbrace{c\la^{p-2}\left(\frac{\rho_2}{\rho_2-\rho_1}\right)^2\fiint_{\mcp^{\la}_{\rho_2}(z_0)}\left|\frac{u-(u)_{\mcp^{\la}_{\rho_2}}}{\la^{-1+\mu}\rho_2}\right|^2 \,dz}_{II}. 
\end{align}
\noindent \textbf{Estimate for I:} Applying \ref{1_lem7.4} of \cref{p_intrinsic poincare 2} with $\theta=1$ and \textbf{p2} of \cref{assmp1b}, we have
\begin{align} \label{EQUATION3.20}
\left(\frac{\rho_2}{\rho_2-\rho_1}\right)^s \fiint_{\mcp^{\la}_{\rho_2}(z_0)}\left|\frac{u-(u)_{\mcp^{\lambda}_{\rho_2}}}{\lambda^{-1+\mu}\rho_2}\right|^p dz \apprle \left(\frac{\rho_2}{\rho_2-\rho_1}\right)^s \la^p.
\end{align}
Now we estimate
\begin{align*}
    \fiint_{\mcp^{\la}_{\rho_2}(z_0)}a(z) \left|\frac{u-(u)_{\mcp^{\lambda}_{\rho_2}}}{\lambda^{-1+\mu}\rho_2}\right|^q dz &\apprle \fiint_{\mcp^{\la}_{\rho_2}(z_0)}|a(z)-\inf_{w\in \mcp^{\la}_{\rho_2}(z_0)}a(w)|\left|\frac{u-(u)_{\mcp^{\lambda}_{\rho_2}}}{\lambda^{-1+\mu}\rho_2}\right|^q dz \\
    &+ \fiint_{\mcp^{\la}_{\rho_2}(z_0)}\inf_{w\in \mcp^{\la}_{\rho_2}(z_0)}a(w)\left|\frac{u-(u)_{\mcp^{\lambda}_{\rho_2}}}{\lambda^{-1+\mu}\rho_2}\right|^q dz.
\end{align*}
Note that the second term of the above inequality can be estimated by using \ref{2_lem7.4} of \cref{p_intrinsic poincare 2} with $\theta=1$. Indeed,
\begin{align}\label{EQUATION3.21}
 \fiint_{\mcp^{\la}_{\rho_2}(z_0)}\inf_{w\in \mcp^{\la}_{\rho_2}(z_0)}a(w)\left|\frac{u-(u)_{\mcp^{\lambda}_{\rho_2}}}{\lambda^{-1+\mu}\rho_2}\right|^q dz &\apprle \fiint_{\mcp^{\la}_{\rho_2}(z_0)}a(z)|\nabla u|^q dz + \la^{q-p} a(z_0) \fiint_{\mcp^{\la}_{\rho_2}(z_0)}a(z)|\nabla u|^q dz\nonumber\\
 &+\la^{(2-p)q}\la^{(p-s)\frac{q}{s}}a(z_0)\left(\fiint_{\mcp^{\la}_{\rhob}(z_0)}b(z)|\nabla u|^s\, dz\right)^{\frac{q(s-1)}{s}}\nonumber\\
 &+ \varepsilon a(z_0)\la^q \nonumber\\
 &\overset{\cref{assmp1a}}{\apprle} \la^p.
\end{align}
Next, we note that
\begin{align}\label{EQUATION3.22}
  \fiint_{\mcp^{\la}_{\rho_2}(z_0)}|a(z)-\inf_{w\in \mcp^{\la}_{\rho_2}(z_0)}a(w)|\left|\frac{u-(u)_{\mcp^{\lambda}_{\rho_2}}}{\lambda^{-1+\mu}\rho_2}\right|^q dz \apprle\la^{p-q} \fiint_{\mcp^{\la}_{\rho_2}(z_0)}\left|\frac{u-(u)_{\mcp^{\lambda}_{\rho_2}}}{\lambda^{-1+\mu}\rho_2}\right|^q dz.
\end{align}
Now we use \cref{g_n} with $\sigma=q, \xi=p, r=2$ and $\vartheta=\frac{p}{q}.$ 
\begin{align}\label{EQUATION3.23}
    \fiint_{\mcp^{\la}_{\rho_2}(z_0)}\left|\frac{u-(u)_{\mcp^{\lambda}_{\rho_2}}}{\lambda^{-1+\mu}\rho_2}\right|^q dz &\apprle \left(\fiint_{\mcp^{\la}_{\rho_2}(z_0)}\left|\frac{u-(u)_{\mcp^{\lambda}_{\rho_2}}}{\lambda^{-1+\mu}\rho_2}\right|^p+|\nabla u|^p dz\right)\left(\sup_{I^{\la, p}_{\rho_2}(t_0)}\fint_{B^{\la}_{\rho_2}(x_0)}\left|\frac{u-(u)_{\mcp^{\la}_{\rho_2}}}{\la^{-1+\mu}\rho_2}\right|^2 dx\right)^{\frac{q-p}{2}}\nonumber\\
    &\overset{\eqref{EQUATION3.20}}{\apprle} \la^p \left(\sup_{I^{\la, p}_{\rho_2}(t_0)}\fint_{B^{\la}_{\rho_2}(x_0)}\left|\frac{u-(u)_{\mcp^{\la}_{\rho_2}}}{\la^{-1+\mu}\rho_2}\right|^2 dx\right)^{\frac{q-p}{2}}.
\end{align}
In a similar way, we estimate
\begin{align*}
 \fiint_{\mcp^{\la}_{\rho_2}(z_0)}b(z) \left|\frac{u-(u)_{\mcp^{\lambda}_{\rho_2}}}{\lambda^{-1+\mu}\rho_2}\right|^s dz &\apprle \fiint_{\mcp^{\la}_{\rho_2}(z_0)}|b(z)-\inf_{w\in \mcp^{\la}_{\rho_2}(z_0)}b(w)|\left|\frac{u-(u)_{\mcp^{\lambda}_{\rho_2}}}{\lambda^{-1+\mu}\rho_2}\right|^s dz \\
    &+ \fiint_{\mcp^{\la}_{\rho_2}(z_0)}\inf_{w\in \mcp^{\la}_{\rho_2}(z_0)}b(w)\left|\frac{u-(u)_{\mcp^{\lambda}_{\rho_2}}}{\lambda^{-1+\mu}\rho_2}\right|^s dz.   
\end{align*}
To estimate the second term of the above inequality, we apply \ref{3_lem7.4} of \cref{p_intrinsic poincare 2} with $\theta =1.$ Indeed,
\begin{align}\label{EQQQQ3.32}
    \fiint_{\mcp^{\la}_{\rho_2}(z_0)}\inf_{w\in \mcp^{\la}_{\rho_2}(z_0)}b(w)\left|\frac{u-(u)_{\mcp^{\lambda}_{\rho_2}}}{\lambda^{-1+\mu}\rho_2}\right|^s dz &\apprle \fiint_{\mcp^{\la}_{\rho_2}(z_0)}b(z)|\nabla u|^s dz + \la^{s-p} b(z_0) \fiint_{\mcp^{\la}_{\rho_2}(z_0)}b(z)|\nabla u|^s dz\nonumber\\
 &+\la^{(2-p)q}\la^{(p-q)\frac{s}{q}}b(z_0)\left(\fiint_{\mcp^{\la}_{\rhob}(z_0)}a(z)|\nabla u|^q\, dz\right)^{\frac{s(q-1)}{q}}\nonumber\\
 &+ \varepsilon b(z_0)\la^s \nonumber\\
 &\overset{\cref{assmp1a}}{\apprle} \la^p.
\end{align}
We also have,
\begin{align*}
\fiint_{\mcp^{\la}_{\rho_2}(z_0)}|b(z)-\inf_{w\in \mcp^{\la}_{\rho_2}(z_0)}b(w)|\left|\frac{u-(u)_{\mcp^{\lambda}_{\rho_2}}}{\lambda^{-1+\mu}\rho_2}\right|^s dz \leq c \la^{p-s}\fiint_{\mcp^{\la}_{\rho_2}(z_0)} \left|\frac{u-(u)_{\mcp^{\lambda}_{\rho_2}}}{\lambda^{-1+\mu}\rho_2}\right|^s\,dz.  
\end{align*}
Again, we use \cref{g_n} with $\sigma=s, \xi=p, r=2$ and $\vartheta=\frac{p}{s}$ to get
\begin{align}\label{EQQQQ3.33}
\fiint_{\mcp^{\la}_{\rho_2}(z_0)}\left|\frac{u-(u)_{\mcp^{\lambda}_{\rho_2}}}{\lambda^{-1+\mu}\rho_2}\right|^s dz \apprle \la^p \left(\sup_{I^{\la, p}_{\rho_2}(t_0)}\fint_{B^{\la}_{\rho_2}(x_0)}\left|\frac{u-(u)_{\mcp^{\la}_{\rho_2}}}{\la^{-1+\mu}\rho_2}\right|^2 dx\right)^{\frac{s-p}{2}}.     
\end{align}
Hence combining the estimates \eqref{EQUATION3.20}, \eqref{EQUATION3.21}, \eqref{EQUATION3.22}, \eqref{EQUATION3.23}, \eqref{EQQQQ3.32} and \eqref{EQQQQ3.33}, we get
\begin{align*}
    I \apprle \left(\frac{\rho_2}{\rho_2-\rho_1}\right)^s \Bigg(\la^p&+ \la^{2p-q}\left(\sup_{I^{\la, p}_{\rho_2}(t_0)}\fint_{B^{\la}_{\rho_2}(x_0)}\left|\frac{u-(u)_{\mcp^{\la}_{\rho_2}}}{\la^{-1+\mu}\rho_2}\right|^2 dx\right)^{\frac{q-p}{2}}\\
    &+\la^{2p-s}\left(\sup_{I^{\la, p}_{\rho_2}(t_0)}\fint_{B^{\la}_{\rho_2}(x_0)}\left|\frac{u-(u)_{\mcp^{\la}_{\rho_2}}}{\la^{-1+\mu}\rho_2}\right|^2 dx\right)^{\frac{s-p}{2}}\Bigg).
\end{align*}

\noindent \textbf{Estimate for II:} We shall make use of \cref{g_n} with $s=p, \sigma=2, \vartheta=\frac{1}{2}$ and $r=2.$ It is easy to see
\begin{align*}
    -\frac{n}{2} < \frac{1}{2} \left(1-\frac{n}{p}\right)-\left(1-\frac{1}{2}\right)\frac{n}{2} \,\,\,\, \text{if and only if}\,\,\,\, p>\frac{2n}{n+2}.
\end{align*}
This gives
\begin{align*}
II &\leq c\la^{p-2}\left(\frac{\rho_2}{\rho_2-\rho_1}\right)^2\fint_{I^{\la, p}_{\rho_2}(t_0)}\left(\fint_{B^{\la}_{\rho_2}(x_0)}\left(\left|\frac{u-(u)_{\mcp^{\la}_{\rho_2}}}{\la^{-1+\mu}\rho_2}\right|^p+|\nabla u|^p\right) dx\right)^{\frac{1}{p}} dt \\
&\times\left(\fint_{B^{\la}_{\rho_2}(x_0)}\left|\frac{u-(u)_{\mcp^{\la}_{\rho_2}}}{\la^{-1+\mu}\rho_2}\right|^2 dx\right)^{\frac{1}{2}}dt\\
&\leq c\la^{p-2}\left(\frac{\rho_2}{\rho_2-\rho_1}\right)^2\left(\fiint_{\mcp^{\la}_{\rho_2}(z_0)}\left|\frac{u-(u)_{\mcp^{\la}_{\rho_2}}}{\la^{-1+\mu}\rho_2}\right|^p dz+ \fiint_{\mcp^{\la}_{\rho_2}(z_0)}|\nabla u|^p dz\right)^{\frac{1}{p}}\\
&\times\left(\sup_{I^{\la, p}_{\rho_2}(t_0)}\fint_{B^{\la}_{\rho_2}(x_0)}\left|\frac{u-(u)_{\mcp^{\la}_{\rho_2}}}{\la^{-1+\mu}\rho_2}\right|^2 dx\right)^{\frac{1}{2}}.
\end{align*}
Using \ref{1_lem7.4} of \cref{p_intrinsic poincare 2} and \textbf{p2} of \cref{assmp1b} we further have
\begin{align*}
    II \apprle\la^{p-2}\left(\frac{\rho_2}{\rho_2-\rho_1}\right)^2 \la \left(\sup_{I^{\la, p}_{\rho_2}(t_0)}\fint_{B^{\la}_{\rho_2}(x_0)}\left|\frac{u-(u)_{\mcp^{\la}_{\rho_2}}}{\la^{-1+\mu}\rho_2}\right|^2 dx\right)^{\frac{1}{2}}. 
\end{align*}
Therefore combining the above estimates and applying Young's inequality,  from \eqref{EQUATION7.16} we get
\begin{align*}
  &\sup_{I^{\la, p}_{\rho_1}(t_0)}\fint_{B^{\la}_{\rho_1}(x_0)}\left|\frac{u-(u)_{\mcp^{\la}_{\rho_1}}}{\la^{-1+\mu}\rho_1}\right|^2 dx\\
  &\leq c \left(\frac{\rho_2}{\rho_2-\rho_1}\right)^s  \la^2
+ c\left(\frac{\rho_2}{\rho_2-\rho_1}\right)^2 \la \left(\sup_{I^{\la, p}_{\rho_2}(t_0)}\fint_{B^{\la}_{\rho_2}(x_0)}\left|\frac{u-(u)_{\mcp^{\la}_{\rho_2}}}{\la^{-1+\mu}\rho_2}\right|^2 dx\right)^{\frac{1}{2}}\\
  &+c\left( \frac{\rho_2}{\rho_2-\rho_1}\right)^s \la^{p-q+2}\left(\sup_{I^{\la, p}_{\rho_2}(t_0)}\fint_{B^{\la}_{\rho_2}(x_0)}\left|\frac{u-(u)_{\mcp^{\la}_{\rho_2}}}{\la^{-1+\mu}\rho_2}\right|^2 dx\right)^{\frac{q-p}{2}}\\
  &+c\left( \frac{\rho_2}{\rho_2-\rho_1}\right)^s \la^{p-s+2}\left(\sup_{I^{\la, p}_{\rho_2}(t_0)}\fint_{B^{\la}_{\rho_2}(x_0)}\left|\frac{u-(u)_{\mcp^{\la}_{\rho_2}}}{\la^{-1+\mu}\rho_2}\right|^2 dx\right)^{\frac{s-p}{2}}\\
  &\leq \varepsilon \left(\sup_{I^{\la, p}_{\rho_2}(t_0)}\fint_{B^{\la}_{\rho_2}(x_0)}\left|\frac{u-(u)_{\mcp^{\la}_{\rho_2}}}{\la^{-1+\mu}\rho_2}\right|^2 dx\right) + c\Bigg[\left(\frac{\rho_2}{\rho_2-\rho_1}\right)^q+\left(\frac{\rho_2}{\rho_2-\rho_1}\right)^4+ \left(\frac{\rho_2}{\rho_2-\rho_1}\right)^{\frac{2s}{p-q+2}}\\
  &+ \left(\frac{\rho_2}{\rho_2-\rho_1}\right)^{\frac{2s}{p-s+2}}\Bigg]\la^2.
\end{align*}
Finally, using \cref{iter_lemma}, we get the desired estimate on $2\varrho.$
\end{proof}
In the next lemma, we estimate the first term on the right-hand side of \cref{scl_energy_p}.
\begin{lemma}\label{LEM5.2}
	Let $u$ be a  weak solution to \eqref{main_eqn} and \cref{scl_energy_p} holds with the \cref{assmp1a} in force. Then there exist constants $c=c(\textnormal{\texttt{data}})$ and $\theta_0\in (0,1)$, depending only on $n,p,q, s$ such that for any $\theta, \varepsilon \in (\theta_0, 1)$,  we have
	\begin{align*}
		\fiint_{\mcp^{\lambda}_{2\varrho}(z_0)} H\left(z, \left|\frac{u-(u)_{\mcp^{\la}_{2\varrho}}}{\lambda^{-1+\mu}2\varrho}\right|\right)\,dz
		&\leq c\lambda^{(1-\theta)p} \fiint_{\mcp^{\lambda}_{2\varrho}(z_0)}\left(|\nabla u|^{\theta p}+a(z)^{\theta}|\nabla u|^{\theta q}+b(z)^{\theta}|\nabla u|^{\theta s}\right)dz+ \varepsilon \la^p
	\end{align*}
 whenever $\mcp^{\la}_{4\varrho}(z_0)\subset \Omega_T.$
\end{lemma}
\begin{proof}
Let us first estimate the first term of $H(z, \cdot)$ on the left hand side. We shall make use of \cref{g_n} with $\sigma = p, \xi=\theta p, \vartheta=\theta$ and $r=2.$ It is easy to see that for $\theta \geq \frac{n}{n+2},$ the hypothesis of the \cref{g_n} satisfied. Indeed,
\begin{align*}
 -\frac{n}{p} \leq \theta \left(1-\frac{n}{\theta p}\right)-(1-\theta)\frac{n}{2} \,\,\, \text{if and only if}\,\,\, \theta \geq \frac{n}{n+2}.
\end{align*}            
Therefore an application of \cref{g_n} gives,
\begin{align*}
&\fiint_{\mcp^{\lambda}_{2\varrho}(z_0)} \left|\frac{u-(u)_{\mcp^{\la}_{2\varrho}}}{\lambda^{-1+\mu}2\varrho}\right|^p dz \\
&\apprle \fiint_{\mcp^{\lambda}_{2\varrho}(z_0)}\left(\left|\frac{u-(u)_{\mcp^{\la}_{2\varrho}}}{\lambda^{-1+\mu}2\varrho}\right|^{\theta p} + |\nabla u|^{\theta p} dz\right)\left(\sup_{I^{\lambda}_{2\varrho}(t_0)}\fint_{B^{\lambda}_{2\varrho}(x_0)}\left|\frac{u-(u)_{\mcp^{\la}_{2\varrho}}}{\lambda^{-1+\mu}2\varrho}\right|^2 dx\right)^{\frac{(1-\theta)p}{2}}\\
&\overset{\cref{p_intrinsic poincare 2}}{\apprle} \left(\fiint_{\mcp^{\lambda}_{2\varrho}(z_0)}|\nabla u|^{\theta p} dz+ \varepsilon \la^{\theta p}\right)\left(\sup_{I^{\lambda}_{2\varrho}(t_0)}\fint_{B^{\lambda}_{2\varrho}(x_0)}\left|\frac{u-(u)_{\mcp^{\la}_{2\varrho}}}{\lambda^{-1+\mu}2\varrho}\right|^2 dx\right)^{\frac{(1-\theta)p}{2}}\\
&\overset{\cref{LEMMA7.5}}{\apprle} \lambda^{(1-\theta)p} \left(\fiint_{\mcp^{\lambda}_{2\varrho}(z_0)}|\nabla u|^{\theta p} dz+ \varepsilon \la^{\theta p}\right).
\end{align*}
Thus, finally we get
\begin{align}\label{1st estimate}
\fiint_{\mcp^{\lambda}_{2\varrho}(z_0)} \left|\frac{u-(u)_{\mcp^{\la}_{2\varrho}}}{\lambda^{-1+\mu}2\varrho}\right|^p dz \apprle \la^{(1-\theta)p}\left(\fiint_{\mcp^{\lambda}_{2\varrho}(z_0)}|\nabla u|^{\theta p} dz\right)+ \varepsilon \la^p.
\end{align}
The second term of $H(z, \cdot)$ on the left hand side can be written as
\begin{align*}
    \fiint_{\mcp^{\la}_{2\varrho}(z_0)} a(z)\left|\frac{u-(u)_{\mcp^{\la}_{2\varrho}}}{\lambda^{-1+\mu}2\varrho}\right|^q dz &\leq \underbrace{\fiint_{\mcp^{\lambda}_{2\varrho}(z_0)} \inf_{w \in \mcp^{\lambda}_{2\varrho}(z_0)} a(w)\left|\frac{u-(u)_{\mcp^{\la}_{2\varrho}}}{\lambda^{-1+\mu}2\varrho}\right|^q dz}_{I}\\
    &+\underbrace{\fiint_{\mcp^{\la}_{2\varrho}(z_0)}c[a]_{\alpha}\max\left\{\left(\lambda^{-1+\mu}2\varrho\right)^{\alpha}, \left(\lambda^{\frac{2\mu-p}{2}}2\varrho\right)^{\alpha}\right\} \left|\frac{u-(u)_{\mcp^{\la}_{2\varrho}}}{\lambda^{-1+\mu}2\varrho}\right|^q dz}_{II}.
\end{align*}
\noindent\textbf{Estimate for I:} To get an estimate of $I,$ we use \cref{g_n} with $\sigma=q, \xi=\theta q, \vartheta=\theta$ and $r=2.$ One can easily check that the hypothesis of \cref{g_n} is satisfied. Indeed,
\begin{align*}
 -\frac{n}{q} \leq \theta \left(1-\frac{n}{\theta q}\right)-(1-\theta)\frac{n}{2}\,\,\,\,\, \text{if and only if}\,\,\,\, \theta \geq \frac{n}{2+n}.   
\end{align*} Now \cref{g_n} gives
\begin{align*}
    I &\leq c  \fiint_{\mcp^{\lambda}_{2\varrho}(z_0)} \left(\inf_{w \in \mcp^{\lambda}_{2\varrho}(z_0)} a(w)^{\theta}\left|\frac{u-(u)_{\mcp^{\la}_{2\varrho}}}{\lambda^{-1+\mu}2\varrho}\right|^{\theta q} + \inf_{w \in \mcp^{\lambda}_{2\varrho}(z_0)} a(w)^{\theta} |\nabla u|^{\theta q}\right)dz\\
 &\times \inf_{z \in \mcp^{\lambda}_{2\varrho}(z_0)} a(z)^{1-\theta} \left(\sup_{I^{\lambda}_{2\varrho}(t_0)}\fint_{B^{\lambda}_{2\varrho}(x_0)}\left|\frac{u-(u)_{\mcp^{\la}_{2\varrho}}}{\lambda^{-1+\mu}2\varrho}\right|^2 dx\right)^{\frac{(1-\theta)q}{2}}.
\end{align*}
Now to estimate the first term above, we use \cref{p_intrinsic poincare 2} \ref{2_lem7.4} to get
\begin{align}\label{EQUATION3.25}
 &\fiint_{\mcp^{\la}_{2\varrho}(z_0)}\inf_{w \in \mcp^{\lambda}_{2\varrho}(z_0)} a(w)^{\theta}\left|\frac{u-(u)_{\mcp^{\la}_{2\varrho}}}{\lambda^{-1+\mu}2\varrho}\right|^{\theta q} \apprle \fiint_{\mcp^{\la}_{2\varrho}(z_0)}\inf_{w \in \mcp^{\lambda}_{2\varrho}(z_0)} a(w)^{\theta} |\nabla u|^{\theta q}dz\nonumber \\
 &+ \la^{(q-p)\theta}\inf_{w \in \mcp^{\lambda}_{2\varrho}(z_0)} a(w)^{\theta}\fiint_{\mcp^{\la}_{2\varrho}(z_0)}a^{\theta}(z)|\nabla u|^{\theta q}dz\nonumber \\
 &+c\la^{(2-p)\theta q}\la^{(p-s)\frac{\theta q}{s}}\inf_{z \in \mcp^{\lambda}_{2\varrho}(z_0)} a(z)^{\theta}\left(\fiint_{\mcp^{\la}_{2\varrho}(z_0)}\inf_{z\in \mcp^{\la}_{2\varrho}(z_0)}b(z)^{\theta}|\nabla u|^{\theta s}\right)^{\frac{q(s-1)}{s}}\nonumber\\
 &+ \varepsilon \inf_{w \in \mcp^{\lambda}_{2\varrho}(z_0)} a(w)^{\theta} \la^{\theta q}. 
\end{align}
Therefore using \textbf{p1} and \textbf{p2} of \cref{assmp1a} in \eqref{EQUATION3.25}, we obtain
\begin{align}\label{EQUATION7.18}
  \fiint_{\mcp^{\la}_{2\varrho}(z_0)}\inf_{w \in \mcp^{\lambda}_{2\varrho}(z_0)} a(w)^{\theta}\left|\frac{u-(u)_{\mcp^{\la}_{2\varrho}}}{\lambda^{-1+\mu}2\varrho}\right|^{\theta q} \apprle \fiint_{\mcp^{\la}_{2\varrho}(z_0)}a^{\theta}(z)|\nabla u|^{\theta q}dz + \fiint_{\mcp^{\la}_{2\varrho}(z_0)}b(z)^{\theta}|\nabla u|^{\theta s}\,dz+\varepsilon \la^{\theta p}.  
\end{align}
The third term above can be estimated using \cref{LEMMA7.5} and \cref{assmp1a} as
\begin{align}\label{EQUATION7.19}
  \inf_{z \in \mcp^{\lambda}_{2\varrho}(z_0)} a(z)^{1-\theta} \left(\sup_{I^{\lambda}_{2\varrho}(t_0)}\fint_{B^{\lambda}_{2\varrho}(x_0)}\left|\frac{u-(u)_{\mcp^{\la}_{2\varrho}}}{\lambda^{-1+\mu}2\varrho}\right|^2 dx\right)^{\frac{(1-\theta)q}{2}} \leq c \la^{(p-q)(1-\theta)} \la^{(1-\theta)q}=\la^{(1-\theta)p}. 
\end{align}
Combining the estimates \eqref{EQUATION7.18}-\eqref{EQUATION7.19}, we obtain
\begin{align}\label{2nd estimate}
    I=&\fiint_{\mcp^{\lambda}_{2\varrho}(z_0)} \inf_{w \in \mcp^{\lambda}_{2\varrho}(z_0)} a(w)\left|\frac{u-(u)_{\mcp^{\la}_{2\varrho}}}{\lambda^{-1+\mu}2\varrho}\right|^q dz\nonumber \\
    &\leq c \la^{(1-\theta)p}\left[\fiint_{\mcp^{\la}_{2\varrho}(z_0)} a(z)^{\theta} |\nabla u|^{\theta q} dz+\fiint_{\mcp^{\la}_{2\varrho}(z_0)} b(z)^{\theta} |\nabla u|^{\theta s} dz\right]+ \varepsilon \la^{p}.
\end{align}
\noindent\textbf{Estimate for II:} As previous, using the definition of $\rhoa$ and $\rhob,$ we have
\begin{align*}
 [a]_{\alpha}\max\left\{\lambda^{\alpha(-1+\mu)} (2\varrho)^{\alpha}, \left(\lambda^{\frac{2\mu-p}{2}}2\varrho\right)^{\alpha}\right\} \leq c\la^{p-q}. 
\end{align*}
Now we use \cref{g_n} with $\sigma=q, \xi=\theta p, \vartheta=\frac{\theta p}{q}$ and $r=2.$ Note that
\begin{align*}
    -\frac{n}{q} \leq \frac{\theta p}{q}\left(1-\frac{n}{\theta p}\right)-\left(1-\frac{\theta p}{q}\right)\frac{n}{2}\, \,\,\,\, \text{if and only if}\,\,\,\,\, \theta \geq \frac{2p}{(n+2)q}.
\end{align*}

\vspace{.3cm}
\noindent Therefore, we find
\begin{align*}
\fiint_{\mcp^{\lambda}_{2\varrho}(z_0)} \left|\frac{u-(u)_{\mcp^{\la}_{2\varrho}}}{\lambda^{-1+\mu}2\varrho}\right|^p dz 
&\overset{\cref{g_n}}{\apprle} \fiint_{\mcp^{\lambda}_{2\varrho}(z_0)}\left(\left|\frac{u-(u)_{\mcp^{\la}_{2\varrho}}}{\lambda^{-1+\mu}2\varrho}\right|^{\theta p} + |\nabla u|^{\theta p} dz\right)\\
&\times\left(\sup_{I^{\lambda}_{2\varrho}(t_0)}\fint_{B^{\lambda}_{2\varrho}(x_0)}\left|\frac{u-(u)_{\mcp^{\la}_{2\varrho}}}{\lambda^{-1+\mu}2\varrho}\right|^2 dx\right)^{\frac{q-\theta p}{2}}\\
&\overset{\cref{p_intrinsic poincare 2}\,\, \text{and}\,\, \cref{LEMMA7.5}}{\apprle}\left(\fiint_{\mcp^{\la}_{2\varrho}(z_0)}|\nabla u|^{\theta p}dz + \varepsilon \la^{\theta p}\right)\la^{q-\theta p}.
\end{align*}

\vspace{.3cm}
\noindent Thus finally, we get
\begin{align}\label{3rd estimate}
    II=&\fiint_{\mcp^{\la}_{2\varrho}(z_0)}c[a]_{\alpha}\max\left\{\left(\lambda^{-1+\mu}2\varrho\right)^{\alpha}, \left(\lambda^{\frac{2\mu-p}{2}}2\varrho\right)^{\alpha}\right\} \left|\frac{u-(u)_{\mcp^{\la}_{2\varrho}}}{\lambda^{-1+\mu}2\varrho}\right|^q dz\nonumber\\ &\apprle \la^{p-q} \left(\fiint_{\mcp^{\la}_{2\varrho}(z_0)}|\nabla u|^{\theta p}dz + \varepsilon \la^{\theta p}\right) \la^{q-\theta p}
    =\la^{(1-\theta)p}\left(\fiint_{\mcp^{\la}_{2\varrho}(z_0)}|\nabla u|^{\theta p}dz\right) + \varepsilon \la^p.
\end{align}

\vspace{.4cm}
\noindent Similar analysis can be carried out for the third term in $H(z, \cdot)$ on the left hand side. Indeed,
\vspace{.3cm}
\begin{align*}
 \fiint_{\mcp^{\la}_{2\varrho}(z_0)} b(z)\left|\frac{u-(u)_{\mcp^{\la}_{2\varrho}}}{\lambda^{-1+\mu}2\varrho}\right|^s dz &\leq \underbrace{\fiint_{\mcp^{\lambda}_{2\varrho}(z_0)} \inf_{w \in \mcp^{\lambda}_{2\varrho}(z_0)} b(w)\left|\frac{u-(u)_{\mcp^{\la}_{2\varrho}}}{\lambda^{-1+\mu}2\varrho}\right|^s dz}_{III}\\
    &+\underbrace{\fiint_{\mcp^{\la}_{2\varrho}(z_0)}c[b]_{\beta}\max\left\{\left(\lambda^{-1+\mu}2\varrho\right)^{\beta}, \left(\lambda^{\frac{2\mu-p}{2}}2\varrho\right)^{\beta}\right\} \left|\frac{u-(u)_{\mcp^{\la}_{2\varrho}}}{\lambda^{-1+\mu}2\varrho}\right|^s dz}_{IV}.   
\end{align*}

\vspace{.4cm}
\noindent Again, using \cref{g_n} with $\sigma=s, \xi=\theta s, \vartheta=\theta$ and $r=2$ and \cref{p_intrinsic poincare 1} \ref{3_lem7.4}, we can estimate
\begin{align}\label{4th estimate}
    III \leq c \la^{(1-\theta)p}\left[\fiint_{\mcp^{\la}_{2\varrho}(z_0)} a(z)^{\theta} |\nabla u|^{\theta q} dz+\fiint_{\mcp^{\la}_{2\varrho}(z_0)} b(z)^{\theta} |\nabla u|^{\theta s} dz\right]+ \varepsilon \la^{p}
\end{align}

\vspace{.3cm}
\noindent On the other hand, using 
\begin{align*}
 [b]_{\beta}\max\left\{\lambda^{\alpha(-1+\mu)} (2\varrho)^{\beta}, \left(\lambda^{\frac{2\mu-p}{2}}2\varrho\right)^{\beta}\right\} \leq c\la^{p-s}   
\end{align*}
and \cref{g_n} with $\sigma=s, \xi=\theta p, \vartheta=\frac{\theta p}{s}$ and $r=2,$ we obtain
\begin{align}\label{5th estimate}
    IV \apprle \la^{(1-\theta)p}\left(\fiint_{\mcp^{\la}_{2\varrho}(z_0)}|\nabla u|^{\theta p}dz\right) + \varepsilon \la^p.
\end{align}
Combining the estimates \eqref{1st estimate}, \eqref{2nd estimate}, \eqref{3rd estimate}, \eqref{4th estimate} and \eqref{5th estimate}, we conclude the lemma.
\end{proof}
\begin{lemma}\label{LEM5.3}
Let $u$ be a weak solution to \eqref{main_eqn} and let \cref{scl_energy_p} hold with the \cref{assmp1a} in force. Then there exist constants $c=c(\textnormal{\texttt{data}})$ and $\theta_0 \in (0,1)$ depending on $n,p,q$ such that for any $\theta, \varepsilon \in (\theta_0,1),$ we have
\begin{align*}
\lambda^{p-2} \fiint_{\mcp^{\lambda}_{2\varrho}(z_0)}\left|\frac{u-(u)_{\mcp^{\la}_{2\varrho}}}{\lambda^{-1+\mu}2\varrho}\right|^2 dz \leq \varepsilon \lambda^p + c \left(\fiint_{\mcp^{\lambda}_{2\varrho}(z_0)}|\nabla u|^{\theta p}\right)^{\frac{1}{\theta}}    
\end{align*}
whenever $\mcp^{\la}_{4\varrho}(z_0)\subset \Omega_T.$
\end{lemma}
\begin{proof}
First we use \cref{g_n} with $\sigma =2, \xi=\theta p, \vartheta=\frac{\theta p}{2}$ and $r=2$ and an application \cref{p_intrinsic poincare 2} gives
\begin{align*} 
\fiint_{\mcp^{\lambda}_{2\varrho}(z_0)}\left|\frac{u-(u)_{\mcp^{\la}_{2\varrho}}}{\lambda^{-1+\mu}(2\varrho)}\right|^2 dz &\apprle \left(\fiint_{\mcp^{\lambda}_{2\varrho}(z_0)} |\nabla u|^{\theta p}dz+ \varepsilon \la^{\theta p}\right) \left(\sup_{I^{\lambda}_{2\varrho}(t_0)}\fint_{B^{\lambda}_{2\varrho}(x_0)}\left|\frac{u-(u)_{\mcp^{\la}_{2\varrho}}}{\lambda^{-1+\mu}2\varrho}\right|^2 dx\right)^{\frac{2-\theta p}{2}}\\
&\overset{\cref{LEMMA7.5}}{\apprle} \lambda^{2-\theta p}\left(\fiint_{Q^{\lambda}_{2\varrho}(z_0)} |\nabla u|^{\theta p}dz+ \varepsilon \la^{\theta p}\right)\\
&=\la^{2-\theta p}\left(\fiint_{Q^{\lambda}_{2\varrho}(z_0)} |\nabla u|^{\theta p}dz\right)+ \varepsilon \la^2.
\end{align*}
Therefore, we have
\begin{align*}
  \lambda^{p-2} \fiint_{\mcp^{\lambda}_{2\varrho}(z_0)}\left|\frac{u-(u)_{\mcp^{\la}_{2\varrho}}}{\lambda^{-1+\mu}2\varrho}\right|^2 dz &\apprle\lambda^{p(1-\theta)}\left(\fiint_{\mcp^{\lambda}_{2\varrho}(z_0)} |\nabla u|^{\theta p}dz\right)+ \varepsilon \la^p\\
  &\apprle\left[\varepsilon \lambda^p + c(\varepsilon) \left(\fiint_{Q^{\lambda}_{2\varrho}(z_0)} |\nabla u|^{\theta p}dz\right)^{\frac{1}{\theta}}\right]+ \varepsilon \la^p.
\end{align*}
We used Young's inequality with $(\frac{1}{1-\theta}, \frac{1}{\theta})$ to obtain the last estimate. This completes the proof.
\end{proof}
Now we are ready to prove the reverse H\"{o}lder inequality for $p$-phase.
\begin{lemma}\label{LEM5.4}
Let $u$ be a weak solution to \eqref{main_eqn}. Moreover, we assume that the energy estimate given in \cref{scl_energy_p} and the \cref{assmp1a} on the $p$-intrinsic cylinders hold. Then there exist constants $c=c(\textnormal{\texttt{data}})$ and $\theta_0 \in (0,1)$ depending on $n,p,q$ such that for any $\theta \in (\theta_0, 1),$ we have
\begin{align*}
\fiint_{\mcp^{\lambda}_{\varrho}(z_0)}H(z, |\nabla u|)\, dz \leq c \left(\fiint_{\mcp^{\lambda}_{2\varrho}(z_0)}H(z, |\nabla u|)^{\theta}\,dz\right)^{\frac{1}{\theta}}
\end{align*}
whenever $\mcp^{\la}_{4\varrho}(z_0)\subset \Omega_T.$
\end{lemma}
\begin{proof}
From the energy estimate \cref{scl_energy_p} and the definition of $\rhoa, \rhob,$ we have
\begin{align}\label{EQUATION7.22}
    \fiint_{\mcp^{\la}_{\varrho}(z_0)} H(z, |\nabla u|)\, dz \apprle \fiint_{\mcp^{\lambda}_{2\varrho}(z_0)} H\left(z, \left|\frac{u-(u)_{\mcp^{\lambda}_{2\varrho}}}{\lambda^{-1+\mu}2\varrho}\right|\right)\,dz
	 +\lambda^{p-2} \fiint_{\mcp^{\lambda}_{2\varrho}(z_0)}\left|\frac{u-(u)_{\mcp^{\lambda}_{2\varrho}}}{\lambda^{-1+\mu}2\varrho}\right|^2\,dz.
\end{align}
From \cref{LEM5.2}, we have
\begin{align*}
\fiint_{\mcp^{\lambda}_{2\varrho}(z_0)} H\Bigg(z, \left|\frac{u-(u)_{\mcp^{\lambda}_{2\varrho}}}{\lambda^{-1+\mu}2\varrho}\right|\Bigg)\,dz
\apprle \lambda^{(1-\theta)p} \fiint_{\mcp^{\lambda}_{2\varrho}(z_0)}\left(|\nabla u|^{\theta p}+a(z)^{\theta}|\nabla u|^{\theta q}+b(z)^{\theta}|\nabla u|^{\theta s}\right)\,dz+ \varepsilon \la^p\\
\overset{\text{Young's inequality}}{\leq} \varepsilon \la^p + c \left(\fiint_{\mcp^{\la}_{2\varrho}(z_0)}\left(|\nabla u|^{\theta p}+a(z)^{\theta}|\nabla u|^{\theta q}+b(z)|\nabla u|^{\theta s}\right)dz\right)^{\frac{1}{\theta}}\\
\leq \varepsilon \la^p + c \left(\fiint_{\mcp^{\la}_{2\varrho}(z_0)}\left(|\nabla u|^{p}+a(z)|\nabla u|^{q}+b(z)|\nabla u|^s\right)^{\theta}dz\right)^{\frac{1}{\theta}}.
\end{align*}
 On the other hand, from \cref{LEM5.3}, we obtain
\begin{align*}
 \lambda^{p-2} \fiint_{\mcp^{\lambda}_{2\varrho}(z_0)}\left|\frac{u-(u)_{\mcp^{\lambda}_{2\varrho}}}{\lambda^{-1+\mu}2\varrho}\right|^2\,dz &\leq \varepsilon \la^p + c \left(\fiint_{\mcp^{\lambda}_{2\varrho}(z_0)}|\nabla u|^{\theta p}\right)^{\frac{1}{\theta}}\\
 &\leq \varepsilon \la^p + c \left(\fiint_{\mcp^{\la}_{2\varrho}(z_0)}\left(|\nabla u|^{p}+a(z)|\nabla u|^{q}+b(z)|\nabla u|^{s}\right)^{\theta}dz\right)^{\frac{1}{\theta}}.
\end{align*}
Substituting the above estimates in \eqref{EQUATION7.22}, we get
\begin{align*}
   \fiint_{\mcp^{\la}_{\varrho}(z_0)} H\left(z, |\nabla u|\right)\, dz \leq \varepsilon \la^p + c \left(\fiint_{\mcp^{\la}_{2\varrho}(z_0)}H\left(z, |\nabla u|\right)^{\theta}dz\right)^{\frac{1}{\theta}}.
\end{align*}
The proof can be completed by absorbing $\varepsilon \la^p$ in the left-hand side, which is allowed because of \textbf{p3} of \cref{assmp1c}.
\end{proof}
\subsection{Reverse H\"{o}lder inequality for $(p,q)$-phase} Now we prove reverse H\"{o}lder inequality for $(p,q)$-phase. We start by defining the $(p,q)$-intrinsic cylinders. 
\begin{assumption}
In the $(p,q)$-phase, we assume that the following is satisfied:
\begin{enumerate}
    \item [\textbf{pq1.}] $K\lambda^p < a(z_0)\lambda^q,\,\,\, K\la^p\geq b(z_0)\la^s.$ \label{assmpq1a}\\
    \item [\textbf{pq2.}] $\fiint_{\mcq^{\lambda}_{\rho}(z_0)}(|\nabla u|^p+a(z)|\nabla u|^q+b(z)|\nabla u|^s) dz < \lambda^p+a(z_0)\la^q\,\,\, \text{holds for all} \,\,\,\rho \in (\rho_a, \rho_b].$ \label{assmpq1b}\\
    \item [\textbf{pq3.}] $\fiint_{\mcq^{\lambda}_{\rho_a}(z_0)}(|\nabla u|^p+a(z)|\nabla u|^q+b(z)|\nabla u|^s) dz =\lambda^p+a(z_0)\la^q.$ \label{assmpq1c}\\
    \item [\textbf{pq4.}] $\frac{a(z_0)}{2} \leq a(z) \leq 2a(z_0)\,\,\, \text{holds for every}\,\,\, z \in \mcq^{\lambda}_{\rhob}(z_0).$ \label{assmpq1d}
\end{enumerate}
\end{assumption}
The proof of the following lemma goes similarly as \cref{p_intrinsic poincare 1} and \cref{p_intrinsic poincare 2}. We provide detailed proof of the lemma as these estimates hold for the $(p,s)$ and $(p,q,s)$-phases also.
\begin{lemma}\label{pq_intrinsic poincare}
Let $u$ be a weak solution to \eqref{main_eqn} and the \cref{assmpq1a} is in force. Then for any $\theta\in \left(\max\left\{\frac{s-1}{p}, \frac{1}{p}\right\}, 1\right]$ and $\varepsilon \in (0,1),$ there exists a constant $c=c(\textnormal{\texttt{data}})$ such that the following estimates 
\begin{enumerate}[label=(\roman*),series=theoremconditions]
\item \label{1_lem7.10} 
\begin{align*}
\fiint_{\mcq^\la_{\rhob}(z_0)}\left|\frac{u-(u)_{\mcq^{\la}_{\rhob}(z_0)}}{\la^{-1+\mu}\rhob}\right|^{\theta p} dz \leq c \fiint_{\mcq^{\la}_{\rhob}(z_0)}\left|\nabla u\right|^{\theta p}dz +\varepsilon \la^{\theta p},
\end{align*}
\item \label{2_lem7.10}
\begin{align*}
    \fiint_{\mcq^{\la}_{\rhob}(z_0)}\left|\frac{u-(u)_{\mcq^{\la}_{\rhob}(z_0)}}{\la^{-1+\mu}\rhob}\right|^{\theta q} dz \leq c \fiint_{\mcq^{\la}_{\rhob}(z_0)}\left|\nabla u\right|^{\theta q}dz+ \varepsilon \la^{\theta q},
\end{align*}
\item \label{3_lem7.10}
\begin{align*}
    \fiint_{\mcq^{\la}_{\rhob}(z_0)}\left|\frac{u-(u)_{\mcq^{\la}_{\rhob}(z_0)}}{\la^{-1+\mu}\rhob}\right|^{\theta s} dz \leq c \fiint_{\mcq^{\la}_{\rhob}(z_0)}\left|\nabla u\right|^{\theta s}dz+ \varepsilon \la^{\theta s},
\end{align*}
\end{enumerate}
hold whenever $\mcq^{\la}_{\rhob}(z_0) \subset \Omega_T.$
\end{lemma}
\begin{proof}
    Using \cref{parabolic poincare for double phase} with $m=p$ and $(p, q)$-intrinsic cylinders, we get
    \begin{align*}
        \fiint_{\mcq^{\la}_{\rhob}(z_0)}\left|\frac{u-(u)_{\mcq^{\la}_{\rhob}}}{\la^{-1+\mu}\rhob}\right|^{\theta p} \apprle \fiint_{\mcq^{\la}_{\rhob}(z_0)}|\nabla u|^{\theta p}dz + \underbrace{\left(\frac{\la^2}{\la^p+a(z_0)\la^q}\fiint_{\mcq^{\la}_{\rhob}(z_0)}\tilde{H}(z, \nabla u)\, dz\right)^{\theta p} }_{I}.
    \end{align*}

\vspace{.3cm}   
\noindent\textbf{Case $p \geq 2$ :} The first term in $I$ can be estimated as
\begin{align*}
 &\left(\frac{\la^2}{\la^p+a(z_0)\la^q}\fiint_{\mcq^{\la}_{\rhob}(z_0)}|\nabla u|^{p-1} dz \right)^{\theta p}\\ &\apprle \left(\frac{\la^2}{\la^p+a(z_0)\la^q}\right)^{\theta p} \left(\fiint_{\mcq^{\la}_{\rhob}(z_0)}|\nabla u|^{\theta p} dz\right)^{p-2} \left(\fiint_{\mcq^{\la}_{\rhob}(z_0)}|\nabla u|^{\theta p} dz\right).
\end{align*}

\vspace{.3cm}
\noindent Next, using \textbf{pq2} of \cref{assmpq1b} we note that
\begin{align*}
 &\left(\frac{\la^2}{\la^p+a(z_0)\la^q}\right)^{\theta p} \left(\fiint_{\mcq^{\la}_{\rhob}(z_0)}|\nabla u|^{\theta p} dz\right)^{p-2}\\
 &\overset{\text{H\"{o}lder's inequality}}{\leq} \left(\frac{\la^2}{\la^p+a(z_0)\la^q}\right)^{\theta p} \left(\fiint_{\mcq^{\la}_{\rhob}(z_0)}|\nabla u|^{p} dz\right)^{(p-2)\theta} \\
 &\leq \left(\frac{\la^2}{\la^p+a(z_0)\la^q}\right)^{\theta p} \left(\la^p+a(z_0)\la^q\right)^{(p-2)\theta} \leq 1.
\end{align*}

\vspace{.3cm}
\noindent Now using \textbf{pq4} of \cref{assmpq1d} in the second term of $I,$ we have
\begin{align*}
 &\left(\frac{\la^2}{\la^p+a(z_0)\la^q}\fiint_{\mcq^{\la}_{\rhob}(z_0)}a(z)|\nabla u|^{q-1} dz \right)^{\theta p}\\
 &\apprle \left(\frac{\la^2 a(z_0)}{\la^p+a(z_0)\la^q}\right)^{\theta p}\left(\fiint_{\mcq^{\la}_{\rhob}(z_0)}|\nabla u|^{q-1} dz\right)^{\theta p}\\
 & \apprle \la^{(2-q)\theta p}\left(\fiint_{\mcq^{\la}_{\rhob}(z_0)}|\nabla u|^{q}dz\right)^{\frac{(q-2)\theta p}{q}} \left(\fiint_{\mcq^{\la}_{\rhob}(z_0)}|\nabla u|^{\theta p}dz\right) \\
 &\apprle \la^{(2-q)\theta p} \la^{(q-2)\theta p}\left(\fiint_{\mcq^{\la}_{\rhob}(z_0)}|\nabla u|^{\theta p}dz\right).
\end{align*}
Note that in the last inequality above, we used
\begin{align}\label{q-bound}
\fiint_{\mcq^{\la}_{\rhob}(z_0)}|\nabla u|^q dz \apprle \la^q
\end{align}
which can be deduced from {\bf pq2} and {\bf pq4} of \cref{assmpq1b}. 
To estimate the third term in $I,$ we use the estimate done in {\bf Step 1} of \cref{p_intrinsic poincare 1}. In particular, we using \eqref{EQQQ3.7}, we get
\begin{align*}  
\left(\frac{\la^2}{\la^p+a(z_0)\la^q}\fiint_{\mcq^{\la}_{\rhob}(z_0)}b(z)|\nabla u|^{s-1}\right)^{\theta p}&\leq \left(\frac{\la^2}{\la^p+a(z_0)\la^q}\right)^{\theta p}\left(\fiint_{\mcq^{\la}_{\rhob}(z_0)}|\nabla u|^{\theta p}\, dz\right)^{p-1}\\
&+c\left(\frac{\la^2}{\la^p+a(z_0)\la^q}\right)^{\theta p}\la^{(p-s)\theta p}\left(\fiint_{\mcq^{\la}_{\rhob}(z_0)}|\nabla u|^{\theta p}\,dz\right)^{s-1}
\end{align*}
Note that the first term of the above inequality can be estimated as previous and we get
\begin{align*}
 \left(\frac{\la^2}{\la^p+a(z_0)\la^q}\right)^{\theta p}\left(\fiint_{\mcq^{\la}_{\rhob}(z_0)}|\nabla u|^{\theta p}\, dz\right)^{p-1} \leq \left(\fiint_{\mcq^{\la}_{\rhob}(z_0)}|\nabla u|^{\theta p}\,dz\right)   
\end{align*}

\noindent The estimation of the second term is also similar. Indeed,
\begin{align*}
 &\left(\frac{\la^2}{\la^p+a(z_0)\la^q}\right)^{\theta p}\la^{(p-s)\theta p}\left(\fiint_{\mcq^{\la}_{\rhob}(z_0)}|\nabla u|^{\theta p}\,dz\right)^{s-1}\\
 &\leq\left(\frac{\la^2}{\la^p+a(z_0)\la^q}\right)^{\theta p}\la^{(p-s)\theta p} \left(\fiint_{\mcq^{\la}_{\rhob}(z_0)}|\nabla u|^{p}\, dz\right)^{(s-2)\theta}\left(\fiint_{\mcq^{\la}_{\rhob}(z_0)}|\nabla u|^{\theta p}\, dz\right)\\
 &\leq \la^{(2-p)\theta p}\la^{(p-s)\theta p}\left(\fiint_{\mcq^{\la}_{\rhob}(z_0)}|\nabla u|^q\,dz\right)^{\frac{\theta p(s-2)}{q}}\left(\fiint_{\mcq^{\la}_{\rhob}(z_0)}|\nabla u|^{\theta p}\, dz\right)\\
 &\overset{\eqref{q-bound}}{\leq} \la^{(2-p)\theta p}\la^{(p-s)\theta p}\la^{(s-2)\theta p}\leq \left(\fiint_{\mcq^{\la}_{\rhob}(z_0)}|\nabla u|^{\theta p}\, dz\right).
\end{align*}

\noindent Therefore, combining the estimates above we obtain
\begin{align*}
    I \apprle \left(\fiint_{\mcq^{\la}_{\rhob}(z_0)}|\nabla u|^{\theta p}dz\right).
\end{align*}
Hence, in this case, we obtain
\begin{align*}
\fiint_{\mcq^\la_{\rhob}(z_0)}\left|\frac{u-(u)_{\mcq^{\la}_{\rhob}(z_0)}}{\la^{-1+\mu}\rhob}\right|^{\theta p} dz \apprle \left(\fiint_{\mcq^{\la}_{\rhob}(z_0)}|\nabla u|^{\theta p}dz\right). 
\end{align*}

\vspace{.3cm}
\noindent \textbf{Case $p<2$ :} Note that $I$ can also be estimated as
\begin{align*}
    I &\leq \Bigg(\frac{\la^2}{\la^p+a(z_0)\la^q}\fiint_{\mcq^{\la}_{\rhob}(z_0)}|\nabla u|^{p-1}dz+ \frac{\la^2 a(z_0)}{\la^p+a(z_0)\la^q}\fiint_{\mcq^{\la}_{\rhob}(z_0)}|\nabla u|^{q-1}dz\\
    &+\frac{\la^2}{\la^p+a(z_0)\la^q}\fiint_{\mcq^{\la}_{\rhob}(z_0)}b(z)|\nabla u|^{s-1}\,dz\Bigg)^{\theta p}\\
    &\overset{\eqref{EQQQ3.7}}{\apprle} \la^{(2-p)\theta p}\left(\fiint_{\mcq^{\la}_{\rhob}(z_0)}|\nabla u|^{p-1} dz\right)^{\theta p}+ \la^{(2-q)\theta p}\left(\fiint_{\mcq^{\la}_{\rhob}(z_0)}|\nabla u|^{q-1}dz\right)^{\theta p}\\
    &+\la^{(2-s)\theta p}\left(\fiint_{\mcq^{\la}_{\rhob}(z_0)}|\nabla u|^{s-1}\right)^{\theta p}.
\end{align*}
Now as in \cref{p_intrinsic poincare 2}, we can use Young's inequality with $\left(\frac{1}{2-p}, \frac{1}{p-1}\right),$ $\left(\frac{1}{2-q}, \frac{1}{q-1}\right)$ and $\left(\frac{1}{2-s}, \frac{1}{s-1}\right)$ to get the desired estimate \ref{1_lem7.10}. Note that, in the above estimate, we assumed $\frac{2n}{n+2}<p\leq q \leq s\leq 2.$ The other possibilities, for instance, $q\geq 2$ or $s\geq 2$ can be handled as in the case $p\geq 2.$

\vspace{.5cm}
Now we shall prove \ref{2_lem7.10}. Again, we use \cref{parabolic poincare for double phase} with $m=q$ and $(p,q)$-intrinsic cylinders.
\begin{align*}
    \fiint_{\mcq^{\la}_{\rhob}(z_0)}\left|\frac{u-(u)_{\mcq^{\la}_{\rhob}}}{\la^{-1+\mu}\rhob}\right|^{\theta q} \leq c \fiint_{\mcq^{\la}_{\rhob}(z_0)}|\nabla u|^{\theta q} dz + c \underbrace{\left(\frac{\la^2}{\la^p+a(z_0)\la^q} \fiint_{\mcq^{\la}_{\rhob}(z_0)}\tilde{H}(z, \nabla u)\right)^{\theta q}}_{I}.
\end{align*}
We estimate $I$ from the above inequality by using \eqref{EQQQ3.7} and {\bf pq4} of \cref{assmpq1d}.
\begin{align}\label{EQUATION3.31}
    I &\apprle \la^{(2-p)\theta q} \left(\fiint_{\mcq^{\la}_{\rhob}(z_0)}|\nabla u|^{\theta q} dz\right)^{p-1}+ \la^{(2-q)\theta q}\left(\fiint_{\mcq^{\la}_{\rhob}(z_0)}|\nabla u|^{\theta q} dz\right)^{q-1}\nonumber\\
    &\left(\frac{\la^2}{\la^p+a(z_0)\la^q}\right)^{\theta q}\la^{(p-s)\theta q}\left(\fiint_{\mcq^{\la}_{\rhob}(z_0)}|\nabla u|^{s-1}\, dz\right)^{\theta q}+ \left(\frac{\la^2}{\la^p+a(z_0)\la^q}\right)^{\theta q}\left(\fiint_{\mcq^{\la}_{\rhob}(z_0)}|\nabla u|^{\theta p}\, dz\right)^{\frac{(p-1)q}{p}}
\end{align}

\vspace{.3cm}
\noindent \textbf{Case $p \geq 2$ :} In this case, the first term of \eqref{EQUATION3.31} can be estimated as
\begin{align*}
 \la^{(2-p)\theta q} \left(\fiint_{\mcq^{\la}_{\rhob}(z_0)}|\nabla u|^{\theta q} dz\right)^{p-1}&=\la^{(2-p)\theta q} \left(\fiint_{\mcq^{\la}_{\rhob}(z_0)}|\nabla u|^{\theta q} dz\right)^{p-2} \left(\fiint_{\mcq^{\la}_{\rhob}(z_0)}|\nabla u|^{\theta q} dz\right)\\
 &\leq c \la^{(2-p)\theta q}\left(\fiint_{\mcq^{\la}_{\rhob}(z_0)}|\nabla u|^{q} dz\right)^{\theta(p-2)}\left(\fiint_{\mcq^{\la}_{\rhob}(z_0)}|\nabla u|^{\theta q} dz\right)\\
 &\leq  c\la^{(2-p)\theta q} \la^{q\theta (p-2)}\left(\fiint_{\mcq^{\la}_{\rhob}(z_0)}|\nabla u|^{\theta q} dz\right).
\end{align*}
Similarly, the second term of \eqref{EQUATION3.31} can be estimated as
\begin{align*}
  \la^{(2-q)\theta q}\left(\fiint_{\mcq^{\la}_{\rhob}(z_0)}|\nabla u|^{\theta q} dz\right)^{q-1}&=\la^{(2-q)\theta q}\left(\fiint_{\mcq^{\la}_{\rhob}(z_0)}|\nabla u|^{\theta q} dz\right)^{q-2}\left(\fiint_{\mcq^{\la}_{\rhob}(z_0)}|\nabla u|^{\theta q} dz\right)\\
  &\leq c \la^{(2-q)\theta q}\left(\fiint_{\mcq^{\la}_{\rhob}(z_0)}|\nabla u|^{q} dz\right)^{\theta(q-2)}\left(\fiint_{\mcq^{\la}_{\rhob}(z_0)}|\nabla u|^{\theta q} dz\right)\\
  &\leq c \la^{(2-q)\theta q}\la^{(q-2)\theta q}\left(\fiint_{\mcq^{\la}_{\rhob}(z_0)}|\nabla u|^{\theta q} dz\right).
\end{align*}
Now we estimate the third term in \eqref{EQUATION3.31}.
\begin{align*}  
&\left(\frac{\la^2}{\la^p+a(z_0)\la^q}\right)^{\theta q}\la^{(p-s)\theta q}\left(\fiint_{\mcq^{\la}_{\rhob}(z_0)}|\nabla u|^{s-1}\, dz\right)^{\theta q}\\
&\leq \left(\frac{\la^2}{\la^p+a(z_0)\la^q}\right)^{\theta q}\la^{(p-s)\theta q} \left(\fiint_{\mcq^{\la}_{\rhob}(z_0)}|\nabla u|^{q}\,dz\right)^{\theta(s-2)}\left(\fiint_{\mcq^{\la}_{\rhob}(z_0)}|\nabla u|^{\theta q}\,dz\right)\\
&\overset{\eqref{q-bound}}{\apprle} \left(\frac{\la^2}{\la^p+a(z_0)\la^q}\right)^{\theta q}\la^{(p-s)\theta q} \la^{\theta q(s-2)}\left(\fiint_{\mcq^{\la}_{\rhob}(z_0)}|\nabla u|^{\theta q}\,dz\right) \apprle \left(\fiint_{\mcq^{\la}_{\rhob}(z_0)}|\nabla u|^{\theta q}\,dz\right) 
\end{align*}

\vspace{.3cm}
\noindent It remains to estimate the fourth term of \eqref{EQUATION3.31}. Although the the estimate is similar, we present it here.
\begin{align*}
&\left(\frac{\la^2}{\la^p+a(z_0)\la^q}\right)^{\theta q}\left(\fiint_{\mcq^{\la}_{\rhob}(z_0)}|\nabla u|^{\theta p}\, dz\right)^{\frac{(p-1)q}{p}}\\
& \leq \left(\frac{\la^2}{\la^p+a(z_0)\la^q}\right)^{\theta q} \left(\fiint_{\mcq^{\la}_{\rhob}(z_0)}|\nabla u|^p\,dz\right)^{ \frac{\theta q(p-2)}{p}}\left(\fiint_{\mcq^{\la}_{\rhob}(z_0)}|\nabla u|^{\theta p}\right)^{\frac{q}{p}}\\
&\leq \left(\frac{\la^2}{\la^p+a(z_0)\la^q}\right)^{\theta q}(\la^p+a(z_0)\la^q)^{\frac{\theta q(p-2)}{p}}\left(\fiint_{\mcq^{\la}_{\rhob}(z_0)}|\nabla u|^{\theta q}\right)\\
&=\left(\frac{\la^2}{(\la^p+a(z_0)\la^q)^{\frac{2}{p}}}\right)^{\theta q}\left(\fiint_{\mcq^{\la}_{\rhob}(z_0)}|\nabla u|^{\theta q}\, dz\right)\leq \left(\fiint_{\mcq^{\la}_{\rhob}(z_0)}|\nabla u|^{\theta q}\, dz\right).
\end{align*}
Therefore, combining all the estimates above, we finally get
\begin{align*}
    I \apprle \left(\fiint_{\mcq^{\la}_{\rhob}(z_0)}|\nabla u|^{\theta q}\, dz\right).
\end{align*}

\vspace{.3cm}
\noindent \textbf{Case $p<2$ :} We note that, using H\"{o}lders inequality and $\la^p < \la^p +a(z_0)\la^q,$ \eqref{EQUATION3.31} can be written as
\begin{align*}
&\la^{(2-p)\theta q} \left(\fiint_{\mcq^{\la}_{\rhob}(z_0)}|\nabla u|^{\theta q} dz\right)^{p-1}+ \la^{(2-q)\theta q}\left(\fiint_{\mcq^{\la}_{\rhob}(z_0)}|\nabla u|^{\theta q} dz\right)^{q-1}\nonumber\\
    &\la^{(2-p)\theta q}\la^{(p-s)\theta q}\left(\fiint_{\mcq^{\la}_{\rhob}(z_0)}|\nabla u|^{\theta q}\, dz\right)^{s-1}+ \la^{(2-p)\theta q}\left(\fiint_{\mcq^{\la}_{\rhob}(z_0)}|\nabla u|^{\theta q}\, dz\right)^{p-1}.    
\end{align*}
Now, we use Young's inequality as in \cref{p_intrinsic poincare 2} to get \ref{2_lem7.10}.

\vspace{.3cm}
\noindent Our next aim is to prove \ref{3_lem7.10}. Using \cref{parabolic poincare for double phase} with $m=s$ and $(p, q)$-intrinsic cylinders, we get
    \begin{align*}
        \fiint_{\mcq^{\la}_{\rhob}(z_0)}\left|\frac{u-(u)_{\mcq^{\la}_{\rhob}}}{\la^{-1+\mu}\rhob}\right|^{\theta s} \apprle \fiint_{\mcq^{\la}_{\rhob}(z_0)}|\nabla u|^{\theta s}dz + \underbrace{\left(\frac{\la^2}{\la^p+a(z_0)\la^q}\fiint_{\mcq^{\la}_{\rhob}(z_0)}\tilde{H}(z, \nabla u)\, dz\right)^{\theta s} }_{I}.
    \end{align*}

\vspace{.3cm}    
\noindent Using \eqref{EQQQ3.7} and {\bf pq4} of \cref{assmpq1d}, we can estimate $I:$
 \begin{align*}
     I &\apprle \left(\frac{\la^2}{\la^p+a(z_0)\la^q}\right)^{\theta s} \left(\fiint_{\mcq^{\la}_{\rhob}(z_0)}|\nabla u|^{p-1}\, dz\right)^{\theta s} + \left(\frac{2 \la^2 a(z_0)}{\la^p+a(z_0)\la^q}\right)^{\theta s}\left(\fiint_{\mcq^{\la}_{\rhob}(z_0)}|\nabla u|^{q-1}\,dz\right)^{\theta s}\\
     &+\left(\frac{\la^2}{\la^p+a(z_0)\la^q}\right)^{\theta s}\left(\fiint_{\mcq^{\la}_{\rhob}(z_0)}|\nabla u|^{\theta p}\,dz\right)^{\frac{s(p-1)}{p}}+ \left(\frac{\la^2}{\la^p+a(z_0)\la^q}\right)^{\theta s} \la^{(p-s)\theta s}\left(\fiint_{\mcq^{\la}_{\rhob}(z_0)}|\nabla u|^{s-1}\,dz\right)^{\theta s}
 \end{align*}

\vspace{.5cm}
\noindent\textbf{Case $p \geq 2$ :} We will estimate each of the above four terms carefully. Let us start with the first term.
\begin{align*}
&\left(\frac{\la^2}{\la^p+a(z_0)\la^q}\right)^{\theta s} \left(\fiint_{\mcq^{\la}_{\rhob}(z_0)}|\nabla u|^{p-1}\, dz\right)^{\theta s} \\
&\leq \left(\frac{\la^2}{\la^p+a(z_0)\la^q}\right)^{\theta s} \left(\fiint_{\mcq^{\la}_{\rhob}(z_0)}|\nabla u|^p\right)^{\frac{\theta s(p-2)}{p}}\left(\fiint_{\mcq^{\la}_{\rhob}(z_0)}|\nabla u|^{\theta p}\,dz\right)^{\frac{s}{p}}\\
&\overset{\cref{assmpq1b}}{\leq}\left(\frac{\la^2}{\la^p+a(z_0)\la^q}\right)^{\theta s} (\la^p+a(z_0)\la^q)^{\frac{\theta s(p-2)}{p}}\left(\fiint_{\mcq^{\la}_{\rhob}(z_0)}|\nabla u|^{\theta s}\,dz\right)\\
&\leq \left(\fiint_{\mcq^{\la}_{\rhob}(z_0)}|\nabla u|^{\theta s}\,dz\right).
\end{align*}
Now we estimate the second term.
\begin{align*}
&\left(\frac{2 \la^2 a(z_0)}{\la^p+a(z_0)\la^q}\right)^{\theta s}\left(\fiint_{\mcq^{\la}_{\rhob}(z_0)}|\nabla u|^{q-1}\,dz\right)^{\theta s}\\
&\leq \left(\frac{2 \la^2 a(z_0)}{\la^p+a(z_0)\la^q}\right)^{\theta s} \left(\fiint_{\mcq^{\la}_{\rhob}(z_0)}|\nabla u|^{q}\,dz\right)^{\frac{\theta s(q-2)}{q}}\left(\fiint_{\mcq^{\la}_{\rhob}(z_0)}|\nabla u|^{\theta q}\, dz\right)^{\frac{s}{q}}\\
&\overset{\eqref{q-bound}}{\leq} \la^{(2-q)\theta s} \la^{\theta s(q-2)}\left(\fiint_{\mcq^{\la}_{\rhob}(z_0)}|\nabla u|^{\theta s}\,dz\right).
\end{align*}

\vspace{.3cm}
\noindent The estimate on the third term similar and rather easy.
\begin{align*}
&\left(\frac{\la^2}{\la^p+a(z_0)\la^q}\right)^{\theta s}\left(\fiint_{\mcq^{\la}_{\rhob}(z_0)}|\nabla u|^{\theta p}\,dz\right)^{\frac{s(p-1)}{p}}\\
&\leq\left(\frac{\la^2}{\la^p+a(z_0)\la^q}\right)^{\theta s} \left(\fiint_{\mcq^{\la}_{\rhob}(z_0)}|\nabla u|^p\,dz\right)^{\frac{\theta s(p-2)}{p}}\left(\fiint_{\mcq^{\la}_{\rhob}(z_0)}|\nabla u|^{\theta p}\,dz\right)^{\frac{s}{p}}\\
&\overset{\cref{assmpq1b}}{\leq}\left(\frac{\la^2}{\la^p+a(z_0)\la^q}\right)^{\theta s} (\la^p+a(z_0)\la^q)^{\frac{\theta s(p-2)}{p}}\left(\fiint_{\mcq^{\la}_{\rhob}(z_0)}|\nabla u|^{\theta s}\,dz\right)\\
&\leq \left(\fiint_{\mcq^{\la}_{\rhob}(z_0)}|\nabla u|^{\theta s}\,dz\right).
\end{align*}

\vspace{.3cm}
\noindent We estimate the fourth term below.
\begin{align*}
&\left(\frac{\la^2}{\la^p+a(z_0)\la^q}\right)^{\theta s} \la^{(p-s)\theta s}\left(\fiint_{\mcq^{\la}_{\rhob}(z_0)}|\nabla u|^{s-1}\,dz\right)^{\theta s} \\
&\leq \left(\frac{\la^2}{\la^p+a(z_0)\la^q}\right)^{\theta s} \la^{(p-s)\theta s} \left(\fiint_{\mcq^{\la}_{\rhob}(z_0)}|\nabla u|^q\,dz\right)^{\frac{\theta s(s-2)}{q}}\left(\fiint_{\mcq^{\la}_{\rhob}(z_0)}|\nabla u|^{\theta q}\,dz\right)^{\frac{s}{q}}\\
&\overset{\eqref{q-bound}}{\leq} \la^{(2-p)\theta s}\la^{(p-s)\theta s}\la^{(s-2)\theta s}\left(\fiint_{\mcq^{\la}_{\rhob}(z_0)}|\nabla u|^{\theta s}\,dz\right)
\end{align*}

\vspace{.3cm}
\noindent Therefore, combining all the estimates above, we conclude for $p\geq 2$ case that,
\begin{align*}
    I \apprle \left(\fiint_{\mcq^{\la}_{\rhob}(z_0)}|\nabla u|^{\theta s}\,dz\right).
\end{align*}

\vspace{.5cm}
\noindent{\bf Case $p< 2:$} We notice that, using H\"{o}lder's inequality and $\la^p< \la^p+a(z_0)\la^q,$ $I$ can be estimates as
\begin{align*}
    I \leq 2 \la^{(2-p)\theta s}\left(\fiint_{\mcq^{\la}_{\rhob}(z_0)}|\nabla u|^{\theta s}\right)^{p-1}&+ \la^{(2-q)\theta s}\left(\fiint_{\mcq^{\la}_{\rhob}(z_0)}|\nabla u|^{\theta s}\,dz\right)^{q-1}\\
    &+ \la^{(2-s)\theta s}\left(\fiint_{\mcq^{\la}_{\rhob}(z_0)}|\nabla u|^{\theta s}\,dz\right)^{s-1}.
\end{align*}
Now using the young's inequality with the pairs $\left(\frac{1}{2-p}, \frac{1}{p-1}\right),\left(\frac{1}{2-q}, \frac{1}{q-1}\right), \left(\frac{1}{2-s}, \frac{1}{s-1}\right),$ we obtain
\begin{align*}
    I \apprle \left(\fiint_{\mcq^{\la}_{\rhob}(z_0)}|\nabla u|^{\theta s}\,dz\right)+ \varepsilon \la^{\theta s}.
\end{align*}
This completes the proof.
\end{proof}
Next, we state energy estimate \eqref{general caccipoli} for the $(p,q)$-phase. Let us consider the following cutoff functions.

\vspace{.5cm}
\noindent \textbf{Cut-off functions for $(p,q)$-phase.} On the other hand, for $(p,q)$-phase we use
\begin{align*}
&\eta:=\eta(x) \in C_c^{\infty}(B_{\la^{-1+\mu}\rhob}(x_0)), \quad  \eta \equiv 1 \,\,\text{on}\, B_{\la^{-1+\mu}\rhoa}(x_0), \quad 0\leq \eta \leq 1 \,\,\,\,\,\text{and}\,\,\,\,\,\, |\nabla \eta| \apprle \frac{1}{\la^{-1+\mu}(\rhob-\rhoa)}\\
&\zeta \in C^{\infty}_c\left(t_0-\frac{\la^{2\mu}}{\la^p+a(z_0)\la^q}(\rhob-h_0)^2, t_0+\frac{\la^{2\mu}}{\la^p+a(z_0)\la^q}(\rhob-h_0)^2\right),\,\,\,\,\, 0\leq \zeta \leq 1,\\
&\zeta \equiv 1 \,\,\text{on}\,\,\left(t_0-\tfrac{\la^{2\mu}}{\la^p + a(z_0)\la^q}\rhoa^2,t_0+\tfrac{\la^{2\mu}}{\la^p + a(z_0)\la^q}\rhoa^2\right),
\,\,\text{and}\,\, |\partial_t \zeta| \apprle \tfrac{1}{\tfrac{\la^{2\mu}}{\la^p + a(z_0)\la^q}(\rhob-\rhoa)^2}.
\end{align*}

\vspace{.3cm}
\noindent Then the energy estimate for reads is as follows.
\begin{lemma}\label{scl_energy_pq}
Let $u$ be a weak solution of \eqref{main_eqn}. Then we have the following energy estimate:
\begin{align*}
&(\lambda^{p-2} + a(z_0)\la^{q-2}) \sup_{I^{\lambda,(p,q)}_{\rhoa}(t_0)}\fint_{B^{\lambda}_{\rhoa}(x_0)}\left|\frac{u-(u)_{\mcq^{\lambda}_{\rhoa}}}{\lambda^{-1+\mu}\rhoa}\right|^2\,dx +\fiint_{\mcq^{\lambda}_{\rhoa}(z_0)}H\left(z, |\nabla u|\right) \,dz\\
&\apprle 
\fiint_{\mcq^{\lambda}_{\rhob}(z_0)} H\left(z, \left|\frac{u-(u)_{\mcq^{\lambda}_{\rhob}}}{\lambda^{-1+\mu}(\rhob-\rhoa)}\right|\right)\,dz + (\lambda^{p-2}+a(z_0)\la^{q-2}) \fiint_{\mcq^{\lambda}_{\rhob}(z_0)}\left|\frac{u-(u)_{\mcq^{\lambda}_{\rhob}}}{\lambda^{-1+\mu}(\rhob-\rhoa)}\right|^2 \,dz.
\end{align*}
\end{lemma}

\vspace{.3cm}
\noindent Next, we prove the following:
\begin{lemma}\label{LEMMA7.11}
Let $u$ be a weak solution to \eqref{main_eqn} and let \cref{scl_energy_pq} hold with the \cref{assmpq1a} in force. Then there exists a constant $c=c(\textnormal{\texttt{data}})$ such that the following estimate
\begin{align*}
\left(\lambda^{p-2}+a(z_0)\la^{q-2} \right)\sup_{I^{\lambda,(p,q)}_{2\varrho}(t_0)}\fint_{B^{\lambda}_{2\rho}(x_0)}\left|\frac{u-(u)_{\mcq^{\la}_{2\rho}}}{\lambda^{-1+\mu}2\varrho}\right|^2 dx \leq c \left(\lambda^p+a(z_0)\la^q\right)
\end{align*}
holds whenever $\mcq^{\la}_{\rhob}(z_0)\subset \Omega_T.$
\end{lemma}
\begin{proof}
We will be estimating the terms on the right hand side of \cref{scl_energy_pq}.
For any $2\varrho \leq \rho_1< \rho_2\leq 4\varrho,$ we have
\begin{align}\label{EQUATION7.24}
   & \la^{p-2}+a(z_0)\la^{q-2}\sup_{I^{\la, (p,q)}_{\rho_1}(t_0)}\fint_{B^{\la}_{\rho_1}(x_0)}\left|\frac{u-(u)_{\mcq^{\la}_{\rho_1}}}{\la^{-1+\mu}\rho_1}\right|^2 dx\nonumber \\
    &\leq \underbrace{c \left(\frac{\rho_2}{\rho_2-\rho_1}\right)^s\fiint_{\mcq^{\lambda}_{\rho_2}(z_0)} \left(\left|\frac{u-(u)_{\mcq^{\lambda}_{\rho_2}}}{\lambda^{-1+\mu}\rho_2}\right|^p + a(z)\left|\frac{u-(u)_{\mcq^{\lambda}_{\rho_2}}}{\lambda^{-1+\mu}\rho_2}\right|^q+b(z)\left|\frac{u-(u)_{\mcq^{\lambda}_{\rho_2}}}{\lambda^{-1+\mu}\rho_2}\right|^s\right) \,dz}_{I}\nonumber\\
    &+ \underbrace{c\left(\la^{p-2}+a(z_0)\la^{q-2}\right)\left(\frac{\rho_2}{\rho_2-\rho_1}\right)^2\fiint_{\mcq^{\la}_{\rho_2}(z_0)}\left|\frac{u-(u)_{\mcq^{\la}_{\rho_2}}}{\la^{-1+\mu}\rho_2}\right|^2\,dz}_{II} 
\end{align}

\noindent\textbf{Estimate for I:} Applying \cref{pq_intrinsic poincare} with $\theta =1$ and using \textbf{pq2} of \cref{assmpq1b} and \textbf{pq4} \cref{assmpq1d}, we can estimate the first two terms:

\vspace{.3cm}
\begin{align*}
    I &\leq \left(\frac{\rho_2}{\rho_2-\rho_1}\right)^s\fiint_{\mcq^{\lambda}_{\rho_2}(z_0)} \left(\left|\frac{u-(u)_{\mcp^{\lambda}_{\rho_2}}}{\lambda^{-1+\mu}\rho_2}\right|^p + a(z)\left|\frac{u-(u)_{\mcq^{\lambda}_{\rho_2}}}{\lambda^{-1+\mu}\rho_2}\right|^q\right)\,dz\\
    &\leq \left(\frac{\rho_2}{\rho_2-\rho_1}\right)^s\fiint_{\mcq^{\lambda}_{\rho_2}(z_0)} \left(\left|\frac{u-(u)_{\mcp^{\lambda}_{\rho_2}}}{\lambda^{-1+\mu}\rho_2}\right|^p + 2a(z_0)\left|\frac{u-(u)_{\mcq^{\lambda}_{\rho_2}}}{\lambda^{-1+\mu}\rho_2}\right|^q\right)\,dz
    \apprle\left(\frac{\rho_2}{\rho_2-\rho_1}\right)^s(\la^p+a(z_0)\la^q).
\end{align*}

\vspace{.3cm}
\noindent Now we estimate the term involving the coefficient $b(z).$
\begin{align*}
&\left(\frac{\rho_2}{\rho_2-\rho_1}\right)^s \fiint_{\mcq^{\la}_{\rhob}(z_0)}b(z)\left|\frac{u-(u)_{\mcq^{\lambda}_{\rho_2}}}{\lambda^{-1+\mu}\rho_2}\right|^s \, dz\\
&\leq\left(\frac{\rho_2}{\rho_2-\rho_1}\right)^s \fiint_{\mcq^{\la}_{\rhob}(z_0)}|b(z)-\inf_{w\in \mcq^{\la}_{\rhob}(z_0)}b(w)|\left|\frac{u-(u)_{\mcq^{\lambda}_{\rho_2}}}{\lambda^{-1+\mu}\rho_2}\right|^s\,dz \\
&+ \left(\frac{\rho_2}{\rho_2-\rho_1}\right)^s \fiint_{\mcq^{\la}_{\rhob}(z_0)}\inf_{w\in \mcq^{\la}_{\rhob}(z_0)}b(w)\left|\frac{u-(u)_{\mcq^{\lambda}_{\rho_2}}}{\lambda^{-1+\mu}\rho_2}\right|^s\,dz
\end{align*}

\vspace{.3cm}
\noindent To get an estimate on the second term of the above inequality, we may use \ref{3_lem7.10} of \cref{pq_intrinsic poincare}.
\begin{align*}
&\left(\frac{\rho_2}{\rho_2-\rho_1}\right)^s \fiint_{\mcq^{\la}_{\rhob}(z_0)}\inf_{w\in \mcq^{\la}_{\rhob}(z_0)}b(w)\left|\frac{u-(u)_{\mcq^{\lambda}_{\rho_2}}}{\lambda^{-1+\mu}\rho_2}\right|^s\,dz \\
& \leq \left(\frac{\rho_2}{\rho_2-\rho_1}\right)^s\fiint_{\mcq^{\la}_{\rho_2}(z_0)}b(z)|\nabla u|^s\,dz+ \varepsilon \left(\frac{\rho_2}{\rho_2-\rho_1}\right)^s \inf_{w \in \mcq^{\la}_{\rhob}(z_0)}b(w)\la^s\\
& \leq \left(\frac{\rho_2}{\rho_2-\rho_1}\right)^s\left(\la^p+a(z_0)\la^q\right)+K \varepsilon \left(\frac{\rho_2}{\rho_2-\rho_1}\right)^s \la^p \apprle \left(\frac{\rho_2}{\rho_2-\rho_1}\right)^s\left(\la^p+a(z_0)\la^q\right).
\end{align*}

\vspace{.5cm}
\noindent Now we shall get an estimate on the first term using H\"{o}lder continuity of $b(z)$ and \cref{g_n} with $\sigma=s, \xi=p, r=2$ and $\vartheta=\frac{p}{s}.$
\begin{align*}
 &\left(\frac{\rho_2}{\rho_2-\rho_1}\right)^s \fiint_{\mcq^{\la}_{\rhob}(z_0)}|b(z)-\inf_{w\in \mcq^{\la}_{\rhob}(z_0)}b(w)|\left|\frac{u-(u)_{\mcq^{\lambda}_{\rho_2}}}{\lambda^{-1+\mu}\rho_2}\right|^s\,dz\\
 & \leq c\left(\frac{\rho_2}{\rho_2-\rho_1}\right)^s [b]_{\beta}\max\left\{\la^{-1+\mu}\rho_2, \frac{\la^\mu \rho_2}{(\la^p+a(z_0)\la^q)^{1/2}}\right\}^{\beta}\fiint_{\mcq^{\la}_{\rhob}(z_0)}\left|\frac{u-(u)_{\mcq^{\lambda}_{\rho_2}}}{\lambda^{-1+\mu}\rho_2}\right|^s\,dz\\
 &c\leq \left(\frac{\rho_2}{\rho_2-\rho_1}\right)^s[b]_{\beta}\max\left\{\la^{-1+\mu}\rho_2, \la^{\frac{2\mu-p}{2}}\rho_2\right\}^{\beta}\fiint_{\mcq^{\la}_{\rhob}(z_0)}\left|\frac{u-(u)_{\mcq^{\lambda}_{\rho_2}}}{\lambda^{-1+\mu}\rho_2}\right|^s\,dz\\
 &\apprle \left(\frac{\rho_2}{\rho_2-\rho_1}\right)^s \la^{p-s}\fiint_{\mcq^{\la}_{\rho_2}}\left(\left|\frac{u-(u)_{\mcq^{\lambda}_{\rho_2}}}{\lambda^{-1+\mu}\rho_2}\right|^p+ |\nabla u|^p\right)\,dz \left(\sup_{I^{\la, (p,q)}_{\rho_2}(t_0)}\fint_{B^{\la}_{\rho_2}(x_0)}\left|\frac{u-(u)_{\mcq^{\lambda}_{\rho_2}}}{\lambda^{-1+\mu}\rho_2}\right|^2\,dz\right)^{\frac{s-p}{2}}\\
& \apprle \left(\frac{\rho_2}{\rho_2-\rho_1}\right)^s \la^{p-s}(\la^p+a(z_0)\la^q)\left(\sup_{I^{\la, (p,q)}_{\rho_2}(t_0)}\fint_{B^{\la}_{\rho_2}(x_0)}\left|\frac{u-(u)_{\mcq^{\lambda}_{\rho_2}}}{\lambda^{-1+\mu}\rho_2}\right|^2\,dz\right)^{\frac{s-p}{2}}.
\end{align*}
Hence, combining all the estimates above, we get
\begin{align*}
    I \apprle \left(\frac{\rho_2}{\rho_2-\rho_1}\right)^s \left((\la^{p}+a(z_0)\la^q)+\la^{p-s}(\la^p+a(z_0)\la^q)\left(\sup_{I^{\la, (p,q)}_{\rho_2}(t_0)}\fint_{B^{\la}_{\rho_2}(x_0)}\left|\frac{u-(u)_{\mcq^{\lambda}_{\rho_2}}}{\lambda^{-1+\mu}\rho_2}\right|^2\,dz\right)^{\frac{s-p}{2}}\right)
\end{align*}
\noindent\textbf{Estimate for II:} Note that combined with \cref{LEMMA7.5}, it is enough to get an estimate for

\begin{align*}
II_{q}=\la^{q-2}\left(\frac{\rho_2}{\rho_2-\rho_1}\right)^2\fiint_{\mcq^{\la}_{\rho_2}(z_0)}\left|\frac{u-(u)_{\mcq^{\la}_{\rho_2}}}{\la^{-1+\mu}\rho_2}\right|^2 \,dz.   
\end{align*}

We shall make use of \cref{g_n} with $s=q, \sigma=2, \vartheta=\frac{1}{2}$ and $r=2.$ It is easy to see
\begin{align*}
    -\frac{n}{2} < \frac{1}{2} \left(1-\frac{n}{q}\right)-\left(1-\frac{1}{2}\right)\frac{n}{2} \,\,\,\, \text{if and only if}\,\,\,\, q>\frac{2n}{n+2}.
\end{align*}
This gives
\begin{align*}
II_{q} &\leq c\la^{q-2}\left(\frac{\rho_2}{\rho_2-\rho_1}\right)^2\fint_{I^{\la, (p,q)}_{\rho_2}(t_0)}\left(\fint_{B^{\la}_{\rho_2}(x_0)}\left(\left|\frac{u-(u)_{\mcq^{\la}_{\rho_2}}}{\la^{-1+\mu}\rho_2}\right|^q+|\nabla u|^q\right) dx\right)^{\frac{1}{q}}\\
&\times\left(\fint_{B^{\la}_{\rho_2}(x_0)}\left|\frac{u-(u)_{\mcq^{\la}_{\rho_2}}}{\la^{-1+\mu}\rho_2}\right|^2 dx\right)^{\frac{1}{2}}dt\\
&\leq c\la^{q-2}\left(\frac{\rho_2}{\rho_2-\rho_1}\right)^2\left(\fiint_{\mcq^{\la}_{\rho_2}(z_0)}\left|\frac{u-(u)_{\mcq^{\la}_{\rho_2}}}{\la^{-1+\mu}\rho_2}\right|^q dz+ \fiint_{\mcq^{\la}_{\rho_2}(z_0)}|\nabla u|^q dz\right)^{\frac{1}{q}}\\
&\times\left(\sup_{I^{\la, (p,q)}_{\rho_2}(t_0)}\fint_{B^{\la}_{\rho_2}(x_0)}\left|\frac{u-(u)_{\mcq^{\la}_{\rho_2}}}{\la^{-1+\mu}\rho_2}\right|^2 dx\right)^{\frac{1}{2}}.
\end{align*}
Note that from \textbf{pq4} of \cref{assmpq1a}, we have $\la^p+a(z_0)\la^q\leq ca(z_0)\la^q \leq c\la^q.$ Hence using this combined with \cref{pq_intrinsic poincare}, we further have
\begin{align*}
    II \apprle (\la^{p-2}+a(z_0)\la^{q-2})\left(\frac{\rho_2}{\rho_2-\rho_1}\right)^2 \la \left(\sup_{I^{\la, (p,q)}_{\rho_2}(t_0)}\fint_{B^{\la}_{\rho_2}(x_0)}\left|\frac{u-(u)_{\mcq^{\la}_{\rho_2}}}{\la^{-1+\mu}\rho_2}\right|^2 dx\right)^{\frac{1}{2}}.
\end{align*}
Therefore combining the above estimates and applying Young's inequality, from \eqref{EQUATION7.24}, we get
\begin{align*}
 & \sup_{I^{\la, (p,q)}_{\rho_1}(t_0)}\fint_{B^{\la}_{\rho_1}(x_0)}\left|\frac{u-(u)_{\mcq^{\la}_{\rho_1}}}{\la^{-1+\mu}\rho_1}\right|^2 dx\\
  &\leq c \left(\frac{\rho_2}{\rho_2-\rho_1}\right)^s  \la^2+ c\left(\frac{\rho_2}{\rho_2-\rho_1}\right)^2 \la \left(\sup_{I^{\la, (p,q)}_{\rho_2}(t_0)}\fint_{B^{\la}_{\rho_2}(x_0)}\left|\frac{u-(u)_{\mcq^{\la}_{\rho_2}}}{\la^{-1+\mu}\rho_2}\right|^2 dx\right)^{\frac{1}{2}}\\
  &+\left(\frac{\rho_2}{\rho_2-\rho_1}\right)^s \la^{2+p-s}\left(\sup_{I^{\la, (p,q)}_{\rho_2}(t_0)}\fint_{B^{\la}_{\rho_2}(x_0)}\left|\frac{u-(u)_{\mcq^{\la}_{\rho_2}}}{\la^{-1+\mu}\rho_2}\right|^2 dx\right)^{\frac{s-p}{2}}\\ 
  &\leq \varepsilon \left(\sup_{I^{\la, (p,q)}_{\rho_2}(t_0)}\fint_{B^{\la}_{\rho_2}(x_0)}\left|\frac{u-(u)_{\mcq^{\la}_{\rho_2}}}{\la^{-1+\mu}\rho_2}\right|^2 dx\right) + c\left[\left(\frac{\rho_2}{\rho_2-\rho_1}\right)^s+\left(\frac{\rho_2}{\rho_2-\rho_1}\right)^4+\left(\frac{\rho_2}{\rho_2-\rho_1}\right)^{\frac{2s}{p-s+2}}\right]\la^2.
\end{align*}
Finally using \cref{iter_lemma}, we get the desired estimate on $2\varrho.$ Thus we obtained
\begin{align*}
    \la^{p-2}+a(z_0)\la^{q-2}\sup_{I^{\la, (p,q)}_{2\varrho}(t_0)}\fint_{B^{\la}_{2\varrho}(x_0)}\left|\frac{u-(u)_{\mcq^{\la}_{2\varrho}}}{\la^{-1+\mu}2\varrho}\right|^2 dx \apprle\la^{p}+ a(z_0)\la^q.
\end{align*}
Hence combined with \cref{LEMMA7.5}, we conclude the proof.
\end{proof}
\begin{lemma}\label{LEM5.6}
Let $u$ be a weak solution to \eqref{main_eqn} and  let \cref{scl_energy_pq} hold with the \cref{assmpq1a} in force. Then there exists a constant $c=c(\textnormal{\texttt{data}})$ such that for any $\theta, \varepsilon \in \left(\theta_0, 1\right)$ for some $\theta_0\in (0,1),$ we have
\begin{align*}
\fiint_{\mcq^{\lambda}_{2\varrho}(z_0)} H\left(z, \left|\frac{u-(u)_{\mcq^{\la}_{2\varrho}}}{\la^{-1+\mu}2\varrho}\right|\right)dz
\leq c \left(\la^p+a(z_0)\la^q\right)^{1-\theta}\fiint_{\mcq^{\lambda}_{2\varrho}(z_0)}H\left(z, |\nabla u|\right)^{\theta}\,dz + \varepsilon (\la^p+a(z_0)\la^q)
\end{align*}
whenever $\mcq^{\la}_{4\varrho}(z_0)\subset \Omega_T.$
\end{lemma}
\begin{proof}
 The proof is similar to the $p$-intrinsic case and here we also have the advantage of the bound $a(z)\leq 2a(z_0).$ We shall proceed as in \cref{LEM5.2}. We will make use of \cref{g_n} with $\sigma=p, \xi=\theta p, \vartheta=\theta, r=2$ to estimate the first term in the left hand side. First, we note that
 \begin{align*}
     -\frac{n}{p}\leq \theta\left(1-\frac{n}{\theta p}\right)-(1-\theta)\frac{n}{2}\,\,\,\, \text{if and only if}\,\,\,\, \theta \geq \frac{n}{n+2}.
 \end{align*}
Then we have

\begin{align*}
 &\fiint_{\mcq^{\lambda}_{2\varrho}(z_0)} \left|\frac{u-(u)_{\mcq^{\la}_{\rhob}}}{\lambda^{-1+\mu}2\varrho}\right|^p dz \\
& \leq c\left(\fiint_{\mcq^{\la}_{2\varrho}(z_0)}\left|\frac{u-(u)_{\mcq^{\la}_{2\varrho}}}{\la^{-1+\mu}2\varrho}\right|^{\theta p} dz+\fiint_{\mcq^{\lambda}_{2\varrho}(z_0)}|\nabla u|^{\theta p} dz\right)\left(\sup_{I^{\lambda, (p,q)}_{2\varrho}(t_0)}\fint_{B^{\lambda}_{2\varrho}(x_0)}\left|\frac{u-(u)_{\mcq^{\la}_{2\varrho}}}{\lambda^{-1+\mu}2\varrho}\right|^2 dx\right)^{\frac{(1-\theta)p}{2}}. 
\end{align*}
Now using \cref{pq_intrinsic poincare} and \cref{LEMMA7.11} in the above inequality, we find

\begin{align}\label{EQUATION7.25}
  \fiint_{\mcq^{\lambda}_{2\varrho}(z_0)} &\left|\frac{u-(u)_{\mcq^{\la}_{\rhob}}}{\lambda^{-1+\mu}2\varrho}\right|^p\,dz \leq \la^{(1-\theta)p}\left(\fiint_{\mcq^{\la}_{2\varrho}(z_0)}|\nabla u|^{\theta p}\,dz\right)+ \varepsilon \la^{p}.   
\end{align}
Similarly using \cref{g_n} with $\sigma=q, \xi=\theta q, \vartheta=\theta, r=2$ we get

\begin{align*}
  \fiint_{\mcq^{\lambda}_{2\varrho}(z_0)} a(z)\left|\frac{u-(u)_{\mcq^{\la}_{\rhob}}}{\lambda^{-1+\mu}2\varrho}\right|^q dz &\leq 2a(z_0)\left|\frac{u-(u)_{\mcq^{\la}_{\rhob}}}{\lambda^{-1+\mu}2\varrho}\right|^q dz \\
 &\leq c\left(\fiint_{\mcq^{\la}_{2\varrho}(z_0)}a(z_0)^{\theta}\left|\frac{u-(u)_{\mcq^{\la}_{2\varrho}}}{\la^{-1+\mu}2\varrho}\right|^{\theta q} dz+\fiint_{\mcq^{\lambda}_{2\varrho}(z_0)}a(z_0)^{\theta}|\nabla u|^{\theta q} dz\right)\\
 &\times\left(\sup_{I^{\lambda, (p,q)}_{2\varrho}(t_0)}\fint_{B^{\lambda}_{2\varrho}(x_0)}a(z_0)^{1-\theta}\left|\frac{u-(u)_{\mcq^{\la}_{2\varrho}}}{\lambda^{-1+\mu}2\varrho}\right|^2 dx\right)^{\frac{(1-\theta)q}{2}}.  
\end{align*}
In the first inequality above, we used \textbf{pq4} of \cref{assmpq1d}. Now using \cref{pq_intrinsic poincare} and \cref{LEMMA7.11} in the above inequality, we find
\begin{align}\label{EQUATION7.26}
  \fiint_{\mcq^{\lambda}_{2\varrho}(z_0)} a(z)\left|\frac{u-(u)_{\mcq^{\la}_{\rhob}}}{\lambda^{-1+\mu}2\varrho}\right|^q dz \leq c (a(z_0)\la^q)^{1-\theta}\fiint_{\mcq^{\la}_{2\varrho}(z_0)}a^{\theta}(z_0)|\nabla u|^{\theta q} dz + \varepsilon a(z_0)\la^q.   
\end{align}
It remains to estimate the third term in the left hand side.

\begin{align*}
 \fiint_{\mcq^{\la}_{\rhob}(z_0)}b(z)\left|\frac{u-(u)_{\mcq^{\la}_{2\varrho}}}{\la^{-1+\mu}2\varrho}\right|^s\,dz &\leq \underbrace{\fiint_{\mcq^{\la}_{2\varrho}(z_0)}\inf_{w\in \mcq^{\la}_{2\varrho}}b(w)\left|\frac{u-(u)_{\mcq^{\la}_{2\varrho}}}{\la^{-1+\mu}2\varrho}\right|^s\,dz}_{I}\\
 &+\underbrace{c[b]_{\beta}\max\left\{\la^{-1+\mu}2\varrho, \frac{\la^\mu 2\varrho}{(\la^p+a(z_0)\la^q)^{1/2}}\right\}^{\beta}\fiint_{\mcq^{\la}_{2\varrho}(z_0)}\left|\frac{u-(u)_{\mcq^{\lambda}_{2\varrho}}}{\lambda^{-1+\mu}2\varrho}\right|^s\,dz}_{II}
\end{align*}
We estimate $I$ by using \cref{g_n} with $\sigma=s, \xi=\theta s, \vartheta =\theta, r=2.$

\begin{align*}
    I &\apprle \fiint_{\mcq^{\la}_{2\varrho}(z_0)}\inf_{w\in \mcq^{\la}_{2\varrho} }b(w)^{\theta}\left|\frac{u-(u)_{\mcq^{\la}_{2\varrho}}}{\la^{-1+\mu}2\varrho}\right|^{\theta s}\,dz+ \fiint_{\mcq^{\la}_{2\varrho}(z_0)}\inf_{w\in \mcq^{\la}_{2\varrho} }b(w)^{\theta} |\nabla u|^{\theta s}\,dz\\
    &\times \inf_{w\in \mcq^{\la}_{2\varrho} }b(w)^{1-\theta}\left(\sup_{I^{\la, (p,q)}_{2\varrho}}\fint_{B^{\la}_{2\varrho}(x_0)}\left|\frac{u-(u)_{\mcq^{\la}_{2\varrho}}}{\la^{-1+\mu}2\varrho}\right|^2\,dx\right)^{\frac{(1-\theta)s}{2}}\\
    &\overset{\cref{p_intrinsic poincare 2},\ref{3_lem7.10}\,\, \text{and}\,\,\cref{LEMMA7.11} }{\apprle} \inf_{w\in \mcq^{\la}_{2\varrho} }b(w)^{1-\theta}\la^{(1-\theta)s}\left(\fiint_{\mcq^{\la}_{2\varrho}(z_0)}\inf_{w\in \mcq^{\la}_{2\varrho} }b(w)^{\theta}|\nabla u|^{\theta s}\,dz+\varepsilon \inf_{w\in \mcq^{\la}_{2\varrho} }b(w)^{\theta} \la^{\theta s}\right)\\
    &\apprle \la^{(p-s)(1-\theta)}\la^{(1-\theta)s}\fiint_{\mcq^{\la}_{2\varrho}(z_0)} b(z)^{\theta}|\nabla u|^{\theta s}\,dz+ \varepsilon \inf_{w\in \mcq^{\la}_{2\varrho}}b(w)\la^s\\
    &\apprle \la^{p(1-\theta)}\fiint_{\mcq^{\la}_{2\varrho}(z_0)} b(z)^{\theta}|\nabla u|^{\theta s}\,dz+ \varepsilon \la^p.
\end{align*}
We use

\begin{align}\label{bounds on maximum radius}
\max\left\{\la^{-1+\mu}2\varrho, \frac{\la^\mu 2\varrho}{(\la^p+a(z_0)\la^q)^{1/2}}\right\}^{\beta} \leq \max\left\{\la^{-1+\mu}2\varrho, \la^{\frac{2\mu-p}{2}}2\varrho\right\}^{\beta}\leq \frac{K}{2[b]_{\beta}} \la^{p-s}    
\end{align}
and \cref{g_n} with $\sigma =s, \xi=\theta p, \vartheta=\frac{\theta p}{s}, r=2.$ Note that the inequality condition on the \cref{g_n} is satisfied with this choice, since
\begin{align*}
    \frac{ns}{(n+2)p}\leq \frac{n}{n+2}\left(1+\frac{2}{(n+2)p}\right)\leq \frac{n}{n+2}\left(1+\frac{1}{n}\right)=\frac{n+1}{n+2}<1,
\end{align*}
where we used $s\leq p+\frac{2}{n+2}$ and $p>\frac{2n}{n+2}.$

\begin{align*}
   II &\apprle \la^{p-s}\left(\fiint_{\mcq^{\la}_{2\varrho}(z_0)}\left|\frac{u-(u)_{\mcq^{\la}_{2\varrho}}}{\la^{-1+\mu}2\varrho}\right|^{\theta p}+|\nabla u|^{\theta p}, dz\right) \left(\sup_{I^{\la, (p,q)}_{2\varrho}}\fint_{B^{\la}_{2\varrho}(x_0)}\left|\frac{u-(u)_{\mcq^{\la}_{2\varrho}}}{\la^{-1+\mu}2\varrho}\right|^2\,dx\right)^{\frac{s-\theta p}{2}}\\
   &\overset{\cref{p_intrinsic poincare 2},\ref{1_lem7.10}\,\, \text{and}\,\,\cref{LEMMA7.11}}{\apprle} \la^{p-s}\left(\fiint_{\mcq^{\la}_{2\varrho}(z_0)}|\nabla u|^{\theta p}\,dz+\varepsilon \la^{\theta p}\right)\la^{s-\theta p} 
\end{align*}
Thus combining the above estimates, we obtain
\begin{align}\label{EQQQ3.46}
\fiint_{\mcq^{\la}_{\rhob}(z_0)}b(z)\left|\frac{u-(u)_{\mcq^{\la}_{2\varrho}}}{\la^{-1+\mu}2\varrho}\right|^s\,dz\apprle \la^{(1-\theta)p}\left(\fiint_{\mcq^{\la}_{2\varrho}(z_0)}|\nabla u|^{\theta p}\,dz+ \fiint_{\mcq^{\la}_{2\varrho}(z_0)}b(z)^{\theta}|\nabla u|^{\theta s}\,dz\right)+ \varepsilon \la^p.   
\end{align}
Hence combining \eqref{EQUATION7.25},\eqref{EQUATION7.26} and \eqref{EQQQ3.46}  with {\bf pq4} of \cref{assmpq1d}, we get 
\begin{align*}
&\fiint_{\mcq^{\lambda}_{2\varrho}(z_0)} H\left(z, \left|\frac{u-(u)_{\mcq^{\la}_{\rhob}}}{\lambda^{-1+\mu}2\varrho}\right|\right)dz\\
&\leq\la^{(1-\theta)p}+\left(a(z_0)\la^q\right)^{1-\theta}\fiint_{\mcq^{\la}_{2\varrho}(z_0)}\left(|\nabla u|^{\theta p}+a(z)^{\theta}|\nabla u|^{\theta q}+b(z)^{\theta}|\nabla u|^{\theta s}\right)dz + \varepsilon \left(\la^p+a(z_0)\la^q\right).
\end{align*}
Since $\theta<1$ and $1-\theta <1$ we can further estimate the right hand side of the above inequality and get the claim by using 
\[|a+b|^{\beta} \leq (1+\delta)^{\beta-1}|a|^{\beta}+\left(1+\frac{1}{\delta}\right)^{\beta-1}|b|^{\beta}\] for $1\leq \beta < \infty$ and arbitrary $a, b\in \RR$ and $\delta >0.$ This completes the proof.
\end{proof}
\begin{lemma}\label{LEM5.7}
Let $u$ be a weak solution to \eqref{main_eqn} and let \cref{scl_energy_pq} hold with \cref{assmpq1a} in force. Then there exist constants $c=c(\textnormal{\texttt{data}})$ and $\theta_0 \in (0, 1)$ depending on $n,p,q$ such that for any $\theta, \varepsilon \in (\theta_0,1),$ we have
\begin{align*}
\lambda^{p-2}+a(z_0)\lambda^{q-2}&\fiint_{\mcq^{\lambda}_{2\varrho}(z_0)}\left|\frac{u-(u)_{\mcq^{\la}_{2\varrho}}}{\lambda^{-1+\mu}2\varrho}\right|^2 dz \\
&\apprle \left[\varepsilon (\lambda^p+a(z_0)\lambda^q )+ c(\varepsilon) \left(\fiint_{\mcq^{\lambda}_{2\varrho}(z_0)}|\nabla u|^{\theta p}+a^{\theta}(z)|\nabla u|^{\theta q}+b(z)^{\theta}|\nabla u|^{\theta s}\,dz\right)^{\frac{1}{\theta}}\right]    
\end{align*}
whenever $\mcq^{\la}_{4\varrho}(z_0)\subset \Omega_T.$
\end{lemma}
\begin{proof}
Following the proof of \cref{LEM5.3} we have

\begin{align}\label{EQUATION7.27}
\lambda^{p-2}\fiint_{\mcq^{\lambda}_{2\varrho}(z_0)}\left|\frac{u-(u)_{\mcq^{\lambda}_{2\varrho}}}{\lambda^{-1+\mu} 2\varrho }\right|^2 dz
\apprle\left[\varepsilon \lambda^p + c(\varepsilon) \left(\fiint_{G^{\lambda}_{2\varrho}(z_0)}|\nabla u|^{\theta p}\right)^{\frac{1}{\theta}}\right].
\end{align}
We shall apply \cref{g_n} with $\sigma=2, \xi=\theta q, \vartheta=\frac{\theta q}{2}$ and $r=2.$ It can be easily checked that
\begin{align*}
    -\frac{n}{2} \leq \frac{\theta q}{2}\left(1-\frac{n}{\theta q}\right)-\left(1-\frac{\theta q}{2}\right)\frac{n}{2}\,\,\,\, \text{if and only if}\,\,\,\, \theta \geq \frac{n}{(n+2)q}.
\end{align*}
Then we have

  \begin{align*} 
&\fiint_{\mcq^{\lambda}_{2\varrho}(z_0)}\left|\frac{u-(u)_{\mcq^{\la}_{2\varrho}}}{\lambda^{-1+\mu}2\varrho}\right|^2 dz\\
&\leq c\left(\fiint_{\mcq^{\la}_{2\varrho}(z_0)}\left|\frac{u-(u)_{\mcq^{\la}_{2\varrho}}}{\la^{-1+\mu}2\varrho}\right|^{\theta q} dz+\fiint_{\mcq^{\lambda}_{2\varrho}(z_0)}|\nabla u|^{\theta q} dz\right)
\left(\sup_{I^{\lambda, (p,q)}_{2\varrho}(t_0)}\fint_{B^{\lambda}_{2\varrho}(x_0)}\left|\frac{u-(u)_{\mcq^{\la}_{2\varrho}}}{\lambda^{-1+\mu}2\varrho}\right|^2 dx\right)^{\frac{2-\theta q}{2}}. 
  \end{align*}
Moreover, applying \cref{pq_intrinsic poincare} and \cref{LEMMA7.11} we find
\begin{align*}
\fiint_{\mcq^{\lambda}_{2\varrho}(z_0)}\left|\frac{u-(u)_{\mcq^{\la}_{2\varrho}}}{\lambda^{-1+\mu}2\varrho}\right|^2 dz \leq c \la^{2-\theta q}\left(\fiint_{\mcq^{\la}_{2\varrho}(z_0)}|\nabla u|^{\theta q}\,dz\right)+\varepsilon \la^2.
\end{align*}
Hence
\begin{align*}
a(z_0)\lambda^{q-2}\fiint_{\mcq^{\lambda}_{2\varrho}(z_0)}\left|\frac{u-(u)_{\mcq^{\la}_{2\varrho}}}{\lambda^{-1+\mu}(2\varrho)}\right|^2 dz
&\apprle  a(z_0) \lambda^{(1-\theta) q} \left(\fiint_{\mcq^{\lambda}_{2\varrho}}|\nabla u|^{\theta q}dz\right)+ \varepsilon a(z_0)\la^q\\
&=c\left(a(z_0)\lambda^q\right)^{1-\theta}\left(\fiint_{\mcq^{\lambda}_{2\varrho}}a(z_0)^{\theta}|\nabla u|^{\theta q}dz\right)+ \varepsilon a(z_0)\la^q.
\end{align*}
Now using Young's inequality we get

\begin{align}\label{EQUATION7.28}
a(z_0)\lambda^{q-2}\fiint_{\mcq^{\lambda}_{2\varrho}(z_0)}\left|\frac{u-(u)_{\mcq^{\la}_{2\varrho}}}{\lambda^{-1+\mu}2\varrho}\right|^2 dz \apprle \left[\varepsilon a(z_0)\lambda^q + c(\varepsilon) \left(\fiint_{\mcq^{\lambda}_{2\varrho}}a(z_0)^{\theta}|\nabla u|^{\theta q}dz\right)^{\frac{1}{\theta}}\right].      
\end{align}
Therefore, adding \eqref{EQUATION7.27} and \eqref{EQUATION7.28}, we get the desired result.
\end{proof}

Now we prove the reverse H\"{o}lder inequality in the $(p,q)$-intrinsic case.
\begin{lemma}\label{RHPQ}
Let $u$ be a weak solution to \eqref{main_eqn}. Moreover, we assume that the energy estimate given in \cref{scl_energy_pq} and \cref{assmpq1a} hold for $(p,q)$-intrinsic cylinders. Then there exists a constant $c=c(\textnormal{\texttt{data}})$ and $\theta_0 \in (0,1)$ such that for any $\theta \in (\theta_0, 1),$ we have

\begin{align*}
\fiint_{\mcq^{\lambda}_{\varrho}(z_0)}H(z, \nabla u)\, dz \leq c \left(\fiint_{\mcq^{\lambda}_{2\varrho}(z_0)}H(z, \nabla u)^{\theta }dz\right)^{\frac{1}{\theta}}.
\end{align*}
\end{lemma}
\begin{proof}
From the energy estimate \cref{scl_energy_pq} and the definition of $\rhoa, \rhob,$ we have

\begin{align}\label{EQUATION7.29}
    \fiint_{\mcq^{\la}_{\varrho}(z_0)} H(z, \nabla u) dz \apprle \fiint_{\mcq^{\lambda}_{2\varrho}(z_0)} H\left(z, \left|\frac{u-(u)_{\mcq^{\lambda}_{2\varrho}}}{\lambda^{-1+\mu}2\varrho}\right| \right)\,dz\nonumber\\
	 +\lambda^{p-2}+a(z_0)\la^{q-2} \fiint_{\mcq^{\lambda}_{2\varrho}(z_0)}\left|\frac{u-(u)_{\mcq^{\lambda}_{2\varrho}}}{\lambda^{-1+\mu}2\varrho}\right|^2\,dz.
\end{align}
From \cref{LEM5.6}, we have

\begin{align*}
\fiint_{\mcq^{\lambda}_{2\varrho}(z_0)} H\Bigg(z, \left|\frac{u-(u)_{\mcq^{\lambda}_{2\varrho}}}{\lambda^{-1+\mu}2\varrho}\right|\Bigg)\,dz
&\apprle (\la^p+a(z_0)\la^q)^{1-\theta}\fiint_{\mcq^{\lambda}_{2\varrho}}H\left(z, |\nabla u|\right)^{\theta}dz + \varepsilon (\la^p+a(z_0)\la^q)\\
&\overset{\text{Young's inequality}}{\leq} \varepsilon (\la^p+a(z_0)\la^q) + c \left(\fiint_{\mcp^{\la}_{2\varrho}(z_0)}H\left(z, |\nabla u|\right)^{\theta}dz\right)^{\frac{1}{\theta}}.
\end{align*}
Moreover, from \cref{LEM5.7}, we obtain

\begin{align*}
 \lambda^{p-2}+a(z_0)\la^{q-2} \fiint_{\mcq^{\lambda}_{2\varrho}(z_0)}\left|\frac{u-(u)_{\mcq^{\lambda}_{2\varrho}}}{\lambda^{-1+\mu}2\varrho}\right|^2\,dz &\apprle\varepsilon (\lambda^p+a(z_0)\lambda^q )+ c \left(\fiint_{\mcq^{\lambda}_{2\varrho}(z_0)}H(z, \nabla u)^{\theta}\,dz\right)^{\frac{1}{\theta}}.
\end{align*}
Substituting the above estimates in \eqref{EQUATION7.29}, we get

\begin{align*}
   \fiint_{\mcq^{\la}_{\varrho}(z_0)} H\left(z, |\nabla u|\right) dz \leq \varepsilon (\la^p+a(z_0)\la^q) + c \left(\fiint_{\mcq^{\la}_{2\varrho}(z_0)}H(z, \nabla u)^{\theta}dz\right)^{\frac{1}{\theta}}.
\end{align*}
The proof can be completed by absorbing $\varepsilon (\la^p+a(z_0)\la^q)$ on the left-hand side, which is allowed due to \textbf{ pq3} of \cref{assmpq1c}.
\end{proof}
\subsection{Reverse H\"{o}lder inequality for $(p, s)$-phase} In this subsection, we prove the reverse H\"{o}lder inequality for $(p, s)$-phase.
\begin{assumption}
In the $(p,s)$-phase, we assume that the following is satisfied:
\begin{enumerate}
    \item [\textbf{ps1.}] $K\lambda^p \geq  a(z_0)\lambda^q,\,\,\, K\la^p< b(z_0)\la^s.$ \label{assmps1a}\\
    \item [\textbf{ps2.}] $\fiint_{\mathcal{S}^{\lambda}_{\rho}(z_0)}(|\nabla u|^p+a(z)|\nabla u|^q+b(z)|\nabla u|^s) dz < \lambda^p+b(z_0)\la^s\,\,\, \text{holds for all} \,\,\,\rho \in (\rho_a, \rho_b].$ \label{assmps1b}\\
    \item [\textbf{ps3.}] $\fiint_{\mathcal{S}^{\lambda}_{\rho}(z_0)}(|\nabla u|^p+a(z)|\nabla u|^q+b(z)|\nabla u|^s) dz =\lambda^p+b(z_0)\la^s.$ \label{assmps1c}\\
    \item [\textbf{ps4.}] $\frac{b(z_0)}{2} \leq b(z) \leq 2b(z_0)\,\,\, \text{holds for every}\,\,\, z \in \mathcal{S}^{\lambda}_{\rhob}(z_0).$ \label{assmps1d}
\end{enumerate}
\end{assumption}
We shall start with parabolic Sobolev-Poincar\'{e} inequalities for $(p,s)$-phase.
\begin{lemma}\label{ps_intrinsic poincare}
 Let $u$ be a weak solution to \eqref{main_eqn} and let \cref{assmps1a} be in force. Then for any $\theta\in \left(\max\left\{\frac{s-1}{p}, \frac{1}{p}\right\}, 1\right]$ and $\varepsilon \in (0,1),$ the estimates \ref{1_lem7.10}, \ref{2_lem7.10} and \ref{3_lem7.10} from \cref{pq_intrinsic poincare} hold whenever $\mathcal{S}^{\lambda}_{\rhob}(z_0) \subset \Omega_T.$
\end{lemma}
\begin{proof}
    We postpone the proof to \cref{appendix}.
\end{proof}
Now we state the energy estimate \eqref{general caccipoli} for the $(p,s)$-phase. We consider the following cutoff functions.

\vspace{.5cm}
\noindent \textbf{Cut-off functions for $(p,s)$-phase.} On the other hand, for $(p,s)$-phase we use
\begin{align*}
&\eta:=\eta(x) \in C_c^{\infty}(B_{\la^{-1+\mu}\rhob}(x_0)), \quad  \eta \equiv 1 \,\,\text{on}\, B_{\la^{-1+\mu}\rhoa}(x_0), \quad 0\leq \eta \leq 1 \,\,\,\,\,\text{and}\,\,\,\,\,\, |\nabla \eta| \apprle \frac{1}{\la^{-1+\mu}(\rhob-\rhoa)}\\
&\zeta \in C^{\infty}_c\left(t_0-\frac{\la^{2\mu}}{\la^p+b(z_0)\la^s}(\rhob-h_0)^2, t_0+\frac{\la^{2\mu}}{\la^p+b(z_0)\la^s}(\rhob-h_0)^2\right),\,\,\,\,\, 0\leq \zeta \leq 1,\\
&\zeta \equiv 1 \,\,\text{on}\,\,\left(t_0-\tfrac{\la^{2\mu}}{\la^p + b(z_0)\la^s}\rhoa^2,t_0+\tfrac{\la^{2\mu}}{\la^p + b(z_0)\la^s}\rhoa^2\right),
\,\,\text{and}\,\, |\partial_t \zeta| \apprle \tfrac{1}{\tfrac{\la^{2\mu}}{\la^p + b(z_0)\la^s}(\rhob-\rhoa)^2}.
\end{align*}
Then the energy estimate for the $(p,s)$-phase reads is as follows.
\begin{lemma}\label{scl_energy_ps}
Let $u$ be a weak solution of \eqref{main_eqn}. Then we have the following energy estimate:
\begin{align*}
&(\lambda^{p-2} + b(z_0)\la^{s-2}) \sup_{I^{\lambda,(p,s)}_{\rhoa}(t_0)}\fint_{B^{\lambda}_{\rhoa}(x_0)}\left|\frac{u-(u)_{\mathcal{S}^{\lambda}_{\rhoa}}}{\lambda^{-1+\mu}\rhoa}\right|^2\,dx +\fiint_{\mathcal{S}^{\lambda}_{\rhoa}(z_0)}H\left(z, |\nabla u|\right) \,dz\\
&\apprle 
\fiint_{\mathcal{S}^{\lambda}_{\rhob}(z_0)} H\left(z, \left|\frac{u-(u)_{\mathcal{S}^{\lambda}_{\rhob}}}{\lambda^{-1+\mu}(\rhob-\rhoa)}\right|\right)\,dz + (\lambda^{p-2}+b(z_0)\la^{s-2}) \fiint_{\mathcal{S}^{\lambda}_{\rhob}(z_0)}\left|\frac{u-(u)_{\mathcal{S}^{\lambda}_{\rhob}}}{\lambda^{-1+\mu}(\rhob-\rhoa)}\right|^2 \,dz.
\end{align*}
\end{lemma}

\vspace{.3cm}
\noindent We need the following result. The proof can be completed by an application of \cref{ps_intrinsic poincare}, \cref{g_n} and \cref{iter_lemma}. Due to its similarity to the $(p,q)$-phase, we omit the proof here.
\begin{lemma}\label{LEMMA3.16}
Let $u$ be a weak solution to \eqref{main_eqn} and let \cref{scl_energy_ps} hold with \cref{assmps1a} in force. Then there exists a constant $c=c(\textnormal{\texttt{data}})$ such that the following estimate
\begin{align*}
\left(\lambda^{p-2}+b(z_0)\la^{s-2} \right)\sup_{I^{\lambda,(p,s)}_{2\varrho}(t_0)}\fint_{B^{\lambda}_{2\rho}(x_0)}\left|\frac{u-(u)_{\mcq^{\la}_{2\rho}}}{\lambda^{-1+\mu}2\varrho}\right|^2 dx \leq c \left(\lambda^p+b(z_0)\la^s\right)
\end{align*}
holds whenever $\mathcal{S}^{\la}_{\rhob}(z_0)\subset \Omega_T.$
\end{lemma}
Now we prove the following lemma, which leads us to revese H\"{o}lder inequality in this case.
\begin{lemma}\label{LEMMA3.17}
Let $u$ be a weak solution to \eqref{main_eqn} and let \cref{scl_energy_ps} hold with \cref{assmps1a} in force. Then there exists a constant $c=c(\textnormal{\texttt{data}})$ such that for any $\theta, \varepsilon \in \left(\theta_0, 1\right)$ for some $\theta_0\in (0,1),$ we have
\begin{align*}
\fiint_{\mathcal{S}^{\lambda}_{2\varrho}(z_0)} H\left(z, \left|\frac{u-(u)_{\mathcal{S}^{\la}_{2\varrho}}}{\la^{-1+\mu}2\varrho}\right|\right)dz
\leq c \left(\la^p+b(z_0)\la^s\right)^{1-\theta}\fiint_{\mathcal{S}^{\lambda}_{2\varrho}(z_0)}H\left(z, |\nabla u|\right)^{\theta}\,dz + \varepsilon (\la^p+b(z_0)\la^s)
\end{align*}
whenever $\mathcal{S}^{\la}_{4\varrho}(z_0)\subset \Omega_T.$
\end{lemma}
\begin{proof}
 The proof is similar to the proof of \cref{LEM5.6}. The first term on the left-hand side of the statement can be estimated exactly the same way as done in \cref{LEM5.6}.

 \begin{align}\label{EQUATION3.50}
     \fiint_{\mathcal{S}^{\lambda}_{2\varrho}(z_0)} &\left|\frac{u-(u)_{\mathcal{S}^{\la}_{\rhob}}}{\lambda^{-1+\mu}2\varrho}\right|^p\,dz \leq \la^{(1-\theta)p}\left(\fiint_{\mathcal{S}^{\la}_{2\varrho}(z_0)}|\nabla u|^{\theta p}\,dz\right)+ \varepsilon \la^{p}.
 \end{align}
 For the third term, we use $b(z)\leq 2b(z_0)$ and \cref{g_n} with $\sigma=s, \xi=\theta s, \vartheta=\theta, r=2.$

 \begin{align}\label{EQUATION3.51} 
 \fiint_{\mathcal{S}^{\lambda}_{2\varrho}(z_0)} b(z)\left|\frac{u-(u)_{\mathcal{S}^{\la}_{\rhob}}}{\lambda^{-1+\mu}2\varrho}\right|^s dz \leq c (b(z_0)\la^s)^{1-\theta}\fiint_{\mathcal{S}^{\la}_{2\varrho}(z_0)}b^{\theta}(z_0)|\nabla u|^{\theta s} dz + \varepsilon b(z_0)\la^s.
 \end{align}
 We will show only the estimate of the second term below.

 \begin{align*}
\fiint_{\mathcal{S}^{\la}_{\rhob}(z_0)}a(z)\left|\frac{u-(u)_{\mathcal{S}^{\la}_{2\varrho}}}{\la^{-1+\mu}2\varrho}\right|^q\,dz &\leq \underbrace{\fiint_{\mathcal{S}^{\la}_{2\varrho}(z_0)}\inf_{w\in \mathcal{S}^{\la}_{2\varrho}}a(w)\left|\frac{u-(u)_{\mathcal{S}^{\la}_{2\varrho}}}{\la^{-1+\mu}2\varrho}\right|^q\,dz}_{I}\\
 &+\underbrace{c[b]_{\beta}\max\left\{\la^{-1+\mu}2\varrho, \frac{\la^\mu 2\varrho}{(\la^p+b(z_0)\la^s)^{1/2}}\right\}^{\alpha}\fiint_{\mathcal{S}^{\la}_{2\varrho}(z_0)}\left|\frac{u-(u)_{\mathcal{S}^{\lambda}_{2\varrho}}}{\lambda^{-1+\mu}2\varrho}\right|^q\,dz.}_{II}     
 \end{align*}

 \vspace{.3cm}
\noindent We estimate $I$ by using \cref{g_n} with $\sigma=q, \xi=\theta q, \vartheta =\theta, r=2.$

\begin{align*}
    I &\apprle \fiint_{\mathcal{S}^{\la}_{2\varrho}(z_0)}\inf_{w\in \mathcal{S}^{\la}_{2\varrho} }a(w)^{\theta}\left|\frac{u-(u)_{\mathcal{S}^{\la}_{2\varrho}}}{\la^{-1+\mu}2\varrho}\right|^{\theta q}\,dz+ \fiint_{\mathcal{S}^{\la}_{2\varrho}(z_0)}\inf_{w\in \mathcal{S}^{\la}_{2\varrho} }a(w)^{\theta} |\nabla u|^{\theta q}\,dz\\
    &\times \inf_{w\in \mathcal{S}^{\la}_{2\varrho}(z_0) }a(w)^{1-\theta}\left(\sup_{I^{\la, (p,s)}_{2\varrho}}\fint_{B^{\la}_{2\varrho}(x_0)}\left|\frac{u-(u)_{\mathcal{S}^{\la}_{2\varrho}}}{\la^{-1+\mu}2\varrho}\right|^2\,dx\right)^{\frac{(1-\theta)q}{2}}\\
    &\overset{\cref{p_intrinsic poincare 2},\ref{3_lem7.10}\,\, \text{and}\,\,\cref{LEMMA3.16} }{\apprle} \inf_{w\in \mathcal{S}^{\la}_{2\varrho}(z_0) }a(w)^{1-\theta}\la^{(1-\theta)q}\left(\fiint_{\mathcal{S}^{\la}_{2\varrho}(z_0)}\inf_{w\in \mathcal{S}^{\la}_{2\varrho} }a(w)^{\theta}|\nabla u|^{\theta q}\,dz+\varepsilon \inf_{w\in \mathcal{S}^{\la}_{2\varrho} }a(w)^{\theta} \la^{\theta q}\right)\\
    &\apprle \la^{(p-q)(1-\theta)}\la^{(1-\theta)q}\fiint_{\mathcal{S}^{\la}_{2\varrho}(z_0)} a(z)^{\theta}|\nabla u|^{\theta q}\,dz+ \varepsilon \inf_{w\in \mathcal{S}^{\la}_{2\varrho}}a(w)\la^q\\
    &\apprle \la^{p(1-\theta)}\fiint_{\mathcal{S}^{\la}_{2\varrho}(z_0)} a(z)^{\theta}|\nabla u|^{\theta q}\,dz+ \varepsilon \la^p.
\end{align*}
We again use
\begin{align*}
\max\left\{\la^{-1+\mu}2\varrho, \frac{\la^\mu 2\varrho}{(\la^p+b(z_0)\la^s)^{1/2}}\right\}^{\alpha} \leq \max\left\{\la^{-1+\mu}2\varrho, \la^{\frac{2\mu-p}{2}}2\varrho\right\}^{\alpha}\apprle \la^{p-q}    
\end{align*}
and \cref{g_n} with $\sigma =q, \xi=\theta p, \vartheta=\frac{\theta p}{q}, r=2.$ Note that the inequality condition on the \cref{g_n} is satisfied with this choice, since

\begin{align*}
    \frac{nq}{(n+2)p}\leq \frac{n}{n+2}\left(1+\frac{2}{(n+2)p}\right)\leq \frac{n}{n+2}\left(1+\frac{1}{n}\right)=\frac{n+1}{n+2}<1,
\end{align*}
where we used $q\leq p+\frac{2}{n+2}$ and $p>\frac{2n}{n+2}.$

\begin{align*}
   II &\apprle \la^{p-q}\left(\fiint_{\mathcal{S}^{\la}_{2\varrho}(z_0)}\left|\frac{u-(u)_{\mathcal{S}^{\la}_{2\varrho}}}{\la^{-1+\mu}2\varrho}\right|^{\theta p}+|\nabla u|^{\theta p}, dz\right) \left(\sup_{I^{\la, (p,s)}_{2\varrho}}\fint_{B^{\la}_{2\varrho}(x_0)}\left|\frac{u-(u)_{\mathcal{S}^{\la}_{2\varrho}}}{\la^{-1+\mu}2\varrho}\right|^2\,dx\right)^{\frac{q-\theta p}{2}}\\
   &\overset{\cref{p_intrinsic poincare 2},\ref{1_lem7.10}\,\, \text{and}\,\,\cref{LEMMA3.16}}{\apprle} \la^{p-q}\left(\fiint_{\mathcal{S}^{\la}_{2\varrho}(z_0)}|\nabla u|^{\theta p}\,dz+\varepsilon \la^{\theta p}\right)\la^{q-\theta p} 
\end{align*}
Thus combining the above estimates, we obtain

\begin{align}\label{EQUATION3.52}
\fiint_{\mathcal{S}^{\la}_{\rhob}(z_0)}a(z)\left|\frac{u-(u)_{\mathcal{S}^{\la}_{2\varrho}}}{\la^{-1+\mu}2\varrho}\right|^q\,dz\apprle \la^{(1-\theta)p}\left(\fiint_{\mathcal{S}^{\la}_{2\varrho}(z_0)}|\nabla u|^{\theta p}\,dz+ \fiint_{\mathcal{S}^{\la}_{2\varrho}(z_0)}a(z)^{\theta}|\nabla u|^{\theta q}\,dz\right)+ \varepsilon \la^p.   
\end{align}
Hence combining \eqref{EQUATION3.50},\eqref{EQUATION3.51} and \eqref{EQUATION3.52}  with {\bf ps4} of \cref{assmps1d}, we get

\begin{align*}
&\fiint_{\mathcal{S}^{\lambda}_{2\varrho}(z_0)} H\left(z, \left|\frac{u-(u)_{\mathcal{S}^{\la}_{\rhob}}}{\lambda^{-1+\mu}2\varrho}\right|\right)dz\\
&\leq\la^{(1-\theta)p}+\left(b(z_0)\la^q\right)^{1-\theta}\fiint_{\mathcal{S}^{\la}_{2\varrho}(z_0)}\left(|\nabla u|^{\theta p}+a(z)^{\theta}|\nabla u|^{\theta q}+b(z)^{\theta}|\nabla u|^{\theta s}\right)dz + \varepsilon \left(\la^p+b(z_0)\la^q\right).
\end{align*}
This completes the proof.
\end{proof}
\noindent We state the next lemma without proof.
\begin{lemma}\label{LEMMA3.18}
Let $u$ be a weak solution to \eqref{main_eqn} and let \cref{scl_energy_ps} hold with \cref{assmps1a} in force. Then there exist constants $c=c(\textnormal{\texttt{data}})$ and $\theta_0 \in (0, 1)$ depending on $n,p,q$ such that for any $\theta, \varepsilon \in (\theta_0,1),$ we have
\begin{align*}
\lambda^{p-2}+b(z_0)\lambda^{s-2}&\fiint_{\mathcal{S}^{\lambda}_{2\varrho}(z_0)}\left|\frac{u-(u)_{\mathcal{S}^{\la}_{2\varrho}}}{\lambda^{-1+\mu}2\varrho}\right|^2 dz \\
&\apprle \left[\varepsilon (\lambda^p+b(z_0)\lambda^s )+ c(\varepsilon) \left(\fiint_{\mathcal{S}^{\lambda}_{2\varrho}(z_0)}|\nabla u|^{\theta p}+a^{\theta}(z)|\nabla u|^{\theta q}+b(z)^{\theta}|\nabla u|^{\theta s}\,dz\right)^{\frac{1}{\theta}}\right]    
\end{align*}
whenever $\mathcal{S}^{\la}_{4\varrho}(z_0)\subset \Omega_T.$
\end{lemma}
Finally, the following lemma is the reverse H\"{o}lder inequality for $(p,s)$-phase.
\begin{lemma}\label{RHPS}
Let $u$ be a weak solution to \eqref{main_eqn}. Moreover, we assume that the energy estimate given in \cref{scl_energy_ps} and \cref{assmps1a} hold for $(p,s)$-intrinsic cylinders. Then there exists a constant $c=c(\textnormal{\texttt{data}})$ and $\theta_0 \in (0,1)$ such that for any $\theta \in (\theta_0, 1),$ we have
\begin{align*}
\fiint_{\mathcal{S}^{\lambda}_{\varrho}(z_0)}H(z, \nabla u)\, dz \leq c \left(\fiint_{\mathcal{S}^{\lambda}_{2\varrho}(z_0)}H(z, \nabla u)^{\theta }dz\right)^{\frac{1}{\theta}}.
\end{align*}
\end{lemma}
\begin{proof}
From the energy estimate \cref{scl_energy_ps} and the definition of $\rhoa, \rhob,$ we have
\begin{align}\label{EQUATION3.53}
    \fiint_{\mathcal{S}^{\la}_{\varrho}(z_0)} H(z, \nabla u) dz \apprle \fiint_{\mathcal{S}^{\lambda}_{2\varrho}(z_0)} H\left(z, \left|\frac{u-(u)_{\mathcal{S}^{\lambda}_{2\varrho}}}{\lambda^{-1+\mu}2\varrho}\right| \right)\,dz\nonumber\\
	 +(\lambda^{p-2}+b(z_0)\la^{s-2} )\fiint_{\mathcal{S}^{\lambda}_{2\varrho}(z_0)}\left|\frac{u-(u)_{\mathcal{S}^{\lambda}_{2\varrho}}}{\lambda^{-1+\mu}2\varrho}\right|^2\,dz.
\end{align}
From \cref{LEMMA3.17}, we have
\begin{align*}
&\fiint_{\mathcal{S}^{\lambda}_{2\varrho}(z_0)} H\Bigg(z, \left|\frac{u-(u)_{\mathcal{S}^{\lambda}_{2\varrho}}}{\lambda^{-1+\mu}2\varrho}\right|\Bigg)\,dz \\
&\apprle (\la^p+b(z_0)\la^s)^{1-\theta}\fiint_{\mathcal{S}^{\lambda}_{2\varrho}}H\left(z, |\nabla u|\right)^{\theta}dz + \varepsilon (\la^p+b(z_0)\la^s)\\
&\overset{\text{Young's inequality}}{\leq} \varepsilon (\la^p+b(z_0)\la^s) + c \left(\fiint_{\mathcal{S}^{\la}_{2\varrho}(z_0)}H\left(z, |\nabla u|\right)^{\theta}dz\right)^{\frac{1}{\theta}}.
\end{align*}
Moreover, from \cref{LEM5.7}, we obtain
\begin{align*}
 \lambda^{p-2}+a(z_0)\la^{s-2} \fiint_{\mathcal{S}^{\lambda}_{2\varrho}(z_0)}\left|\frac{u-(u)_{\mathcal{S}^{\lambda}_{2\varrho}}}{\lambda^{-1+\mu}2\varrho}\right|^2\,dz &\apprle\varepsilon (\lambda^p+b(z_0)\lambda^s )+ c \left(\fiint_{\mathcal{S}^{\lambda}_{2\varrho}(z_0)}H(z, \nabla u)^{\theta}\,dz\right)^{\frac{1}{\theta}}.
\end{align*}
Substituting the above estimates in \eqref{EQUATION3.53}, we get
\begin{align*}
   \fiint_{\mathcal{S}^{\la}_{\varrho}(z_0)} H\left(z, |\nabla u|\right) dz \leq \varepsilon (\la^p+b(z_0)\la^s) + c \left(\fiint_{\mathcal{S}^{\la}_{2\varrho}(z_0)}H(z, \nabla u)^{\theta}dz\right)^{\frac{1}{\theta}}.
\end{align*}
The proof can be completed by absorbing $\varepsilon (\la^p+b(z_0)\la^s)$ on the left-hand side, which is allowed because of \textbf{ps3} of \cref{assmps1c}.
\end{proof}
\subsection{Reverse H\"{o}lder inequality for $(p, q, s)$-phase} In this subsection, we prove the reverse H\"{o}lder inequality for $(p,q, s)$-phase.
\begin{assumption}
In the $(p,q,s)$-phase, we assume that the following is satisfied:
\begin{enumerate}
    \item [\textbf{pqs1.}] $K\lambda^p <  a(z_0)\lambda^q,\,\,\, K\la^p< b(z_0)\la^s.$ \label{assmpqs1a}\\
    \item [\textbf{pqs2.}] $\fiint_{\mathcal{G}^{\lambda}_{\rhob}(z_0)}(|\nabla u|^p+a(z)|\nabla u|^q+b(z)|\nabla u|^s) dz < \lambda^p+a(z_0)\la^p+b(z_0)\la^s\,\,\, \text{holds for all} \,\,\,\rho \in (\rho_a, \rho_b].$ \label{assmpqs1b}\\
    \item [\textbf{pqs3.}] $\fiint_{\mathcal{G}^{\lambda}_{\rhob}(z_0)}(|\nabla u|^p+a(z)|\nabla u|^q+b(z)|\nabla u|^s) dz =\lambda^p+a(z_0)\la^q+b(z_0)\la^s.$ \label{assmpqs1c}\\
    \item [\textbf{pqs4.}]$\frac{b(z_0)}{2} \leq b(z) \leq 2b(z_0)$\,\,\, and\,\,\,  $\frac{b(z_0)}{2} \leq b(z) \leq 2b(z_0)\,\,\, \text{holds for every}\,\,\, z \in \mathcal{G}^{\lambda}_{\rhob}(z_0).$ \label{assmpqs1d}
\end{enumerate}
\end{assumption}
We shall start with parabolic Sobolev-Poincar\'{e} inequalities for $(p,q, s)$-phases.
\begin{lemma}\label{pqs_intrinsic poincare}
 Let $u$ be a weak solution to \eqref{main_eqn} and the \cref{assmpqs1a} is in force. Then for any $\theta\in \left(\max\left\{\frac{s-1}{p}, \frac{1}{p}\right\}, 1\right]$ and $\varepsilon \in (0,1),$ the estimates \ref{1_lem7.10}, \ref{2_lem7.10} and \ref{3_lem7.10} from \cref{pq_intrinsic poincare} hold whenever $\mathcal{G}^{\lambda}_{\rhob}(z_0) \subset \Omega_T.$
\end{lemma}
\begin{proof}
    We postpone the proof to \cref{appendix}.
\end{proof}
Now we state the energy estimate \eqref{general caccipoli} for the $(p,s)$-phase. We consider the following cutoff functions.

\noindent \textbf{Cut-off functions for $(p,q,s)$-phase.} On the other hand, for $(p,q, s)$-phase we use
\begin{align*}
&\eta:=\eta(x) \in C_c^{\infty}(B_{\la^{-1+\mu}\rhob}(x_0)), \quad  \eta \equiv 1 \,\,\text{on}\, B_{\la^{-1+\mu}\rhoa}(x_0), \quad 0\leq \eta \leq 1 \,\,\,\,\,\text{and}\,\,\,\,\,\, |\nabla \eta| \apprle \frac{1}{\la^{-1+\mu}(\rhob-\rhoa)}\\
&\zeta \in C^{\infty}_c\left(t_0-\frac{\la^{2\mu}}{\la^p+a(z_0)\la^q+b(z_0)\la^s}(\rhob-h_0)^2, t_0+\frac{\la^{2\mu}}{\la^p+a(z_0)\la^q+b(z_0)\la^s}(\rhob-h_0)^2\right),\,\,\,\,\, 0\leq \zeta \leq 1,\\
&\zeta \equiv 1 \,\,\text{on}\,\,\left(t_0-\tfrac{\la^{2\mu}}{\la^p + a(z_0)\la^q+b(z_0)\la^s}\rhoa^2,t_0+\tfrac{\la^{2\mu}}{\la^p +a(z_0)\la^q+ b(z_0)\la^s}\rhoa^2\right),
\,\,\text{and}\,\, |\partial_t \zeta| \apprle \tfrac{1}{\tfrac{\la^{2\mu}}{\la^p +a(z_0)\la^q+ b(z_0)\la^s}(\rhob-\rhoa)^2}.
\end{align*}
Then the energy estimate for the $(p,q,s)$-phase reads is as follows.
\begin{lemma}\label{scl_energy_pqs}
Let $u$ be a weak solution of \eqref{main_eqn}. Then we have the following energy estimate:
\begin{align*}
&(\lambda^{p-2} +a(z_0)\la^{q-2}+ b(z_0)\la^{s-2}) \sup_{I^{\lambda,(p,q,s)}_{\rhoa}(t_0)}\fint_{B^{\lambda}_{\rhoa}(x_0)}\left|\frac{u-(u)_{\mathcal{G}^{\lambda}_{\rhoa}}}{\lambda^{-1+\mu}\rhoa}\right|^2\,dx +\fiint_{\mathcal{G}^{\lambda}_{\rhoa}(z_0)}H\left(z, |\nabla u|\right) \,dz\\
&\apprle 
\fiint_{\mathcal{G}^{\lambda}_{\rhob}(z_0)} H\left(z, \left|\frac{u-(u)_{\mathcal{G}^{\lambda}_{\rhob}}}{\lambda^{-1+\mu}(\rhob-\rhoa)}\right|\right)\,dz
+ (\lambda^{p-2}+a(z_0)\la^{q-2}+b(z_0)\la^{s-2}) \fiint_{\mathcal{G}^{\lambda}_{\rhob}(z_0)}\left|\frac{u-(u)_{\mathcal{G}^{\lambda}_{\rhob}}}{\lambda^{-1+\mu}(\rhob-\rhoa)}\right|^2 \,dz.
\end{align*}
\end{lemma}
We omit the proof of the next lemma as it can be completed using similar arguments as previous.
\begin{lemma}\label{LEMMA3.22}
Let $u$ be a weak solution to \eqref{main_eqn} and let \cref{scl_energy_pqs} hold with the \cref{assmpqs1a} in force. Then there exists a constant $c=c(\textnormal{\texttt{data}})$ such that the following estimate
\begin{align*}
\left(\lambda^{p-2}+a(z_0)\la^{q-2}+b(z_0)\la^{s-2} \right)\sup_{I^{\lambda,(p,s)}_{2\varrho}(t_0)}\fint_{B^{\lambda}_{2\rho}(x_0)}\left|\frac{u-(u)_{\mcq^{\la}_{2\rho}}}{\lambda^{-1+\mu}2\varrho}\right|^2 dx \leq c \left(\lambda^p+a(z_0)\la^q+b(z_0)\la^s\right)
\end{align*}
holds whenever $\mathcal{G}^{\la}_{\rhob}(z_0)\subset \Omega_T.$
\end{lemma}
Now we prove the following lemma, which leads us to reverse H\"{o}lder inequality in this case.
\begin{lemma}\label{LEMMA3.23}
Let $u$ be a weak solution to \eqref{main_eqn} and  let \cref{scl_energy_pqs} hold with the \cref{assmpqs1a} in force. Then there exists a constant $c=c(\textnormal{\texttt{data}})$ such that for any $\theta, \varepsilon \in \left(\theta_0, 1\right)$ for some $\theta_0\in (0,1),$ we have
\begin{align*}
\fiint_{\mathcal{G}^{\lambda}_{2\varrho}(z_0)} H\left(z, \left|\frac{u-(u)_{\mathcal{G}^{\la}_{2\varrho}}}{\la^{-1+\mu}2\varrho}\right|\right)dz
&\leq c \left(\la^p+a(z_0)\la^q+b(z_0)\la^s\right)^{1-\theta}\fiint_{\mathcal{G}^{\lambda}_{2\varrho}(z_0)}H\left(z, |\nabla u|\right)^{\theta}\,dz \\
&+ \varepsilon (\la^p+b(z_0)\la^q+a(z_0)\la^s)
\end{align*}
whenever $\mathcal{G}^{\la}_{4\varrho}(z_0)\subset \Omega_T.$
\end{lemma}
\begin{proof}
The first term on the right hand side of the statement is estimated as
\begin{align}\label{EQUATION3.54}
\fiint_{\mathcal{G}^{\lambda}_{2\varrho}(z_0)} &\left|\frac{u-(u)_{\mathcal{G}^{\la}_{\rhob}}}{\lambda^{-1+\mu}2\varrho}\right|^p\,dz \leq \la^{(1-\theta)p}\left(\fiint_{\mathcal{S}^{\la}_{2\varrho}(z_0)}|\nabla u|^{\theta p}\,dz\right)+ \varepsilon \la^{p}.    
\end{align}
Using {\bf pqs4} and \cref{g_n} with $\sigma=q, \xi=\theta q, \vartheta=\theta, r=2$, we get an estimate on the second term.
\begin{align}\label{EQUATION3.55}
\fiint_{\mathcal{G}^{\lambda}_{2\varrho}(z_0)} a(z)\left|\frac{u-(u)_{\mathcal{G}^{\la}_{\rhob}}}{\lambda^{-1+\mu}2\varrho}\right|^q dz \leq c (a(z_0)\la^q)^{1-\theta}\fiint_{\mathcal{G}^{\la}_{2\varrho}(z_0)}a^{\theta}(z_0)|\nabla u|^{\theta q} dz + \varepsilon a(z_0)\la^q.    
\end{align}
Lastly, using {\bf pqs4} and \cref{g_n} with $\sigma=s, \xi=\theta s, \vartheta=\theta, r=2,$ we estimate the third term. 
\begin{align}\label{EQUATION3.56}
 \fiint_{\mathcal{G}^{\lambda}_{2\varrho}(z_0)} b(z)\left|\frac{u-(u)_{\mathcal{G}^{\la}_{\rhob}}}{\lambda^{-1+\mu}2\varrho}\right|^s dz \leq c (b(z_0)\la^s)^{1-\theta}\fiint_{\mathcal{G}^{\la}_{2\varrho}(z_0)}b^{\theta}(z_0)|\nabla u|^{\theta s} dz + \varepsilon b(z_0)\la^s.   
\end{align}
Combining the estimates \eqref{EQUATION3.54}, \eqref{EQUATION3.55} and \eqref{EQUATION3.56}, we conclude the proof.
\end{proof}
Since the proofs of next two lemmas will be similar as previous, we only state them below.
\begin{lemma}\label{LEMMA3.24}
Let $u$ be a weak solution to \eqref{main_eqn} and \cref{scl_energy_pqs} holds with the \cref{assmpqs1a} is in force. Then there exist constants $c=c(\textnormal{\texttt{data}})$ and $\theta_0 \in (0, 1)$ depending on $n,p,q$ such that for any $\theta, \varepsilon \in (\theta_0,1),$ we have
\begin{align*}
(\lambda^{p-2}+a(z_0)\la^{q-2}+b(z_0)\lambda^{s-2})&\fiint_{\mathcal{G}^{\lambda}_{2\varrho}(z_0)}\left|\frac{u-(u)_{\mathcal{G}^{\la}_{2\varrho}}}{\lambda^{-1+\mu}2\varrho}\right|^2 dz \\
&\apprle \left[\varepsilon (\lambda^p+a(z_0)\la^q+b(z_0)\lambda^s )+ c(\varepsilon) \left(\fiint_{\mathcal{G}^{\lambda}_{2\varrho}(z_0)}H(z, |\nabla u|)^{\theta}\,dz\right)^{\frac{1}{\theta}}\right]    
\end{align*}
whenever $\mathcal{G}^{\la}_{4\varrho}(z_0)\subset \Omega_T.$
\end{lemma}
Applying the above \cref{LEMMA3.22}, \cref{LEMMA3.23} and \cref{LEMMA3.24}, we deduce the reverse H\"{o}lder inequality for $(p,q,s)$-phase.
\begin{lemma}\label{RHPQS}
Let $u$ be a weak solution to \eqref{main_eqn}. Moreover, we assume that the energy estimate given in \cref{scl_energy_pqs} and \cref{assmpqs1a} hold for $(p,q, s)$-intrinsic cylinders. Then there exists a constant $c=c(\textnormal{\texttt{data}})$ and $\theta_0 \in (0,1)$ such that for any $\theta \in (\theta_0, 1),$ we have
\begin{align*}
\fiint_{\mathcal{G}^{\lambda}_{\varrho}(z_0)}H(z, \nabla u)\, dz \leq c \left(\fiint_{\mathcal{G}^{\lambda}_{2\varrho}(z_0)}H(z, \nabla u)^{\theta }dz\right)^{\frac{1}{\theta}}.
\end{align*}
\end{lemma}
\section{Gradient higher integrability}\label{sec4}
In this section, we prove the gradient higher integrability for the solutions to \eqref{main_eqn}. In particular, we prove \cref{main_thm}. Let us recall the notations  \ref{not4}, \cref{Notation} i.e.,
\[\frac{n}{n+2}<\mu\leq \min\left\{1, \frac{p}{2}\right\},\] and 
\ref{not5.5}, \cref{Notation} with $2r,$ i.e.,
\[\mcp_{2r}(z_0):=B_{2r}(x_0)\times (t_0-4r^2, t_0+4r^2).\]

\subsection{Existence of intrinsic cylinders on superlevel sets}\label{STA} 
We first show the existence of $p$-intrinsic and $(p,q)$-intrinsic cylinders on the superlevel sets. Let

\begin{align}\label{def_la_0}
\lambda_0^{(n+2)\mu-n}=\fiint_{\mcp_{2r}(z_0)}H(z, |\nabla u|) dz + 1,\,\,\, \text{and}\,\,\,\,\,\, \Lambda_0=\lambda_0^p+\sup_{z\in \mcp_{2r}(z_0)}a(z)\lambda_0^q+\sup_{z\in \mcp_{2r}(z_0)}b(z)\lambda_0^s.   
\end{align} 
For $\uprho \in [r, 2r],$ we define the superlevel sets

\begin{equation}\label{defn_E_La}
\begin{aligned}
&E_{\Lambda}=\left\{z \in \Omega_T \,\, \Big|\,\, H\left(z, |\nabla u|\right) > \Lambda\right\}\,\,\,\,\text{and}\,\,\,\,\\
&E_{\Lambda, \uprho}=E_{\Lambda}\cap \mcp_\uprho (z_0):=\left\{z \in \mcp_\uprho(z_0)\,\,\Big|\,\,H\left(z, |\nabla u|\right)>\Lambda\right\}.
\end{aligned}  
\end{equation}
Let $0<r\leq r_1<r_2\leq 2r$ and

\begin{align}\label{EQQ6.1}
\Lambda> \left(\frac{4\kappa r}{r_2-r_1}\right)^{\frac{s(n+2)}{(n+2)\mu-n}}\Lambda_0\,\,\,\,\, \text{for some}\,\,\,\, \kappa >1.
\end{align}

Let $\lambda_{\tz}>0$ be such that $\Lambda=\lambda^p_{\tz}+a(\tz)\lambda^q_{\tz}+b(\tz)\lambda^s_{\tz}$ for any $\tz \in E_{\Lambda, r_1}.$ We claim that 
\begin{align}\label{EQQ6.2}
  \lambda_{\tz} > \left(\frac{4 \kappa r}{r_2-r_1}\right)^{\frac{n+2}{(n+2)\mu-n}}\lambda_0.  
\end{align}
On the contrary, let us assume 
\[\lambda_{\tz} \leq  \left(\frac{4 \kappa r}{r_2-r_1}\right)^{\frac{n+2}{(n+2)\mu-n}}\lambda_0.\] Then we find
\begin{align*}
    \Lambda&=\lambda^p_{\tz}+a(\tz)\lambda^q_{\tz}+b(\tz)\la^s_{\tz}\\
    &\leq \left(\frac{4 \kappa r}{r_2-r_1}\right)^{\frac{p(n+2)}{(n+2)\mu-n}}\lambda^p_0+a(\tz)\left(\frac{4 \kappa r}{r_2-r_1}\right)^{\frac{q(n+2)}{(n+2)\mu-n}}\lambda^q_0+b(\tz)\left(\frac{4 \kappa r}{r_2-r_1}\right)^{\frac{s(n+2)}{(n+2)\mu-n}}\la^s_0\\
    &\leq \left(\frac{4 \kappa r}{r_2-r_1}\right)^{\frac{s(n+2)}{(n+2)\mu-n}}\left(\lambda^p_0+a(\tz)\lambda^q_0+b(\tz)\la^s_0\right)\\
    &\leq\left(\frac{4 \kappa r}{r_2-r_1}\right)^{\frac{s(n+2)}{(n+2)\mu-n}}\Lambda_0,
\end{align*}
which is a contradiction to \eqref{EQQ6.1}. In the third inequality above, we used that $\frac{4\kappa r}{r_2-r_1}\geq 1,$ and $p\leq q\leq s.$ Therefore for $\uprho \in \left[\frac{r_2-r_1}{2\kappa}, r_2-r_1\right),$ we have

\begin{align*}
    \fiint_{\mcp^{\lambda_{\tz}}_\uprho(\tz)}H(z,|\nabla u|)\, dz&=\frac{1}{\lambda^{2\mu-p+(\mu-1)n}_{\tz} \uprho^{n+2}}\iint_{\mcp^{\lambda_{\tz}}_\uprho(\tz)}H(z,|\nabla u|)\, dz\nonumber\\
    &\leq \lambda^{p-2\mu+(1-\mu)n}_{\tz}\left(\frac{2r}{\uprho}\right)^{n+2}\fiint_{\mcp_{2r}(z_0)}H(z,|\nabla u|) dz\nonumber\\
    &\leq \left(\frac{4\kappa r}{r_2-r_1}\right)^{n+2}\lambda^{p-2\mu+(1-\mu)n}_{\tz}\fiint_{\mcp_{2r}(z_0)}H(z,|\nabla u|)\, dz.
\end{align*}
In the last inequality, we used the choice of $\uprho$ given above. Further estimating the right hand side, we get
\begin{align}\label{EQQ6.3}
\fiint_{\mcp^{\lambda_{\tz}}_\uprho(\tz)}H(z,|\nabla u|)\, dz  &\overset{\eqref{def_la_0}}{\leq} \left(\frac{4\kappa r}{r_2-r_1}\right)^{n+2}\lambda^{p-2\mu+(1-\mu)n}_{\tz} \lambda^{(n+2)\mu-n}_0\nonumber\\
&\overset{\eqref{EQQ6.2}}{<} \left(\frac{\lambda_{\tz}}{\lambda_0}\right)^{(n+2)\mu-n}\lambda^{p-2\mu+(1-\mu)n}_{\tz} \lambda^{(n+2)\mu-n}_0=\lambda^p_{\tz}.
\end{align}
Since $\tz \in E_{\Lambda, r_1}$ and $\Lambda=\lambda^p_{\tz}+a(\tz)\lambda^q_{\tz}+b(\tz)\la^s_{\tz}$ for any $\tz \in E_{\Lambda, r_1},$ we conclude that $\tz \in E_{\lambda^p_{\tz}, r_1}.$ Now for each $\tz \in E_{\Lambda, r_1},$ by Lebesgue differentiation theorem we have
\begin{align}\label{EQQ6.4}
\lim_{\uprho \to 0}\fiint_{\mcp^{\lambda_{\tz}}_{\uprho}(\tz)}|H(z,|\nabla u|)\,dz = |\nabla u(\tz)|^p+a(\tz)|\nabla u(\tz)|^q+b(\tz)|\nabla u(\tz)| > \lambda^p_{\tz}.
\end{align}
Now from \eqref{EQQ6.3} and \eqref{EQQ6.4} we conclude that there exists a radius $\uprho_{\tz} \in \left(0, \frac{r_2-r_1}{2\kappa}\right)$ such that 
\begin{align}\label{EQQ6.5}
\fiint_{\mcp^{\lambda_{\tz}}_{\uprho_{\tz}}(\tz)}H(z,|\nabla u|)\, dz =\lambda^p_{\tz}
\end{align}
and for every $\uprho \in \left(\uprho_{\tz}, \frac{r_2-r_1}{2\kappa}\right]$ we have
\begin{align}\label{EQQ6.6}
\fiint_{\mcp^{\lambda_{\tz}}_{\uprho}(\tz)}H(z,|\nabla u|)\, dz <\lambda^p_{\tz}.   
\end{align}
Now proceeding as \eqref{EQQ6.3} and using \eqref{EQQ6.5} we get
\begin{align*}
\lambda^p_{\tz}=\fiint_{\mcp^{\lambda_{\tz}}_{\uprho_{\tz}}(\tz)}H(\tz,|\nabla u|)\, dz&=\frac{1}{\uprho^{n+2}_{\tz} \lambda^{2\mu-p+(\mu-1)n}_{\tz}}\iint_{\mcp^{\lambda_{\tz}}_{\uprho_{\tz}}(\tz)}H(\tz,|\nabla u|)\, dz\\
&\leq \frac{(2r)^{n+2}}{\uprho^{n+2}_{\tz} \lambda^{2\mu-p+(\mu-1)n}_{\tz}}\fiint_{\mcp_{2r}(z_0)}H(\tz,|\nabla u|)\,dz\\
& \overset{\eqref{def_la_0}}{<} \left(\frac{2r}{\uprho_{\tz}}\right)^{n+2} \lambda^{p-2\mu+(1-\mu)n}_{\tz}\lambda^{(n+2)\mu-n}_0.
\end{align*}
This implies 
\begin{align}\label{EQUATION4.9}
 \lambda_{\tz} < \left(\frac{2r}{\uprho_{\tz}}\right)^{\frac{n+2}{(n+2)\mu-n}} \lambda_0.
\end{align}
Now for $K$ defined in \eqref{new_k}, we may have the following cases:
\begin{align}
    &K\la^p_{\tz}\geq a(\tz)\la^q_{\tz},\label{EQUATION4.10}\\
    &K\la^p_{\tz}< a(\tz)\la^q_{\tz},\label{EQUATION4.11}\\
    &K\la^p_{\tz}\geq b(\tz)\la^s_{\tz},\label{EQUATION4.12}\\
    &K\la^p_{\tz}< b(\tz)\la^s_{\tz},\label{EQUATION4.13}\\
    &a(\tz) \geq 2[a]_{\alpha}\max\left\{( \lambda^{-1+\mu}_{\tz} \uprho_{\tz})^{\alpha}, \left(\frac{\la^\mu_{\tz}}{\sqrt{\la^p_{\tz}+a(\tz)\la^q_{\tz}}}\uprho_{\tz}\right)^{\alpha}\right\},\label{EQUATION4.14}\\
    &a(\tz) < 2[a]_{\alpha}\max\left\{( \lambda^{-1+\mu}_{\tz} \uprho_{\tz})^{\alpha}, \left(\frac{\la^\mu_{\tz}}{\sqrt{\la^p_{\tz}+a(\tz)\la^q_{\tz}}}\uprho_{\tz}\right)^{\alpha}\right\},\label{EQUATION4.15}\\
    &b(\tz) \geq 2[b]_{\beta}\max\left\{( \lambda^{-1+\mu}_{\tz} \uprho_{\tz})^{\beta}, \left(\frac{\la^\mu_{\tz}}{\sqrt{\la^p_{\tz}+b(\tz)\la^s_{\tz}}}\uprho_{\tz}\right)^{\beta}\right\},\label{EQUATION4.16}\\
    &b(\tz) < 2[b]_{\beta}\max\left\{( \lambda^{-1+\mu}_{\tz} \uprho_{\tz})^{\beta}, \left(\frac{\la^\mu_{\tz}}{\sqrt{\la^p_{\tz}+b(\tz)\la^s_{\tz}}}\uprho_{\tz}\right)^{\beta}\right\},\label{EQUATIONN4.17}\\
    &a(\tz) \geq 2[a]_{\alpha}\max\left\{( \lambda^{-1+\mu}_{\tz} \uprho_{\tz})^{\alpha}, \left(\frac{\la^\mu_{\tz}}{\sqrt{\la^p_{\tz}+a(\tz)\la^q_{\tz}+b(\tz)\la^s_{tz}}}\uprho_{\tz}\right)^{\alpha}\right\},\label{EQUATION4.18}\\
    &a(\tz) < 2[a]_{\alpha}\max\left\{( \lambda^{-1+\mu}_{\tz} \uprho_{\tz})^{\alpha}, \left(\frac{\la^\mu_{\tz}}{\sqrt{\la^p_{\tz}+a(\tz)\la^q_{\tz}+b(\tz)\la^s_{tz}}}\uprho_{\tz}\right)^{\alpha}\right\},\label{EQUATION4.19}\\
    &b(\tz) \geq 2[b]_{\beta}\max\left\{( \lambda^{-1+\mu}_{\tz} \uprho_{\tz})^{\beta}, \left(\frac{\la^\mu_{\tz}}{\sqrt{\la^p_{\tz}+a(\tz)\la^q_{\tz}+b(\tz)\la^s_{\tz}}}\uprho_{\tz}\right)^{\beta}\right\},\label{EQUATION4.20}\\
    &b(\tz) <2[b]_{\beta}\max\left\{( \lambda^{-1+\mu}_{\tz} \uprho_{\tz})^{\beta}, \left(\frac{\la^\mu_{\tz}}{\sqrt{\la^p_{\tz}+a(\tz)\la^q_{\tz}+b(\tz)\la^s_{\tz}}}\uprho_{\tz}\right)^{\beta}\right\}.\label{EQUATION4.21}
\end{align}
We consider the following five cases:
\begin{enumerate}[label=(\roman*),series=theoremconditions]
\item \label{1_exists_cylin}  \eqref{EQUATION4.10} and \eqref{EQUATION4.12} hold.
\item \label{2_exists_cylin}  \eqref{EQUATION4.11}, \eqref{EQUATION4.12} and \eqref{EQUATION4.14} hold.
\item \label{3_exists_cylin} \eqref{EQUATION4.10}, \eqref{EQUATION4.13} and \eqref{EQUATION4.16} hold.
\item \label{4_exists_cylin}   \eqref{EQUATION4.11},  \eqref{EQUATION4.13},  \eqref{EQUATION4.18} and  \eqref{EQUATION4.20} hold.
\item \label{5_exists_cylin} Otherwise.
\end{enumerate}
In the case of \ref{1_exists_cylin}, the existence of $p$-intrinsic cylinders defined in \cref{assmp1a} follows directly by setting $\tz=z_0,$ $\uprho_{\tz}=\rhoa, \frac{r_2-r_1}{2\kappa}=\rhob$ in \eqref{EQQ6.5} and \eqref{EQQ6.6}.

We will show that \ref{2_exists_cylin} implies the existence of $(p, q)$- cylinders defined in \cref{assmpq1a}. From \eqref{EQUATION4.11} of \ref{2_exists_cylin}, clearly we have $a(\tz)>0.$ Then by using \eqref{EQQ6.6} and $\mcq^{\lambda_{\tz}}_{\uprho}(\tz) \subset \mcp^{\lambda_{\tz}}_{\uprho}(\tz),$ for every $\uprho \in (\uprho_{\tz}, r_2-r_1)$ we have

\begin{align*}
    \fiint_{\mcq^{\lambda_{\tz}}_{\uprho}(\tz)}H(z, |\nabla u|)\, dz \leq \frac{\lambda^p_{\tz}+a(\tz)\lambda^q_{\tz}}{\lambda^p_{\tz}}\fiint_{\mcp^{\lambda_{\tz}}_{\uprho}(\tz)}H(z, |\nabla u|)\, dz < \lambda^p_{\tz}+a(\tz)\lambda^q_{\tz}.
\end{align*}
Recalling that $\tz \in E_{\Lambda, r_1}$ and again using Lebesgue differentiation theorem we find that there exists a $\tilde{\uprho}_{\tz,q} \in (0, \uprho_{\tz})$ such that
$\fiint_{\mcq^{\lambda_{\tz}}_{\tilde{\uprho}_{\tz,q}}(\tz)}H(z, |\nabla u|)\,dz =\lambda^p_{\tz}+a(\tz)\lambda^q_{\tz}$
and $\fiint_{\mcq^{\lambda_{\tz}}_{\uprho}(\tz)}H(z, |\nabla u|)\,dz <\lambda^p_{\tz}+a(\tz)\lambda^q_{\tz}$ for any $\uprho \in \left(\tilde{\uprho}_{\tz,q} , r_2-r_1\right).$ Moreover, by \eqref{EQQ6.5}

\begin{align*}
    \fiint_{\mcq^{\la_{\tz}}_{\uprho_{\tz}}(\tz)}H(z, |\nabla u|)\, dz \leq \frac{\la^p_{\tz}+a(\tz)\la^q_{\tz}}{\la^p_{\tz}}\fiint_{\mcp^{\la_{\tz}}_{\uprho_{\tz}}(\tz)}H(z, |\nabla u|)\,dz=\la^p_{\tz}+a(\tz)\la^q_{\tz}.
\end{align*}
Hence, we proved \textbf{pq2}-\textbf{pq3} of \cref{assmpq1a} by replacing $\tz=z_0, \rhoa=\tilde{\uprho}_{\tz, q}$ and $\uprho_{\tz}=\rhob.$ Now we shall show \textbf{pq4} of \cref{assmpq1d}. From \eqref{EQUATION4.14} of \ref{2_exists_cylin}, we have 
\begin{align*}
    &2[a]_{\alpha}\max\left\{( \lambda^{-1+\mu}_{\tz} \uprho_{\tz})^{\alpha}, \left(\frac{\la^\mu_{\tz}}{\sqrt{\la^p_{\tz}+a(\tz)\la^q_{\tz}}}\uprho_{\tz}\right)^{\alpha} \right\}\\
    &< a(\tz)
    \leq \inf_{z\in\mcq^{\la_{\tz}}_{\uprho_{\tz}}(\tz)}a(z)+[a]_{\alpha}\max\left\{( \lambda^{-1+\mu}_{\tz} \uprho_{\tz})^{\alpha}, \left(\frac{\la^\mu_{\tz}}{\sqrt{\la^p_{\tz}+a(\tz)\la^q_{\tz}}}\uprho_{\tz}\right)^{\alpha} \right\}.
\end{align*}
From the above inequality, we get
\begin{align*}
 [a]_{\alpha}\max\left\{( \lambda^{-1+\mu}_{\tz} \uprho_{\tz})^{\alpha}, \left(\frac{\la^\mu_{\tz}}{\sqrt{\la^p_{\tz}+a(\tz)\la^q_{\tz}}}\uprho_{\tz}\right)^{\alpha} \right\} \leq \inf_{z\in\mcq^{\la_{\tz}}_{\uprho_{\tz}}(\tz)}a(z).   
\end{align*}
Using these we conclude that 
\begin{align*}
    \sup_{z \in \mcq^{\la_{\tz}}_{\uprho_{\tz}}(\tz)}a(z) \leq 2 \inf_{z\in \mcq^{\la_{\tz}}_{\uprho_{\tz}}(\tz)}a(z).
\end{align*}
Hence we obtain
\begin{align*}
    \inf_{z \in \mcq^{\la_{\tz}}_{\uprho_{\tz}}(\tz)} a(z)\geq \frac{a(\tz)}{2}\,\,\, \text{and}\,\,\, \sup_{z \in \mcq^{\la_{\tz}}_{\uprho_{\tz}}(\tz)} a(z)\leq 2 a(\tz).
\end{align*}
This implies 
\begin{align*}
\frac{a(\tz)}{2}\leq a(z)\leq 2a(\tz)\,\,\,\, \text{for any}\,\,\,\,  z \in \mcq^{\la_{\tz}}_{\uprho_{\tz}}(\tz).
\end{align*}

Now we shall show that \ref{3_exists_cylin} implies the existence of $(p, s)$-cylinders defined in \cref{assmps1a}. From \eqref{EQUATION4.13} of \ref{3_exists_cylin}, clearly we have $b(\tz)>0.$ Then by using \eqref{EQQ6.6} and $\mathcal{S}^{\lambda_{\tz}}_{\uprho}(\tz) \subset \mcp^{\lambda_{\tz}}_{\uprho}(\tz),$ for every $\uprho \in (\uprho_{\tz}, r_2-r_1)$ we have
\begin{align*}
    \fiint_{\mathcal{S}^{\lambda_{\tz}}_{\uprho}(\tz)}H(z, |\nabla u|)\, dz \leq \frac{\lambda^p_{\tz}+b(\tz)\lambda^s_{\tz}}{\lambda^p_{\tz}}\fiint_{\mcp^{\lambda_{\tz}}_{\uprho}(\tz)}H(z, |\nabla u|)\, dz < \lambda^p_{\tz}+b(\tz)\lambda^s_{\tz}.
\end{align*}
We recall that $\tz \in E_{\Lambda, r_1}$ and again using Lebesgue differentiation theorem we find that there exists a $\tilde{\uprho}_{\tz,s} \in (0, \uprho_{\tz})$ such that
$\fiint_{\mathcal{S}^{\lambda_{\tz}}_{\tilde{\uprho}_{\tz,s}}(\tz)}H(z, |\nabla u|)\,dz =\lambda^p_{\tz}+b(\tz)\lambda^s_{\tz}$
and $\fiint_{\mathcal{S}^{\lambda_{\tz}}_{\uprho}(\tz)}H(z, |\nabla u|)\,dz <\lambda^p_{\tz}+b(\tz)\lambda^s_{\tz}$ for any $\uprho \in \left(\tilde{\uprho}_{\tz,s} , r_2-r_1\right).$ Moreover, using \eqref{EQQ6.5}, we have 
\begin{align*}    \fiint_{\mathcal{S}^{\la_{\tz}}_{\uprho_{\tz}}(\tz)}H(z, |\nabla u|)\, dz \leq \frac{\la^p_{\tz}+b(\tz)\la^s_{\tz}}{\la^p_{\tz}}\fiint_{\mcp^{\la_{\tz}}_{\uprho_{\tz}}(\tz)}H(z, |\nabla u|)\,dz=\la^p_{\tz}+b(\tz)\la^s_{\tz}.
\end{align*}
Hence, we proved \textbf{ps2}-\textbf{ps3} of \cref{assmps1a} by replacing $\tz=z_0, \rhoa=\tilde{\uprho}_{\tz, s}$ and $\uprho_{\tz}=\rhob.$ Now we shall show \textbf{ps4} of \cref{assmps1d}. From \eqref{EQUATION4.16} of \ref{3_exists_cylin}, we have 
\begin{align*}
    &2[b]_{\beta}\max\left\{( \lambda^{-1+\mu}_{\tz} \uprho_{\tz})^{\beta}, \left(\frac{\la^\mu_{\tz}}{\sqrt{\la^p_{\tz}+b(\tz)\la^s_{\tz}}}\uprho_{\tz}\right)^{\beta} \right\}\\
    &< b(\tz)
    \leq \inf_{z\in\mathcal{S}^{\la_{\tz}}_{\uprho_{\tz}}(\tz)}b(z)+[b]_{\beta}\max\left\{( \lambda^{-1+\mu}_{\tz} \uprho_{\tz})^{\beta}, \left(\frac{\la^\mu_{\tz}}{\sqrt{\la^p_{\tz}+b(\tz)\la^s_{\tz}}}\uprho_{\tz}\right)^{\beta} \right\}.
\end{align*}
From the above inequality, we get
\begin{align*}
 [b]_{\beta}\max\left\{( \lambda^{-1+\mu}_{\tz} \uprho_{\tz})^{\beta}, \left(\frac{\la^\mu_{\tz}}{\sqrt{\la^p_{\tz}+b(\tz)\la^s_{\tz}}}\uprho_{\tz}\right)^{\beta} \right\} \leq \inf_{z\in\mathcal{S}^{\la_{\tz}}_{\uprho_{\tz}}(\tz)}b(z).   
\end{align*}
Using these we conclude that 
\begin{align*}
    \sup_{z \in \mathcal{S}^{\la_{\tz}}_{\uprho_{\tz}}(\tz)}b(z) \leq 2 \inf_{z\in \mathcal{S}^{\la_{\tz}}_{\uprho_{\tz}}(\tz)}b(z).
\end{align*}
Hence we obtain
\begin{align*}
    \inf_{z \in \mathcal{S}^{\la_{\tz}}_{\uprho_{\tz}}(\tz)} b(z)\geq \frac{b(\tz)}{2}\,\,\, \text{and}\,\,\, \sup_{z \in \mathcal{S}^{\la_{\tz}}_{\uprho_{\tz}}(\tz)} b(z)\leq 2 b(\tz).
\end{align*}
This implies 
\begin{align*}
\frac{b(\tz)}{2}\leq b(z)\leq 2b(\tz)\,\,\,\, \text{for any}\,\,\,\,  z \in \mathcal{S}^{\la_{\tz}}_{\uprho_{\tz}}(\tz).
\end{align*}

Now we show that \ref{4_exists_cylin} implies the existence of $(p,q,s)$-cylinders defined in \cref{assmpqs1a}. Since $\mathcal{G}^{\la_{\tz}}_{\uprho}(\tz)\subset \mathcal{Q}^{\la_{\tz}}_{\uprho}(\tz)\cup \mathcal{S}^{\la_{\tz}}_{\uprho}(\tz)\subset \mathcal{P}^{\la_{\tz}}_{\uprho}(\tz),$ for every $\uprho \in \left(\min\{\tilde{\rho}_{\tz, q}, \tilde{\rho}_{\tz, s}\}, r_2-r_1\right),$ we have
\begin{align*}
    \fiint_{\mathcal{G}^{\la_{\tz}}_{\uprho}(\tz)}H(z, |\nabla u|)\,dz < \la^p_{\tz}+a(\tz)\la^q_{\tz}+b(\tz)\la^s_{\tz}
\end{align*}
Furthermore, for $\tz \in E_{\Lambda, r_1}$ and again using Lebesgue differentiation theorem we find that there exists a $\tilde{\uprho}_{\tz,q, s} \in (0, \min\{\tilde{\rho}_{\tz, q}, \tilde{\rho}_{\tz, s}\})$ such that
$\fiint_{\mathcal{G}^{\lambda_{\tz}}_{\tilde{\uprho}_{\tz,q, s}}(\tz)}H(z, |\nabla u|)\,dz =\lambda^p_{\tz}+a(\tz)\la^q_{\tz}+b(\tz)\lambda^s_{\tz}$
and $\fiint_{\mathcal{G}^{\lambda_{\tz}}_{\uprho}(\tz)}H(z, |\nabla u|)\,dz <\lambda^p_{\tz}+a(\tz)\la^q_{\tz}+b(\tz)\lambda^s_{\tz}$ for any $\uprho \in \left(\tilde{\uprho}_{\tz,q, s} , r_2-r_1\right).$ Moreover, using \eqref{EQQ6.5}, we have 
\begin{align*}    \fiint_{\mathcal{G}^{\la_{\tz}}_{\uprho_{\tz}}(\tz)}H(z, |\nabla u|)\, dz \leq \frac{\la^p_{\tz}+a(\tz)\la^q_{\tz}+b(\tz)\la^s_{\tz}}{\la^p_{\tz}}\fiint_{\mcp^{\la_{\tz}}_{\uprho_{\tz}}(\tz)}H(z, |\nabla u|)\,dz=\la^p_{\tz}+a(\tz)\la^q_{\tz}+b(\tz)\la^s_{\tz}.
\end{align*}
Hence, we proved \textbf{pqs2}-\textbf{pqs3} of \cref{assmpqs1a} by replacing $\tz=z_0, \rhoa=\tilde{\uprho}_{\tz, q,s}$ and $\uprho_{\tz}=\rhob.$ \textbf{pqs4} of \cref{assmpqs1d} can be proved similarly as previous.

Next, we will show that \ref{5_exists_cylin} cannot hold. We need to show three cases: \eqref{EQUATION4.11} and \eqref{EQUATION4.15} cannot coexist, \eqref{EQUATION4.13} and \eqref{EQUATIONN4.17} cannot coexist, and \eqref{EQUATION4.11}, \eqref{EQUATION4.13}, \eqref{EQUATION4.21} and \eqref{EQUATION4.19} cannot coexist. We start with the first case. We consider the following two subcases.

\noindent{\bf Case $\max\left\{( \lambda^{-1+\mu}_{\tz} \uprho_{\tz})^{\alpha}, \left(\frac{\la^\mu_{\tz}}{\sqrt{\la^p_{\tz}+a(\tz)\la^q_{\tz}}}\uprho_{\tz}\right)^{\alpha} \right\}=( \lambda^{-1+\mu}_{\tz} \uprho_{\tz})^{\alpha}$ :} In this case,
we have
\begin{align}\label{EQUATION8.8}
    K\la^p_{\tz}\leq 2[a]_{\alpha}( \lambda^{-1+\mu}_{\tz} \uprho_{\tz})^{\alpha} \la^q_{\tz}.
\end{align}
Note that from \eqref{EQQ6.5}, we have
\begin{align*}
\la^p_{\tz}=\fiint_{\mcp^{\la_{\tz}}_{\uprho_{\tz}}(\tz)}|\nabla u|^p+a(z)|\nabla u|^q dz \leq \frac{\la^{n-(n+2)\mu}_{\tz}\la^p_{\tz}}{\uprho^{n+2}_{\tz}|B_1|}\iint_{\Omega_T}|\nabla u|^p+a(z)|\nabla u|^q dz <\frac{K\la^{n-(n+2)\mu}_{\tz}\la^p_{\tz}}{\uprho^{n+2}_{\tz}\left(2[a]_{\alpha}\right)^{\frac{n+2}{\alpha}}}.
\end{align*}
This gives
\begin{align*}
    \uprho^{\alpha}_{\tz}< \la^{\alpha\left(\frac{n}{n+2}-\mu\right)}_{\tz}\frac{K^{\frac{\alpha}{n+2}}}{2[a]_{\alpha}}\leq \la^{\alpha\left(\frac{n}{n+2}-\mu\right)}_{\tz} \frac{K}{2[a]_{\alpha}}.
\end{align*}
Using the restriction on $q$ from \eqref{def_pq}, we get
\begin{align*}
&\uprho^{\alpha}_{\tz}\la^q_{\tz}< \la^{\alpha\left(\frac{n}{n+2}-\mu\right)}\la^p_{\tz}\la^{\frac{2\alpha}{n+2}}_{\tz}\frac{K}{2[a]_{\alpha}}=\la^{\alpha(1-\mu)}_{\tz}\la^p_{\tz}\frac{K}{2[a]_{\alpha}}\\
&\implies \left(\la^{-1+\mu}_{\tz}\uprho_{\tz}\right)^{\alpha}< \la^{p-q}_{\tz}\frac{K}{2[a]_{\alpha}}.
\end{align*}
Substituting the above estimate in \eqref{EQUATION8.8}, we obtain
\begin{align*}
K\la^p_{\tz}\leq 2[a]_{\alpha}( \lambda^{-1+\mu}_{\tz} \uprho_{\tz})^{\alpha} \la^q_{\tz} < 2[a]_{\alpha} \la^{p-q}_{\tz}\frac{K}{2[a]_{\alpha}}=K\la^p_{\tz}    
\end{align*}
which is a contradiction.

\noindent{\bf Case $\max\left\{( \lambda^{-1+\mu}_{\tz} \uprho_{\tz})^{\alpha}, \left(\frac{\la^\mu_{\tz}}{\sqrt{\la^p_{\tz}+a(\tz)\la^q_{\tz}}}\uprho_{\tz}\right)^{\alpha} \right\}=\left(\frac{\la^\mu_{\tz}}{\sqrt{\La}}\uprho_{\tz}\right)^{\alpha}$ :} In this case, we have
\begin{align*}
    K\la^p_{\tz} \leq 2[a]_{\alpha}\left(\frac{\la^\mu_{\tz}}{\sqrt{\la^p_{\tz}+a(\tz)\la^q_{\tz}}}\uprho_{\tz}\right)^{\alpha} \la^q_{\tz}<  2[a]_{\alpha}\left(\la^{\frac{2\mu-p}{2}}_{\tz}\uprho_{\tz}\right)^{\alpha}\la^q_{\tz}.
\end{align*}
Again using the restriction on $q$ from \eqref{def_pq} and \eqref{EQQ6.5}, we arrive at a contradiction as before. Thus, \eqref{EQUATION4.11} and \eqref{EQUATION4.15} cannot hold together. The non-coexistence of other possibilities can be shown in a similar manner.
\subsection{Vitali covering argument} \label{vitali_cover}
We start by denoting the following cylinders for any $z\in E_{\La, r_1},$
\begin{equation}\label{both intrinsic cylinders}
    Q^{\La}_z(z)=\begin{cases}
        \mcp^{\la_z}_{r_z}(z)\,\,\,\,\,\, \text{if}\,\,\, K\la^p_z \geq a(z) \la^q_z,\,\, K\la^p_z \geq b(z) \la^s_z\\
        \mcq^{\la_z}_{r_z}(z)\,\,\,\,\,\, \text{if}\,\,\, K\la^p_z < a(z) \la^q_z,\,\, K\la^p_z \geq b(z) \la^s_z\\
        \mathcal{S}^{\la_z}_{r_z}(z)\,\,\,\,\,\, \text{if}\,\,\, K\la^p_z \geq a(z) \la^q_z,\,\, K\la^p_z < b(z) \la^s_z\\
        \mathcal{G}^{\la_z}_{r_z}(z)\,\,\,\,\,\, \text{if}\,\,\, K\la^p_z < a(z) \la^q_z,\,\, K\la^p_z < b(z) \la^s_z
    \end{cases}
\end{equation}
and for $v, w \in E_{\La, r_1},$ we denote $\La=\la^p_i+a(i)\la^q_i+b(i)\la^s_i$ where $i \in \{v, w\}.$
Let us consider the collection of cylinders $\mathcal{F}$ that covers $E_{\La, r_1},$ i.e.,
\begin{align*}
    \mathcal{F}:=\left\{Q^{\La}_w(w)\,\,\big|\,\, w \in E_{\La, r_1}\right\}.
\end{align*}
Setting $R=\frac{r_2-r_1}{2\kappa},$ consider the following subcollection 
\begin{align*}
    \mathcal{F}_j:=\left\{Q^{\La}_w(w) \in \mathcal{F}\,\, \Big|\,\,\,\,\, \frac{R}{2^j} < r_w < \frac{R}{2^{j-1}}\right\},\,\,\,\,\,\, j\in \NN.
\end{align*}
We construct disjoint subcollections $\mathcal{I}_j \subset \mathcal{F}_j$ for $j \in \NN$ as follows. Let $\mathcal{I}_1$ be the maximal disjoint  subcollection in $\mathcal{F}_1.$ We note that from
\begin{align*}
    \lim_{\La\to \infty}\La|E_{\La}|\leq \lim_{\La \to \infty}\iint_{E_{\La}}H(z, |\nabla u|)\, dz=0,
\end{align*}
we have $|E_{\La}|<\infty.$ Using $|E_{\La, r_1}|\leq |E_{\La}|< \infty$ and \eqref{EQUATION4.9}, we conclude that the number of cylinders in $\mathcal{I}_1$ is finite. Suppose we have already chosen $\mathcal{I}_j \subset \mcf_j$ for $j=1,..,k-1.$ Then we construct $\mathcal{I}_k$ as
\begin{align*}
    \mathcal{I}_k=\left\{Q^{\La}_w(w)\in \mcf_k \,\, \Big|\,\,Q^{\La}_w(w)\cap Q^{\La}_v(v)=\emptyset\,\,\, \text{for every}\,\,\, Q^{\La}_v(v) \in \cup_{j=1}^{k-1}\mathcal{I}_j \right\}. 
\end{align*}
Therefore, \[\mathcal{I}=\cup_{j=1}^{\infty}\mathcal{I}_j\]
would be the maximal collection of pairwise disjoint cylinders in $\mcf.$ Now to prove Vitali's covering lemma, we need to show:
\begin{itemize}
    \item[1.] For any $Q^{\La}_w(w) \in \mcf,$ there exists $Q^{\La}_v(v) \in \mathcal{I}$ such that $Q^{\La}_w(w) \cap Q^{\La}_v(v) \neq \emptyset.$
    \item [2.] There exists a universal constant $\kappa>1$ such that $Q^{\La}_w(w) \subset \kappa Q^{\La}_v(v).$
\end{itemize}
To show the first assertion, fix some $Q^{\La}_w (w) \in \mcf.$ Then $Q^{\La}_w(w) \in \mcf_j$ for some $j \in \NN.$ Using the maximality of $\mathcal{I}_j,$ we find that there exists $Q^{\La}_v(v) \in \cup_{i=1}^{j}\mathcal{I}_i$ such that $Q^{\La}_w(w)\cap Q^{\La}_v(v)\neq \emptyset.$ 
We note that by the definition of $\mathcal{I}_i,$ $i=1, 2,...,j,$ it can be shown that the radii are comparable, i.e.,
\begin{align*}
    r_w \leq 2r_v,\,\,\,\, \text{and}\,\,\, r_v \leq 2r_w.
\end{align*}
Now we prove the second assertion, i.e., $Q^{\La}_w(w)\subset \kappa Q^{\La}_v(v)$ in the rest of this section.
We consider the following 16 cases:
\begin{enumerate}[label=(\roman*),series=theoremconditions]
    \item \label{i} $Q^{\La}_w(w)=\mcp^{\la_w}_{r_w}(w)$ and $Q^{\La}_v(v)=\mcp^{\la_v}_{r_v}(v),$

     \item \label{ii} $Q^{\La}_w(w)=\mcp^{\la_w}_{r_w}(w)$ and $Q^{\La}_v(v)=\mcq^{\lambda_v}_{r_v}(v),$
     
    \item \label{iii} $Q^{\La}_w(w)=\mcp^{\la_w}_{r_w}(w)$ and $Q^{\La}_v(v)=\mathcal{S}^{\lambda_v}_{r_v}(v),$
    \item \label{iv} $Q^{\La}_w(w)=\mcp^{\la_w}_{r_w}(w)$ and $Q^{\La}_v(v)=\mathcal{G}^{\lambda_v}_{r_v}(v),$
    \item \label{v} $Q^{\La}_w(w)=\mcq^{\lambda_w}_{r_w}(w)$ and $Q^{\La}_v(v)=\mcq^{\lambda_v}_{r_v}(v),$
    \item \label{vi} $Q^{\La}_w(w)=\mcq^{\lambda_w}_{r_w}(w)$ and $Q^{\La}_v(v)=\mcp^{\lambda_v}_{r_v}(v),$
    \item \label{vii} $Q^{\La}_w(w)=\mcq^{\la_w}_{r_w}(w)$ and $Q^{\La}_v(v)=\mathcal{S}^{\lambda_v}_{r_v}(v),$
    \item \label{viii} $Q^{\La}_w(w)=\mcq^{\la_w}_{r_w}(w)$ and $Q^{\La}_v(v)=\mathcal{G}^{\lambda_v}_{r_v}(v),$
     \item \label{ix} $Q^{\La}_w(w)=\mathcal{S}^{\la_w}_{r_w}(w)$ and $Q^{\La}_v(v)=\mathcal{S}^{\la_v}_{r_v}(v),$
    \item \label{x} $Q^{\La}_w(w)=\mathcal{S}^{\lambda_w}_{r_w}(w)$ and $Q^{\La}_v(v)=\mcp^{\lambda_v}_{r_v}(v),$
    \item \label{xi} $Q^{\La}_w(w)=\mathcal{S}^{\lambda_w}_{r_w}(w)$ and $Q^{\La}_v(v)=\mcq^{\lambda_v}_{r_v}(v),$
    \item \label{xii} $Q^{\La}_w(w)=\mathcal{S}^{\la_w}_{r_w}(w)$ and $Q^{\La}_v(v)=\mathcal{G}^{\lambda_v}_{r_v}(v),$
    \item \label{xiii} $Q^{\La}_w(w)=\mathcal{G}^{\la_w}_{r_w}(w)$ and $Q^{\La}_v(v)=\mathcal{G}^{\lambda_v}_{r_v}(v),$
    \item \label{xiv} $Q^{\La}_w(w)=\mathcal{G}^{\la_w}_{r_w}(w)$ and $Q^{\La}_v(v)=\mcp^{\lambda_v}_{r_v}(v),$
    \item \label{xv} $Q^{\La}_w(w)=\mathcal{G}^{\la_w}_{r_w}(w)$ and $Q^{\La}_v(v)=\mcq^{\lambda_v}_{r_v}(v),$
    \item \label{xvi} $Q^{\La}_w(w)=\mathcal{G}^{\la_w}_{r_w}(w)$ and $Q^{\La}_v(v)=\mathcal{S}^{\lambda_v}_{r_v}(v).$
\end{enumerate}
First, we show that $\la_v$ and $\la_w$ are comparable in intrinsic cylinders. We note that it is enough to prove the comparison of $\la_v, \la_w$ for \ref{i}, \ref{v}, \ref{ix} and \ref{xiii} cases.  First, we consider the case \ref{i}. By the construction of Vitali covering, we already have $\mcp^{\la_w}_{r_w}(w)\cap \mcp^{\la_v}_{r_v}(v)\neq \emptyset$ and $\frac{1}{2}r_v \leq r_w \leq 2r_v.$

\vspace{.3cm}
\noindent\textbf{Claim :} $\la_v$ and $\la_w$ are comparable, i.e.,
\begin{align*}
    \frac{1}{K}\la_v \leq \la_w \leq K\la_v.
\end{align*}
Let $v=(x_v, t_v)$ and $w=(x_w, t_w).$ Using the H\"{o}lder continuity of the coefficient $a(z),$ we get
\begin{align*}
    |a(v)-a(w)| &\leq [a]_{\alpha} \max\left\{|x_v-x_w|, |t_v-t_w|^{\frac{1}{2}}\right\}^{\alpha}\\
    & \apprle[a]_{\alpha} \max\left\{2\left(\la^{-1+\mu}_v+\la^{-1+\mu}_w\right)r_w, 2\left(\la^{2\mu-p}_v+\la^{2\mu-p}_w\right)^{\frac{1}{2}}r_w\right\}^{\alpha}.
\end{align*}
Let $\la_w \leq \la_v$ and we claim that $\la_v \leq K \la_w.$ On the contrary, let $\la_v> K \la_w.$ Using the assumption $\la_w \leq \la_v,$ we have

\begin{align*}
    |a(v)-a(w)|\apprle [a]_{\alpha} \max\left\{2\la^{-1+\mu}_w r_w, 2\la^{\frac{2\mu-p}{2}}_w r_w\right\}^{\alpha}.
\end{align*}
Similarly, the H\"{o}lder continuity of $b(z)$ implies

\begin{align*}
|b(v)-b(w)|\apprle [b]_{\beta} \max\left\{2\la^{-1+\mu}_w r_w, 2\la^{\frac{2\mu-p}{2}}_w r_w\right\}^{\beta}.   
\end{align*}
We consider the following two cases.

\noindent\textbf{Case $\max\left\{\la^{-1+\mu}_w r_w, \la^{\frac{2\mu-p}{2}}_w r_w\right\}^{\alpha}=\left(\la^{-1+\mu}_wr_w\right)^{\alpha}$ :} This case corresponds to 
\begin{align*}
    |a(v)-a(w)| \apprle [a]_{\alpha} \left(2\la^{-1+\mu}_w r_w\right)^{\alpha}
\end{align*}
and following \eqref{Estimates on raddi with k} we further have
\begin{align*}
    r^{\alpha}_w\la^{\alpha\left(\mu-\frac{n}{n+2}\right)}_w \leq \frac{K}{2[a]_{\alpha}}
\end{align*}
and 
\begin{align*}
r^{\beta}_w\la^{\beta\left(\mu-\frac{n}{n+2}\right)}_w \leq \frac{K}{2[b]_{\beta}} .  
\end{align*}
From the definition of $\La,$ we have

\begin{align*}
    \La=\la^p_w+a(w)\la^q_w+b(w)\la^s_w &\leq \la^p_w+[a]_{\alpha} 2^{\alpha} \la^{(-1+\mu)\alpha}_w r^{\alpha}_w \la^q_w\\
    &+ a(v)\la^q_w +[b]_{\beta} \max\left\{2\la^{-1+\mu}_w r_w, 2\la^{\frac{2\mu-p}{2}}_w r_w\right\}^{\beta} \la^s_w+b(v)\la^s_w \\
    &\overset{\eqref{bounds on maximum radius}}{\leq} \la^p_w+a(v)\la^q_w+K\la^p_w+b(v)\la^s_w.
\end{align*}
Now substituting $\la_v > K\la_w$ and $\la_w \leq \la_v$ in the above estimate, we obtain,

\begin{align*}
    \La< \frac{\la^p_v}{K^p}+a(v)\la^q_v+\frac{K\la^p_v}{K^p}+b(v)\la^s_v&=\left(\frac{1}{K^p}+\frac{1}{K^{p-1}}\right)\la^p_v+a(v)\la^q_v+b(v)\la^s_v \\
    &\leq \frac{2}{K^{p-1}}\la^p_v+a(v)\la^q_v+b(v)\la^s_v< \La=\la^p_v+a(v)\la^q_v+b(v)\la^s_v
\end{align*}
since $K > 2^{\frac{1}{p-1}}.$ This gives a contradiction.

\noindent\textbf{Case $\max\left\{\la^{-1+\mu}_w r_w, \la^{\frac{2\mu-p}{2}}_w r_w\right\}^{\alpha}=\left(\la^{\frac{2\mu-p}{2}}_wr_w\right)^{\alpha}$ :} This case corresponds to 
\begin{align*}
    |a(v)-a(w)| \apprle [a]_{\alpha}\left(2\la^{\frac{2\mu-p}{2}}_w r_w\right)^{\alpha}
\end{align*}
and using the range of $q,$ i.e.,
\begin{align*}
    q \leq p+ \alpha \left(\frac{p}{2}-\frac{n}{n+2}\right),
\end{align*}
we arrive at the same conclusion.

Next, letting $\la_v \leq \la_w,$ in a similar way as above, we can show that $\la_w \leq K \la_v$ by interchanging the role of $v$ and $w$ above. The cases \ref{ii}, \ref{iii}, \ref{iv} follow from case \ref{i}.

In case \ref{v}, let $ \mcq^{\la_w}_{r_w}(w)\cap \mcq^{\la_v}_{r_v}(v)\neq \emptyset.$ We claim $\frac{1}{K}\la_v\leq \la_w \leq K\la_{v}.$ On the contrary, let us assume that $\la_w > K\la_{v}.$ Using \eqref{EQUATION4.10} and the H\"{o}lder regularity of $b(z),$ we obtain

\begin{align*}
    \La=\lambda^p_{w}+a(w)\lambda^q_{w}+b(w)\la^{s}_w < \frac{\la^p_v}{K^p}+\frac{2a(v)\la^q_v}{K^q}+\frac{\la^p_v}{2K^{p-1}}+b(v)\la^s_v\\
    < \left(\frac{1}{K^p}+\frac{1}{2K^{p-1}}\right)\la^p_v+a(v)\la^p_v+b(v)\la^s_v\\
    \leq \frac{1}{K^{p-1}}\la^p_v+a(v)\la^q_v+b(v)\la^s_v<\La
\end{align*}
which is a contradiction. Similarly, we can show that $\la_{v} \leq K\la_w.$ The cases \ref{vi}, \ref{vii} and \ref{viii} follow from the case \ref{v}.

Now we consider case \ref{ix}. Let $\tz \in \mathcal{S}^{\la_w}_{r_w}(w)\cap \mathcal{S}^{\la_v}_{r_v}(v).$ In this case, we can use $\frac{b(\tz)}{2}\leq b(w)\leq 2b(\tz)$ and H\"{o}lder regularity of $a(z)$ to conclude that $\frac{1}{K}\la_v \leq \la_w\leq K \la_v.$ Furthermore, the cases \ref{x}, \ref{xi} and \ref{xii} follow from \ref{ix}.

In the case \ref{xiii}, for $\tz \in \mathcal{G}^{\la_w}_{r_w}(w)\cap \mathcal{G}^{\la_v}_{r_v}(v),$ we can use $\frac{b(\tz)}{2}\leq b(w)\leq 2b(\tz)$ and $\frac{a(\tz)}{2}\leq a(w)\leq 2a(\tz)$ to conclude the same. The rest of the cases follows from \ref{xiii}.

\noindent{\bf Space inclusion.} Let $w=(w_0, t_w)$ and $v=(v_0, t_v).$  By considering the space part, we have
\begin{align*}
    Q^{\La}_w(w)=B_{\lambda^{-1+\mu}_w r_w}(w_0)\,\,\,\text{and}\,\,\, 
    \kappa Q^{\La}_v(v)=B_{\lambda^{-1+\mu}_v \kappa r_v}(v_0).
\end{align*}
First we note that 
\[B_{\lambda^{-1+\mu}_w r_w}(w_0)\cap B_{\lambda^{-1+\mu}_vr_v}(v_0)\neq \phi.\]
Let $x \in B_{\lambda^{-1+\mu}_w r_w}(w_0).$ Then
\[|x-v_0| \leq |x-w_0|+|w_0-v_0| \leq 2 \lambda^{-1+\mu}_w r_w+ \lambda^{-1+\mu}_v r_v.\]Using $K^{-1} \lambda_w \leq \lambda_v \leq K \lambda_w$ and $r_w \leq 2r_v,$ we conclude that the trailing inequality 
\begin{align*}
    2 \lambda^{-1+\mu}_w r_w+ \lambda^{-1+\mu}_v r_v \leq  \left(K^{1-\mu}+1\right) \lambda^{-1+\mu}_v r_v \leq \kappa\lambda^{-1+\mu}_v  r_v
\end{align*}
holds if 
\begin{align*}
 \kappa \geq 1+4 K^{1-\mu}.   
\end{align*}

\noindent{\bf Time inclusion.} Now we will show that the inclusion of the time intervals hold in each case.

\vspace{.3cm}
\noindent \underline{Situation \ref{i}:} $Q^{\La}_w(w)=\mcp^{\lambda_w}_{r_w}(w)$ and $Q^{\La}_v(v)=\mcp^{\lambda_v}_{r_v}(v).$ Considering the time part only, we have 
    \begin{align*}
    &Q^{\La}_w(w)=I^{\lambda_w, p}_{r_w}(t_w):=\left(t_w-\lambda^{2\mu-p}_w r_w^2, t_w+\lambda^{2\mu-p}_w r^2_w\right)\\
    &\kappa Q^{\La}_v(v)=I^{\lambda_v, p}_{\kappa r_v}(t_v):=\left(t_v-\lambda^{2\mu-p}_v \kappa r_v^2, t_v+\lambda^{2\mu-p}_v \kappa r_v^2\right).
    \end{align*}
    
    Let $t \in I^{\lambda_w, p}_{r_w}(t_w).$ Then
    \begin{align*}
        |t-t_v| \leq |t-t_w|+|t_v-t_w| \leq  2\lambda^{2\mu-p}_w r^2_w + \lambda^{2\mu-p}_vr^2_v &\leq 8 K^{p-2\mu} \lambda^{2\mu-p}_v r^2_v+ \lambda^{2\mu-p}_vr^2_v \\
        &\leq \kappa \lambda^{2\mu-p}_v r_v^2
    \end{align*}
    if 
   \begin{align*}
       \kappa \geq 1+8K^{p-2\mu}.
   \end{align*}

\vspace{.3cm}
\noindent \underline{Situation \ref{ii}:} $Q^{\La}_w(w)=\mcp^{\lambda_w}_{r_w}(w)$ and $Q^{\La}_v(v)=\mcq^{\lambda_v}_{r_v}(v).$ Considering the time intervals, 
\begin{align*}
  & Q^{\La}_w(w)=I^{\lambda_w, p}_{r_w}(t_w):=\left(t_w-\lambda^{2\mu-p}_w r^2_w, t_w+\lambda^{2\mu-p}_w r^2_w\right)\\
  &\kappa Q^{\La}_v(v)=I^{\lambda_v, (p,q)}_{\kappa r_v}(t_v)=\left(t_v-\frac{\lambda^{2\mu}_v \kappa r_v^2}{\la^p_v+a(v)\la^q_v}, t_v+\frac{\lambda^{2\mu}_v \kappa  r_v^2}{\la^p_v+a(v)\la^q_v}\right).
\end{align*}
Let $t \in I^{\lambda_w, p}_{r_w}(t_w).$ Then 
\begin{align*}
    |t-t_v| \leq |t-t_w|+|t_w-t_v| \leq 2\lambda^{2\mu-p}_w r^2_w+\frac{\lambda^{2\mu}_v}{\lambda^p_v+a(v)\lambda^q_v}r^2_v
\end{align*}
Note that, in this case we have $K \lambda^p_w\geq a(w) \lambda^q_w.$ This implies 
\begin{align*}
    \frac{1}{\lambda^p_w} \leq \frac{1+K}{\lambda^p_w+a(w)\lambda^q_w}.
\end{align*}
Therefore we have
\begin{align}\label{EQUATION4.17}
    \frac{2 \lambda^{2\mu}_w }{\lambda^p_w}r^2_w+\frac{\lambda^{2\mu}_v}{\lambda^p_v+a(v)\lambda^q_v}r^2_v &\leq \frac{2(1+K)\lambda^{2\mu}_w}{\lambda^p_w+a(w)\lambda^q_w}r^2_w+\frac{\lambda^{2\mu}_v}{\lambda^p_v+a(v)\lambda^q_v}r^2_v.
\end{align}
Now we claim
\begin{align*}
    \la^p_v+a(v)\la^q_v \apprle \la^p_w+a(w)\la^q_w.
\end{align*}
We use $\frac{1}{K}\la_w \leq \la_v\leq K \la_w$ and H\"{o}lder continuity of $a(z):$
\begin{align*}
    \la^p_v+a(v)\la^q_v&\leq K^p\la^p_w+|a(v)-a(w)|\la^q_v+a(w)K^q\la^q_w\\
    &\leq K^p\la^p_w+K\la^{p-q}_v\la^q_v+a(w)K^q\la^q_w\\
    &\leq (1+K)K^p\la^p_w+K^q a(w)\la^q_w\\
    &\leq (1+K)K^q\left(\la^p_w+a(w)\la^q_w\right).
\end{align*}
Now from \eqref{EQUATION4.17}, we get
\begin{align*}
    \frac{2 \lambda^{2\mu}_w }{\lambda^p_w}r^2_w+\frac{\lambda^{2\mu}_v}{\lambda^p_v+a(v)\lambda^q_v}r^2_v &\leq \frac{2(1+K)\lambda^{2\mu}_w}
    {\lambda^p_w+a(w)\lambda^q_w}r^2_w+\frac{\lambda^{2\mu}_v}{\lambda^p_v+a(v)\lambda^q_v}r^2_v\\
    &\leq \frac{2(1+K)^2K^q \la^{2\mu}_w}{\la^p_v+a(v)\la^q_v}r^2_w+\frac{\lambda^{2\mu}_v}{\lambda^p_v+a(v)\lambda^q_v}r^2_v\\
    &=\left(8(1+K)^2K^{q+2\mu}+1\right)\frac{\la^{2\mu}_v r^2_v}{\la^p_v+a(v)\la^q_v}
    \leq \kappa \frac{\la^{2\mu}_v r^2_v}{\la^p_v+a(v)\la^q_v}
\end{align*}
if 
\begin{align*}
    \kappa \geq \left(8(1+K)^2K^{q+2\mu}+1\right).
\end{align*}

\vspace{.3cm}
\noindent \underline{Situation \ref{iii}:} $Q^{\La}_w(w)=\mcp^{\lambda_w}_{r_w}(w)$ and $Q^{\La}_v(v)=\mathcal{S}^{\lambda_v}_{r_v}(v).$ \begin{align*}
  & Q^{\La}_w(w)=I^{\lambda_w, p}_{r_w}(t_w):=\left(t_w-\lambda^{2\mu-p}_w r^2_w, t_w+\lambda^{2\mu-p}_w r^2_w\right)\\
  &\kappa Q^{\La}_v(v)=I^{\lambda_v, (p,s)}_{\kappa r_v}(t_v)=\left(t_v-\frac{\lambda^{2\mu}_v \kappa r_v^2}{\la^p_v+b(v)\la^s_v}, t_v+\frac{\lambda^{2\mu}_v \kappa  r_v^2}{\la^p_v+b(v)\la^s_v}\right).
\end{align*}
Now using {\bf p1} of \cref{assmp1a} and following the same calculation as above, we find that $Q^{\La}_w(w)\subset \kappa Q^{\La}_v(v)$ if $k \geq \left(8(1+K)^2K^{s+2\mu}+1\right).$

\vspace{.3cm}
\noindent \underline{Situation \ref{iv}:} $Q^{\La}_w(w)=\mcp^{\lambda_w}_{r_w}(w)$ and $Q^{\La}_v(v)=\mathcal{G}^{\lambda_v}_{r_v}(v).$ \begin{align*}
& Q^{\La}_w(w)=I^{\lambda_w, p}_{r_w}(t_w):=\left(t_w-\lambda^{2\mu-p}_w r^2_w, t_w+\lambda^{2\mu-p}_w r^2_w\right)\\
&\kappa Q^{\La}_v(v)=I^{\lambda_v, (p,q,s)}_{\kappa r_v}(t_v)=\left(t_v-\frac{\lambda^{2\mu}_v \kappa r_v^2}{\la^p_v+a(v)\la^q_v+b(v)\la^s_v}, t_v+\frac{\lambda^{2\mu}_v \kappa  r_v^2}{\la^p_v+a(v)\la^q_v+b(v)\la^s_v}\right).
\end{align*} 
Following the same calculation as previous and using H\"{o}lder regularity of $a(z)$ and $b(z)$, we find that $Q^{\La}_w(w)\subset \kappa Q^{\La}_v(v)$ if $k \geq \left(8(2+2K)^2K^{s+2\mu}+1\right).$

\vspace{.3cm}
\noindent \underline{Situation \ref{v}:} $Q^{\La}_w(w)=\mcq^{\lambda_w}_{r_w}(w)$ and $Q^{\La}_v(v)=\mcq^{\lambda_v}_{r_v}(v).$ Considering the time part, we have 
    \begin{align*}
        &Q^{\La}_w(w)=I^{\lambda_w, (p,q)}_{r_w}(t_w)=\left(t_w-\frac{\lambda^{2\mu}_w}{\la^p_w+a(w)\la^q_w}r^2_w, t_w+\frac{\lambda^{2\mu}_w}{\la^p_w+a(w)\la^q_w}r^2_w\right)\\
        &\kappa Q^{\La}_v(v)=I^{\lambda_v, (p,q)}_{\kappa r_v}(t_v)=\left(t_v-\frac{\lambda^{2\mu}_v \kappa r_v^2}{\la^p_v+a(v)\la^q_v}, t_v+\frac{\lambda^{2\mu}_v \kappa r_v^2}{\la^p_v+a(v)\la^q_v}\right).
    \end{align*}
  Let $t \in I^{\lambda_w, (p,q)}_{r_w} (t_w).$ Then
  \begin{align*}
      |t-t_v| \leq |t-t_w|+|t_v-t_w| \leq \frac{2\la^{2\mu}_w}{\la^p_w+a(w)\la^q_w}r^2_w+\frac{\la^{2\mu}_v}{\la^p_v+a(v)\la^q_v}r^2_w 
  \end{align*} 
 Note that, in situation \ref{ii}, we already proved $\la^p_v+a(v)\la^q_v\leq (1+K)K^q (\la^p_w+a(w)\la^q_w)$ and we may use it here since $\mcq^{\la_w}_{r_w}(w)\subset \mcp^{\la_w}_{r_w}(w).$ This gives
 \begin{align*}
|t-t_v| \leq \frac{2\la^{2\mu}_w}{\la^p_w+a(w)\la^q_w}r^2_w+\frac{\la^{2\mu}_v}{\la^p_v+a(v)\la^q_v}r^2_v&=\left(8(1+K)K^{q+2\mu}+1\right)\frac{\la^{2\mu}_v}{\la^p_v+a(v)\la^q_v}r^2_v\\
&\leq \frac{\kappa\la^{2\mu}_v}{\la^p_v+a(v)\la^q_v}r^2_v
 \end{align*}
if $\kappa > \left(8(1+K)K^{q+2\mu}+1\right).$ 

\vspace{.3cm}
 \noindent \underline{Situation \ref{vi}:} $Q^{\La}_w(w)=\mcq^{\lambda_w}_{r_w}(w)$ and $Q^{\La}_v(v)=\mcp^{\lambda_v}_{r_v}(v).$ We consider the time intervals, i.e., 
  \begin{align*}
  &Q^{\La}_w(w)=I^{\lambda_w, (p,q)}_{r_w}(t_w)=\left(t_w-\frac{\lambda^{2\mu}_w}{\la^p_w+a(w)\la^q_w}r^2_w, t_w+\frac{\lambda^{2\mu}_w}{\la^p_w+a(w)\la^q_w}r^2_w\right),\\
  & \kappa Q^{\La}_v(v)=I^{\lambda_v, p}_{\kappa r_v}(t_v)=\left(t_v-\lambda^{2\mu-p}_v \kappa r_v^2, \lambda^{2\mu-p}_v \kappa r_v^2\right).
  \end{align*}
  Let $t \in I^{\lambda_w, (p,q)}_{r_w}(t_w).$ Then 
\begin{align*}
    |t-t_v| \leq |t-t_w|+|t_w-t_v| \leq \frac{2 \lambda^{2\mu}_w}{\la^p_w+a(w)\la^q_w}r^2_w+ \lambda^{2\mu-p}_v r^2_v.
\end{align*}
Now using $K \lambda^p_w \leq a(w)\lambda^q_w,$ we obtain
\begin{align*}
 \frac{2 \lambda^{2\mu}_w }{\lambda^p_w+a(w)\lambda^q_w}r^2_w+ \lambda^{2\mu-p}_v r^2_v  &\leq \frac{2 \lambda^{2\mu}_w}{(1+K)\lambda^p_w}r^2_w + \lambda^{2\mu-p}_v r^2_v\\
 &= \frac{2}{1+K} \lambda^{2\mu-p}_wr^2_w+ \lambda^{2\mu-p}_v r^2_v \\
 &\leq \frac{8}{1+K} K^{p-2\mu} \lambda^{2\mu-p}_vr^2_v+\lambda^{2\mu-p}_vr^2_v\\
 &\leq \kappa \lambda^{2\mu-p}_vr_v^2
\end{align*}
if 
\begin{align*}
    \kappa \geq \frac{8}{1+K}K^{p-2\mu}+1.
\end{align*}

\vspace{.3cm}
 \noindent \underline{Situation \ref{vii}:} $Q^{\La}_w(w)=\mcq^{\lambda_w}_{r_w}(w)$ and $Q^{\La}_v(v)=\mathcal{S}^{\lambda_v}_{r_v}(v).$ We consider the time intervals, i.e., 
  \begin{align*}
  &Q^{\La}_w(w)=I^{\lambda_w, (p,q)}_{r_w}(t_w)=\left(t_w-\frac{\lambda^{2\mu}_w}{\la^p_w+a(w)\la^q_w}r^2_w, t_w+\frac{\lambda^{2\mu}_w}{\la^p_w+a(w)\la^q_w}r^2_w\right),\\
  & \kappa Q^{\La}_v(v)=I^{\lambda_v, (p,s)}_{\kappa r_v}(t_v)=\left(t_v-\frac{\kappa\lambda^{2\mu}_v}{\la^p_v+b(v)\la^s_v}r^2_v, t_v+\frac{\kappa\lambda^{2\mu}_v}{\la^p_v+b(v)\la^s_v}r^2_v\right).
  \end{align*}
  Let $t \in I^{\lambda_w, (p,q)}_{r_w}(t_w).$ Then, using {\bf pq1} of \cref{assmpq1a} and {\bf ps1} of \cref{assmps1a}, we get
\begin{align*}
    |t-t_v| \leq |t-t_w|+|t_w-t_v| &\leq \frac{2 \lambda^{2\mu}_w}{\la^p_w+a(w)\la^q_w}r^2_w+ \frac{\lambda^{2\mu}_v}{\la^p_v+b(v)\la^s_v}r^2_v\\
    &\leq \frac{2\la^{2\mu}_w}{\la^p_w+b(w)\la^s_w} r^2_w+ \frac{\lambda^{2\mu}_v}{\la^p_v+b(v)\la^s_v} r^2_v.
\end{align*}
Again, following the calculations as above, we have $\la^p_v+b(v)\la^s_v\leq (1+K)K^s(\la^p_w+b(w)\la^s_w)$ and we conclude that $Q^{\La}_w(w)\subset \kappa Q^{\La}_v(v)$ if $k \geq 8(1+K)^2K^{s+2\mu}+1.$

\noindent \underline{Situation \ref{viii}:} $Q^{\La}_w(w)=\mcq^{\lambda_w}_{r_w}(w)$ and $Q^{\La}_v(v)=\mathcal{G}^{\lambda_v}_{r_v}(v).$ We consider the time intervals, i.e., 
  \begin{align*}
  &Q^{\La}_w(w)=I^{\lambda_w, (p,q)}_{r_w}(t_w)=\left(t_w-\frac{\lambda^{2\mu}_w}{\la^p_w+a(w)\la^q_w}r^2_w, t_w+\frac{\lambda^{2\mu}_w}{\la^p_w+a(w)\la^q_w}r^2_w\right),\\
  & \kappa Q^{\La}_v(v)=I^{\lambda_v, (p,q,s)}_{\kappa r_v}(t_v)=\left(t_v-\frac{\kappa\lambda^{2\mu}_v}{\la^p_v+a(v)\la^q_v+b(v)\la^s_v}r^2_v, t_v+\frac{\kappa\lambda^{2\mu}_v}{\la^p_v+a(v)\la^q_v+b(v)\la^s_v}r^2_v\right).
  \end{align*}
This situation is similar to situation \ref{iv} and $Q^{\La}_w(w)\subset \kappa Q^{\La}_v(v)$ if $k \geq \left(8(2+K)^2K^{s+2\mu}+1\right).$

\vspace{.2cm}
\noindent \underline{Situation \ref{ix}:} $Q^{\La}_w(w)=\mathcal{S}^{\lambda_w}_{r_w}(w)$ and $Q^{\La}_v(v)=\mathcal{S}^{\lambda_v}_{r_v}(v).$ We consider the time intervals, i.e., 
\begin{align*}
&Q^{\La}_w(w)=I^{\lambda_w, (p,s)}_{r_w}(t_w)=\left(t_w-\frac{\lambda^{2\mu}_w}{\la^p_w+b(w)\la^s_w}r^2_w, t_w+\frac{\lambda^{2\mu}_w}{\la^p_w+b(w)\la^s_w}r^2_w\right),\\
& \kappa Q^{\La}_v(v)=I^{\lambda_v, (p,s)}_{\kappa r_v}(t_v)=\left(t_v-\frac{\lambda^{2\mu}_v \kappa r_v^2}{\la^p_v+b(v)\la^s_v}, t_v+\frac{\lambda^{2\mu}_v \kappa  r_v^2}{\la^p_v+b(v)\la^s_v}\right).
\end{align*}
This situation is the same as the situation \ref{vii} with $\kappa \geq 8(1+K)^2K^{s+2\mu}+1.$
  
\vspace{.2cm}
\noindent \underline{Situation \ref{x}:} $Q^{\La}_w(w)=\mathcal{S}^{\lambda_w}_{r_w}(w)$ and $Q^{\La}_v(v)=\mathcal{P}^{\lambda_v}_{r_v}(v).$ We consider the time intervals, i.e., 
\begin{align*}
&Q^{\La}_w(w)=I^{\lambda_w, (p,s)}_{r_w}(t_w)=\left(t_w-\frac{\lambda^{2\mu}_w}{\la^p_w+b(w)\la^s_w}r^2_w, t_w+\frac{\lambda^{2\mu}_w}{\la^p_w+b(w)\la^s_w}r^2_w\right),\\
& \kappa Q^{\La}_v(v)=I^{\lambda_v, p}_{\kappa r_v}(t_v)=\left(t_v-\frac{\lambda^{2\mu}_v \kappa r_v^2}{\la^p_v}, t_v+\frac{\lambda^{2\mu}_v \kappa  r_v^2}{\la^p_v}\right).
\end{align*}
This situation is the same as the situation \ref{i} with $\kappa \geq 1+8K^{p-2\mu}.$

 \vspace{.2cm} 
\noindent \underline{Situation \ref{xi}:} $Q^{\La}_w(w)=\mathcal{S}^{\lambda_w}_{r_w}(w)$ and $Q^{\La}_v(v)=\mathcal{Q}^{\lambda_v}_{r_v}(v).$ We consider the time intervals, i.e., 
\begin{align*}
&Q^{\La}_w(w)=I^{\lambda_w, (p,s)}_{r_w}(t_w)=\left(t_w-\frac{\lambda^{2\mu}_w}{\la^p_w+b(w)\la^s_w}r^2_w, t_w+\frac{\lambda^{2\mu}_w}{\la^p_w+b(w)\la^s_w}r^2_w\right),\\
& \kappa Q^{\La}_v(v)=I^{\lambda_v, (p,q)}_{\kappa r_v}(t_v)=\left(t_v-\frac{\lambda^{2\mu}_v \kappa r_v^2}{\la^p_v+a(v)\la^q_v}, t_v+\frac{\lambda^{2\mu}_v \kappa  r_v^2}{\la^p_v+a(v)\la^q_v}\right).
\end{align*}
This situation is the same as the situation \ref{vii} with $\kappa \geq 8(1+K)^2K^{q+2\mu}+1.$

\vspace{.2cm}
\noindent \underline{Situation \ref{xii}:} $Q^{\La}_w(w)=\mathcal{S}^{\lambda_w}_{r_w}(w)$ and $Q^{\La}_v(v)=\mathcal{G}^{\lambda_v}_{r_v}(v).$ We consider the time intervals, i.e., 
\begin{align*}
&Q^{\La}_w(w)=I^{\lambda_w, (p,s)}_{r_w}(t_w)=\left(t_w-\frac{\lambda^{2\mu}_w}{\la^p_w+b(w)\la^s_w}r^2_w, t_w+\frac{\lambda^{2\mu}_w}{\la^p_w+b(w)\la^s_w}r^2_w\right),\\
& \kappa Q^{\La}_v(v)=I^{\lambda_v, (p,q,s)}_{\kappa r_v}(t_v)=\left(t_v-\frac{\lambda^{2\mu}_v \kappa r_v^2}{\la^p_v+a(v)\la^q_v+b(v)\la^s_v}, t_v+\frac{\lambda^{2\mu}_v \kappa  r_v^2}{\la^p_v+a(v)\la^q_v+b(v)\la^s_v}\right).
\end{align*}
This situation is same as situation \ref{viii} with $\kappa \geq 8(2+K)^2K^{s+2\mu}+1.$

\noindent \underline{Situation \ref{xiii}:} $Q^{\La}_w(w)=\mathcal{G}^{\lambda_w}_{r_w}(w)$ and $Q^{\La}_v(v)=\mathcal{G}^{\lambda_v}_{r_v}(v).$ We consider the time intervals, i.e., 
\begin{align*}
&Q^{\La}_w(w)=I^{\lambda_w, (p,q,s)}_{r_w}(t_w)=\left(t_w-\frac{\lambda^{2\mu}_w}{\la^p_w+a(w)\la^q_w+b(w)\la^s_w}r^2_w, t_w+\frac{\lambda^{2\mu}_w}{\la^p_w+a(w)\la^q_w+b(w)\la^s_w}r^2_w\right),\\
& \kappa Q^{\La}_v(v)=I^{\lambda_v, (p,q,s)}_{\kappa r_v}(t_v)=\left(t_v-\frac{\kappa\lambda^{2\mu}_v}{\la^p_v+a(v)\la^q_v+b(v)\la^s_v}r^2_v, t_v+\frac{\kappa\lambda^{2\mu}_v}{\la^p_v+a(v)\la^q_v+b(v)\la^s_v}r^2_v\right).
\end{align*}
This situation is similar to the situation \ref{v}. Indeed, using H\"{o}lder regularity of $a(z)$ and $b(z),$ one can show 
\begin{align*}
    \la^p_v+a(v)\la^q_v+b(v)\la^s_v\leq (2+K)K^s(\la^p_w+a(w)\la^q_w+b(w)\la^s_w)
\end{align*}
Following the calculations of situation \ref{v}, we find $Q^{\La}_w(w)\subset \kappa Q^{\La}_v(v)$ if $k\geq 8(2+K)K^{s+2\mu}+1.$

\noindent \underline{Situation \ref{xiv}:} $Q^{\La}_w(w)=\mathcal{G}^{\lambda_w}_{r_w}(w)$ and $Q^{\La}_v(v)=\mathcal{P}^{\lambda_v}_{r_v}(v).$ We consider the time intervals, i.e., 
  \begin{align*}
  &Q^{\La}_w(w)=I^{\lambda_w, (p,q,s)}_{r_w}(t_w)=\left(t_w-\frac{\lambda^{2\mu}_w}{\la^p_w+a(w)\la^q_w+b(w)\la^s_w}r^2_w, t_w+\frac{\lambda^{2\mu}_w}{\la^p_w+a(w)\la^q_w+b(w)\la^s_w}r^2_w\right),\\
  & \kappa Q^{\La}_v(v)=I^{\lambda_v, (p)}_{\kappa r_v}(t_v)=\left(t_v-\kappa\lambda^{2\mu-p}_v r^2_v, t_v+\kappa\lambda^{2\mu-p}_v r^2_v\right).
  \end{align*}
  This situation can be concluded from the situation \ref{i} with $\kappa \geq 1+8K^{p-2\mu}$
  
\noindent \underline{Situation \ref{xv}:} $Q^{\La}_w(w)=\mathcal{G}^{\lambda_w}_{r_w}(w)$ and $Q^{\La}_v(v)=\mathcal{Q}^{\lambda_v}_{r_v}(v).$ We consider the time intervals, i.e., 
  \begin{align*}
  &Q^{\La}_w(w)=I^{\lambda_w, (p,q,s)}_{r_w}(t_w)=\left(t_w-\frac{\lambda^{2\mu}_w}{\la^p_w+a(w)\la^q_w+b(w)\la^s_w}r^2_w, t_w+\frac{\lambda^{2\mu}_w}{\la^p_w+a(w)\la^q_w+b(w)\la^s_w}r^2_w\right),\\
  & \kappa Q^{\La}_v(v)=I^{\lambda_v, (p,q)}_{\kappa r_v}(t_v)=\left(t_v-\frac{\lambda^{2\mu}_v \kappa r_v^2}{\la^p_v+a(v)\la^q_v}, t_v+\frac{\lambda^{2\mu}_v \kappa  r_v^2}{\la^p_v+a(v)\la^q_v}\right).
  \end{align*}
  This situation is the same as situation \ref{v} with $\kappa \geq \left(8(1+K)K^{q+2\mu}+1\right). $

\noindent \underline{Situation \ref{xvi}:} $Q^{\La}_w(w)=\mathcal{G}^{\lambda_w}_{r_w}(w)$ and $Q^{\La}_v(v)=\mathcal{S}^{\lambda_v}_{r_v}(v).$ We consider the time intervals, i.e., 
  \begin{align*}
  &Q^{\La}_w(w)=I^{\lambda_w, (p,q,s)}_{r_w}(t_w)=\left(t_w-\frac{\lambda^{2\mu}_w}{\la^p_w+a(w)\la^q_w+b(w)\la^s_w}r^2_w, t_w+\frac{\lambda^{2\mu}_w}{\la^p_w+a(w)\la^q_w+b(w)\la^s_w}r^2_w\right),\\
  & \kappa Q^{\La}_v(v)=I^{\lambda_v, (p,s)}_{\kappa r_v}(t_v)=\left(t_v-\frac{\lambda^{2\mu}_v \kappa r_v^2}{\la^p_v+b(v)\la^s_v}, t_v+\frac{\lambda^{2\mu}_v \kappa  r_v^2}{\la^p_v+b(v)\la^s_v}\right).
  \end{align*}
  This situation is the same as situation \ref{v} with $\kappa \geq \left(8(1+K)K^{q+2\mu}+1\right)$.

Hence, combining all the above cases, it is enough to take \begin{align}\label{value of kappa}\kappa:=32(1+K)^2K^{s+2\mu}+1.\end{align} This completes the proof of 2, and hence the covering argument. We summarize the above content in the following lemma.
\begin{proposition}\label{vitali covering}
Let $\kappa$ be as  in \eqref{value of kappa} and $E_{\La, r_1}$ as in \eqref{defn_E_La}. There exists a collection $\{Q^{\Lambda}_{z_i}\}_{i\in \NN}$ of cylinders defined in \eqref{both intrinsic cylinders} that satisfies the following.
\begin{itemize}
\item [(i)] $\cup_{i\in\NN}\kappa Q^{\Lambda}_{z_i}=E_{\La, r_1}$.
\item[(ii)] $Q^{\La}_{z_i}\cap Q^{\La}_{z_j}= \emptyset$ for every $i,j\in \NN$ with $i\ne j$.
\end{itemize}
\end{proposition} 
\subsection{Proof of main theorem} In this section, we complete the proof of \cref{main_thm}. First, we derive the following lemma as a consequence of reverse H\"{o}lder inequalities.
\begin{lemma}\label{LEM6.1}
Let $u$ be a weak solution to \eqref{main_eqn}. Then there exist constants $c=c(\textnormal{\texttt{data}})$ and $\theta_0\in (0,1)$ such that for any $\theta \in (\theta_0, 1),$ we have
\begin{enumerate}[label=(\roman*),series=theoremconditions]
\item \label{1_LEM6.1}\begin{align*}
\iint_{\mcp^{\lambda}_{2\kappa\varrho}(z_0)}H(z, |\nabla u|)\, dz \leq c\Lambda^{1-\theta}\iint_{\mcp^{\lambda}_{2\varrho}(z_0) \cap E_{\La/c} }H(z, |\nabla u|)\, dz,
\end{align*}
\item \label{2_LEM6.1}\begin{align*}
\iint_{\mcq^{\lambda}_{2\kappa\varrho}(z_0)}H(z, |\nabla u|)\, dz \leq c\Lambda^{1-\theta}\iint_{\mcq^{\lambda}_{2\varrho}(z_0) \cap E_{\La/c}} H(z, |\nabla u|)\, dz,
\end{align*}
\item \label{3_LEM6.1}\begin{align*}
\iint_{\mathcal{S}^{\lambda}_{2\kappa\varrho}(z_0)}H(z, |\nabla u|)\, dz \leq c\Lambda^{1-\theta}\iint_{\mathcal{S}^{\lambda}_{2\varrho}(z_0) \cap E_{\La/c}} H(z, |\nabla u|)\, dz,
\end{align*}
\item \label{4_LEM6.1}\begin{align*}
\iint_{\mathcal{G}^{\lambda}_{2\kappa\varrho}(z_0)}H(z, |\nabla u|)\, dz \leq c\Lambda^{1-\theta}\iint_{\mathcal{G}^{\lambda}_{2\varrho}(z_0) \cap E_{\La/c}} H(z, |\nabla u|)\, dz,
\end{align*}
\end{enumerate}
whenever $\mcp^{\lambda}_{\rhob}(z_0),\, \mcq^{\lambda}_{\rhob}(z_0), \mathcal{S}^{\la}_{\rhob}(z_0), \mathcal{G}^{\la}_{\rhob}(z_0)  \subset \Omega_T.$
\end{lemma}
\begin{proof}
First we note that using \textbf{p2} of \cref{assmp1c}, we get
\begin{align*}
    \left(\fiint_{\mcp^{\lambda}_{2\varrho}(z_0)}H(z, |\nabla u|)^{\theta}\, dz\right)^{\frac{1-\theta}{\theta}} \leq \left(\fiint_{\mcp^{\lambda}_{2\varrho}(z_0)}H(z, |\nabla u|)\, dz\right)^{1-\theta}<\la^{p(1-\theta)}.
\end{align*}
Hence, we have
\begin{align}\label{EQQ6.32}
 \left(\fiint_{\mcp^{\lambda}_{2\varrho}(z_0)}H(z, |\nabla u|)^{\theta}\,dz\right)^{\frac{1}{\theta}}\leq \lambda^{p(1-\theta)}\fiint_{\mcp^{\lambda}_{2\varrho}(z_0)}H(z, |\nabla u|)^{\theta}\,dz.    
\end{align}
Now we write
\begin{align*}
\mcp^{\lambda}_{2\varrho}(z_0)= \left(\mcp^{\lambda}_{2\varrho} \cap E_{\frac{\lambda^p}{(2c)^{1/\theta}}}\right) \cup \left(\mcp^{\lambda}_{2\varrho} \setminus E_{\frac{\lambda^p}{(2c)^{1/\theta}}}\right).
\end{align*}
From the definition of $E_{\La}$ given in \eqref{defn_E_La}, the right hand side of \eqref{EQQ6.32} can be estimated as

\begin{align}\label{EQUATION8.11}
 \left(\fiint_{\mcp^{\lambda}_{2\varrho}(z_0)}H(z, |\nabla u|)^{\theta}\,dz\right)^{\frac{1}{\theta}}
 &\leq \lambda^{p(1-\theta)}\frac{\lambda^{p \theta}}{2c}\frac{\left|\mcp^{\lambda}_{2\varrho} \setminus E_{\frac{\lambda^p}{(2c)^{1/\theta}}}\right|}{\left|\mcp^{\lambda}_{2\varrho}\right|} + \frac{\lambda^{p(1-\theta)}}{\left|\mcp^{\lambda}_{2\varrho}\right|}\iint_{\mcp^{\lambda}_{2\varrho} \cap E_{\frac{\lambda^p}{(2c)^{1/\theta}}}} H(z, |\nabla u|)^\theta\, dz\nonumber\\
 &\leq \frac{\lambda^p}{2c} + \frac{\lambda^{p(1-\theta)}}{\left|\mcp^{\lambda}_{2\varrho}\right|}\iint_{\mcp^{\lambda}_{2\varrho} \cap E_{\frac{\lambda^p}{(2c)^{1/\theta}}}} H(z, |\nabla u|)^\theta\, dz.
\end{align}
Now from the reverse H\"{o}lder inequality (\cref{LEM5.4}) and \eqref{EQUATION8.11}, we obtain that
\begin{align}\label{EQUATION8.12}
    \fiint_{\mcp^{\lambda}_{\varrho}(z_0)} H(z, |\nabla u|)\, dz &\leq c\left(\fiint_{\mcp^{\lambda}_{2\varrho}(z_0)}H(z, |\nabla u|)^{\theta}\,dz\right)^{\frac{1}{\theta}}
    \leq \frac{\lambda^p}{2} + \frac{c\lambda^{p(1-\theta)}}{\left|\mcp^{\lambda}_{2\varrho}\right|}\iint_{\mcp^{\lambda}_{2\varrho} \cap E_{\frac{\lambda^p}{(2c)^{1/\theta}}}} H(z, |\nabla u|)^\theta\, dz.
\end{align}
Now we choose $\rhoa=\varrho,$ and from \textbf{p3} of \cref{assmp1c} we have
\begin{align*}
  \fiint_{\mcp^{\lambda}_{\varrho}(z_0)} H(z, |\nabla u|)\, dz=\la^p. 
\end{align*}
Therefore from \eqref{EQUATION8.12} and \cref{assmp1b} we get
\begin{align*}
   \frac{1}{2} \fiint_{\mcp^{\lambda}_{\rhob}(z_0)}H(z, |\nabla u|)\, dz <\frac{\lambda^p}{2} \leq \frac{c\lambda^{p(1-\theta)}}{\left|\mcp^{\lambda}_{2\varrho}\right|}\iint_{\mcp^{\lambda}_{2\varrho} \cap E_{\frac{\lambda^p}{(2c)^{1/\theta}}}} H(z, |\nabla u|)^\theta dz.
\end{align*}
This implies
\begin{align*}
    \fiint_{\mcp^{\lambda}_{\rhob}(z_0)}H(z, |\nabla u|)\, dz < \frac{2c\lambda^{p(1-\theta)}}{\left|\mcp^{\lambda}_{2\varrho}\right|}\iint_{\mcp^{\lambda}_{2\varrho} \cap E_{\frac{\lambda^p}{(2c)^{1/\theta}}}} H(z, |\nabla u|)^\theta\, dz.
\end{align*}
Now we choose $\rhob=2\kappa \varrho$ to get
\begin{align}\label{EQQ6.33}
    \iint_{\mcp^{\lambda}_{2\kappa \varrho}}H(z, |\nabla u|)\, dz \leq 2c\kappa^{n+2} \lambda^{p(1-\theta)}\iint_{\mcp^{\lambda}_{2\varrho}\cap E_{\frac{\lambda^p}{(2c)^{1/\theta}}}}H(z, |\nabla u|)^\theta dz.
\end{align}
Using $K \lambda^p \geq \frac{a(z_0)}{K}\lambda^q,$ we get $\lambda^p \geq \frac{\lambda^p+a(z_0)\lambda^q}{2K}=\frac{\Lambda}{2K}.$ Hence we have $\frac{\lambda^p}{(2c)^{1/\theta}}\geq \frac{\lambda^p}{(2c)^{1/ \theta_0}} \geq \frac{\Lambda}{2K(2c)^{1/ \theta_0}}.$ This gives that $E_{\frac{\lambda^p}{(2c)^{1/\theta}}} \subset E_{\frac{\Lambda}{2K(2c)^{1/\theta_0}}}.$ Now by setting $2K(2c)^{1/\theta_0}=c,$ from \eqref{EQQ6.33} we obtain
\begin{align*}
\iint_{\mcp^{\lambda}_{2\kappa \varrho}}H(z, |\nabla u|) \,dz \leq c\Lambda^{1-\theta} \iint_{\mcp^{\lambda}_{2\varrho}\cap E_{\Lambda/c}} H(z, |\nabla u|)^\theta dz.
\end{align*}
The same proof works for the $(p,q), (p, s)$ and $(p,q,s)$-intrinsic cylinders also. This completes the proof.
\end{proof}
\begin{remark}
Observe that in \cref{choice of radii}, we may consider $\rhob\in [2\varrho, 4\kappa \varrho],$ making the above selection of $\rhob$ admissible.
\end{remark}
Now we are ready to prove \cref{main_thm}. We present a detailed proof for the sake of completeness.
\begin{proof}[Proof of \cref{main_thm}]
Following \cref{vitali covering}, we have a countable pairwise disjoint collection $\mathcal{I}:=\left\{Q^{\La}_{z_j}\right\}_{j=1}^{\infty}$ for $z_j \in E_{\Lambda, r_1}.$ From the previous \cref{LEM6.1}, there exists constants $c$ and $\theta_0\in (0,1)$ we have
\begin{align*}
    \iint_{\kappa Q^{\La}_{z_j}} H(z, |\nabla u|)\, dz \leq c\Lambda^{1-\theta}\iint_{Q^{\La}_{z_j} \cap E_{\frac{\Lambda}{c}}}H(z, |\nabla u|)^{\theta}\, dz
\end{align*}
for every $j \in \NN$ and $\theta \in (\theta_0, 1).$ Since the cylinders in $\mathcal{I}$ are pairwise disjoint, we get
\begin{align}\label{EQQ6.34}
    \iint_{E_{\Lambda. r_1}} H(z, |\nabla u|)\, dz \leq \sum_{j=1}^{\infty} \iint_{\kappa Q^{\La}_{z_j}} H(z, |\nabla u|)\, dz 
    &\leq c\Lambda^{1-\theta} \sum_{j=1}^{\infty}\iint_{Q^{\La}_{z_j} \cap E_{\frac{\Lambda}{c}}}H(z, |\nabla u|)^{\theta} dz \nonumber\\
    &\leq c\Lambda^{1-\theta}\iint_{E_{\frac{\Lambda}{c}, r_2}}H(z, |\nabla u|)^{\theta} dz.
\end{align}
On the other hand, we have
\begin{align}\label{EQQ6.35}
    \iint_{E_{\frac{\Lambda}{c}, r_1}\setminus E_{\Lambda, r_1}}H(z, |\nabla u|)\, dz
    &= \iint_{E_{\frac{\Lambda}{c}, r_1}\setminus E_{\Lambda, r_1}}H(z, |\nabla u|)^\theta H(z, |\nabla u|)^{1-\theta}\, dz\nonumber\\
    &\leq \Lambda^{1-\theta}\iint_{E_{\frac{\Lambda}{c}}, r_2} H(z, |\nabla u|)^{\theta}\,dz.
\end{align}
Now combining \eqref{EQQ6.34} and \eqref{EQQ6.35} we conclude that 
\begin{align}\label{EQQ6.36}
    \iint_{E_{\frac{\Lambda}{c}, r_1}}H(z, |\nabla u|) dz&=\iint_{E_{\Lambda, r_1}}H(z, |\nabla u|) dz+\iint_{E_{\frac{\Lambda}{c}, r_1}\setminus E_{\Lambda, r_1}}H(z, |\nabla u|) dz\nonumber\\
    &\leq c\Lambda^{1-\theta}\iint_{E_{\frac{\Lambda}{c}, r_2}}H(z, |\nabla u|)^\theta dz.
\end{align}
For any $k\in \NN,$ we define the truncation operator as
\begin{align*}
    H(z, |\nabla u|)_k =\min \left\{H(z, |\nabla u|), k\right\}
\end{align*}
and the superlevel set as 
\begin{align*}
    E^k_{\Lambda, \uprho}:=\left\{z\in \mcp_{\uprho}(z_0)\,\,\Big|\,\, H(z, |\nabla u|)_k > \Lambda\right\}.
\end{align*}
From the above definition, we clearly have 
\begin{align}\label{EQQ6.37}
    E^k_{\Lambda, \uprho}=\begin{cases}
        E_{\Lambda, \uprho}\,\,\,\, &\text{if}\,\,\, \Lambda \leq k\\
        \emptyset\,\,\,\, &\text{if}\,\,\, \Lambda> k.
    \end{cases}
\end{align}
Therefore from \eqref{EQQ6.36} and \eqref{EQQ6.37}, we deduce that 
\begin{align*}
    \iint_{E^k_{\frac{\Lambda}{c}, r_1}}H(z, |\nabla u|)_k^{1-\theta}H(z, |\nabla u|)^\theta dz \leq \iint_{E_{\frac{\Lambda}{c}, r_1}}H(z, |\nabla u|) dz 
    &\leq c\Lambda^{1-\theta}\iint_{E_{\frac{\Lambda}{c}, r_2}}H(z, |\nabla u|)^\theta dz\\
    &\leq c\Lambda^{1-\theta}\iint_{E^k_{\frac{\Lambda}{c}, r_2}}H(z, |\nabla u|)^\theta dz.
\end{align*}
Recalling \eqref{EQQ6.1} we denote
\begin{align*}
\Lambda_1=\frac{1}{c}\left(\frac{4\kappa r}{r_2-r_1}\right)^{\frac{s(n+2)}{(n+2)\mu-n}}\Lambda_0.
\end{align*}
Then for any $\Lambda > \Lambda_1,$ we have
\begin{align}\label{EQQ6.38}
\iint_{E^k_{\Lambda, r_1}}H(z, |\nabla u|)_k^{1-\theta}H(z, |\nabla u|)^\theta dz \leq c\Lambda^{1-\theta}\iint_{E^k_{\Lambda, r_2}}H(z, |\nabla u|)^\theta dz.
\end{align}
Let $\varepsilon \in (0, 1)$ to be determined later. Multiplying the above inequality \eqref{EQQ6.38} with $\Lambda^{\varepsilon-1}$ and integrating over $(\Lambda_1, \infty),$ we obtain
\begin{align}\label{EQQ6.39}
\int_{\Lambda_1}^{\infty}\Lambda^{\varepsilon-1}\iint_{E^k_{\Lambda, r_1}}H(z, |\nabla u|)_k^{1-\theta}H(z, |\nabla u|)^\theta dz d\Lambda
\leq c\int_{\Lambda_1}^{\infty} \Lambda^{\varepsilon-\theta}\iint_{E^k_{\Lambda, r_2}}H(z, |\nabla u|)^\theta dz d\Lambda.
\end{align}
Now using Fubini's theorem on the left hand side of the above inequality \eqref{EQQ6.39}, we get

\begin{align*}
&\int_{\Lambda_1}^{\infty}\Lambda^{\varepsilon-1}\iint_{E^k_{\Lambda, r_1}}H(z, |\nabla u|)_k^{1-\theta}H(z, |\nabla u|)^\theta dz d\Lambda\\
&=\iint_{\mcp_{r_1}(z_0)}\int_{\Lambda_1}^{H(z, |\nabla u|)_k}\Lambda^{\varepsilon-1}H(z, |\nabla u|)_k^{1-\theta}H(z, |\nabla u|)^\theta d\Lambda dz\\
&=\frac{1}{\varepsilon}\iint_{\mcp_{r_1}(z_0)}H(z, |\nabla u|)_k^{1-\theta+\varepsilon}H(z, |\nabla u|)^\theta dz-\frac{\Lambda^{\varepsilon}_1}{\varepsilon}\iint_{\mcp_{r_1}(z_0)}H(z, |\nabla u|)_k^{1-\theta}H(z, |\nabla u|)^\theta dz\\
&\geq \frac{1}{\varepsilon}\iint_{\mcp_{r_1}(z_0)} H(z, |\nabla u|)_k^{1-\theta+\varepsilon}H(z, |\nabla u|)^\theta dz-\frac{\Lambda^{\varepsilon}_1}{\varepsilon}\iint_{\mcp_{2r}(z_0)}H(z, |\nabla u|)_k^{1-\theta}H(z, |\nabla u|)^\theta dz.
\end{align*}
Similarly applying Fubini's theorem on the right hand side of \eqref{EQQ6.39}, we get

\begin{align*}
&\int_{\Lambda_1}^{\infty} \Lambda^{\varepsilon-\theta}\iint_{E^k_{\Lambda, r_2}}H(z, |\nabla u|)^\theta dz d\Lambda = \iint_{\mcp_{r_2}(z_0)}\int_{\Lambda_1}^{H(z, |\nabla u|)_k}\Lambda^{\varepsilon-\theta}H(z, |\nabla u|)^\theta d \Lambda dz \\
&=\frac{1}{1+\varepsilon-\theta}\iint_{\mcp_{r_2}(z_0)}H(z, |\nabla u|)_k^{1-\theta+\varepsilon}H(z, |\nabla u|)^\theta dz-\frac{\Lambda^{1+\varepsilon-\theta}}{1+\varepsilon-\theta}\iint_{\mcp_{r_2}(z_0)}H(z, |\nabla u|)^\theta dz \\
&\leq \frac{1}{1+\varepsilon-\theta}\iint_{\mcp_{r_2}(z_0)}H(z, |\nabla u|)_k^{1-\theta+\varepsilon}H(z, |\nabla u|)^\theta dz.
\end{align*}
Plugging the above estimates in \eqref{EQQ6.39}, we obtain
\begin{align*}
    \iint_{\mcp_{r_1}(z_0)} H(z, |\nabla u|)_k^{1-\theta+\varepsilon}H(z, |\nabla u|)^\theta dz &\leq \Lambda^{\varepsilon}_1\iint_{\mcp_{2r}(z_0)}H(z, |\nabla u|)_k^{1-\theta}H(z, |\nabla u|)^\theta dz\nonumber \\
    &+ \frac{c\varepsilon}{1+\varepsilon-\theta} \iint_{\mcp_{r_2}(z_0)}H(z, |\nabla u|)_k^{1-\theta+\varepsilon}H(z, |\nabla u|)^\theta dz.
\end{align*}
Now we choose a small $\varepsilon < \varepsilon_0 < 1$ such that $\frac{c\varepsilon}{1+\varepsilon-\theta} \leq \frac{1}{2}.$ Now we aim to use \cref{iter_lemma}. For that purpose, we define
\begin{align*}
    h(r)=\iint_{\mcp_{r}(z_0)} H(z, |\nabla u|)_k^{1-\theta+\varepsilon}H(z, |\nabla u|)^\theta dz.
\end{align*}
Considering the choice of $\Lambda_1,$ we have

\begin{align*} 
&\Lambda^{\varepsilon}_1\iint_{\mcp_{2r}(z_0)}H(z, |\nabla u|)_k^{1-\theta}H(z, |\nabla u|)^\theta dz\\
&=\frac{1}{(r_2-r_1)^{\frac{\varepsilon s(n+2)}{(n+2)\mu-n}}}\frac{(4\kappa r)^{\frac{\varepsilon s(n+2)}{(n+2)\mu-n}}}{c^{\varepsilon}} \Lambda^{\varepsilon}_0\iint_{\mcp_{2r}(z_0)}H(z, |\nabla u|)_k^{1-\theta}H(z, |\nabla u|)^\theta dz\\
&:=\frac{A}{(r_2-r_1)^{\frac{\varepsilon s(n+2)}{(n+2)\mu-n}}}.
\end{align*}
Now applying \cref{iter_lemma} with $\varrho_1=r_1, \varrho_2=r_2, R=2r, \vartheta=\frac{c\varepsilon}{1+\varepsilon-\theta}, B=0$ and $\gamma_1=\frac{\varepsilon q(n+2)}{(n+2)\mu-n},$ we obtain
\begin{align*}
&\iint_{\mcp_{r}(z_0)} H(z, |\nabla u|)_k^{1-\theta+\varepsilon}H(z, |\nabla u|)^\theta dz \\
&\apprle_{(\vartheta,\gamma_1)}\frac{1}{r^{\frac{\varepsilon q(n+2)}{(n+2)\mu-n}}}\frac{(4\kappa r)^{\frac{\varepsilon q(n+2)}{(n+2)\mu-n}}}{c^{\varepsilon}} \Lambda^{\varepsilon}_0 \iint_{\mcp_{2r}(z_0)}H(z, |\nabla u|)_k^{1-\theta}H(z, |\nabla u|)^\theta dz.
\end{align*}
Letting $k \to \infty,$ we finally obtain

\begin{align*}
  \iint_{\mcp_{r}(z_0)} H(z, |\nabla u|)^{1+\varepsilon}\, dz \leq c \Lambda^{\varepsilon}_0 \iint_{\mcp_{2r}(z_0)}H(z, |\nabla u|)\,dz.
\end{align*}
Taking care of $\Lambda^{\varepsilon}_0$ we obtain

\begin{align*}
\fiint_{\mcp_{r}(z_0)} H(z, |\nabla u|)^{1+\varepsilon}\, dz \leq c\left(\fiint_{\mcp_{2r}(z_0)}H(z, |\nabla u|)\right)^{1+\frac{\varepsilon s}{(n+2)\mu-n}}.
\end{align*}
This completes the proof.
\end{proof}

\section{Remarks on general multi-phase problem}\label{general multiphase sec}
In this section, we present a brief discussion on the parabolic multi-phase problem with finitely many phases, omitting detailed proofs for brevity.  To be more precise, we wish to deal with the system:
\begin{align}\label{general multiphase}
    u_t-\operatorname*{div}\left(|\nabla u|^{p-2}\nabla u + \sum_{i=1}^m a_i(z)|\nabla u|^{p_i-2}\nabla u\right)=0
\end{align}
where $\frac{2n}{n+2}<p\leq p_1\leq \cdots \leq p_m <\infty$ and $0 \leq a_i(\cdot)\in C^{0,\alpha_i}(\Omega_T)$ with $\alpha_i \in (0,1]$ for all $i\in\{1,\cdots, m\}.$ As previous, we also assume

\begin{align*}
    \frac{2n}{n+2}<p\leq p_i\leq p+ \min\left\{\alpha_i\left(\frac{p}{2}-\frac{n}{n+2}\right), \frac{2\alpha_i}{n+2}\right\}<\infty\,\,\, \text{for all}\,\,\, i \in \{1, \cdots, m\}.
\end{align*}
Moreover, we denote
\begin{align}\label{New H}
    H(z, |\nabla u|):=|\nabla u|^p+\sum_{i=1}^m a_i(z)|\nabla u|^{p_i},
\end{align}
\begin{align}\label{new H tilde}
    \tilde{H}(z, |\nabla u|):=|\nabla u|^{p-1}+\sum_{i=1}^m a_i(z)|\nabla u|^{p_i-1},
\end{align}
With these notations in mind, we can derive a \textit{Caccioppoli inequality} following \cite[Lemma 3.1]{KO2024}.
\begin{lemma}\label{caccioppoli_multiphase}
  Let $u$ be a weak solution to \eqref{general multiphase}. Then there exists a constant $c=c(n, p, p_1, \cdots, p_m)$ such that

\begin{align*}
    &\sup_{t \in l_{S_2}(t_0)}\fint_{B_{R_2}(x_0)}\frac{\left|u-(u)_{Q_{R_2, S_2}}\right|^2}{S_2}\, dx+ \fiint_{Q_{R_2, S_2}(z_0)}H(z, |\nabla u|)\,dx\,dt\\
 &\apprle
 \fiint_{Q_{R_2, S_2}(z_0)} H\left(z,\left|\frac{u-(u)_{Q_{R_2, S_2}}}{R_2-R_1}\right|\right)\,dx\,dt
	 +  \fiint_{Q_{R_2, S_2}(z_0)}\frac{\left|u-(u)_{Q_{R_2, S_2}}\right|^2}{S_2-S_1}\,dx\,dt
\end{align*}
holds for every $Q_{R_2, S_2}(z_0)\subset \Omega_T,$ where $H(z, \cdot)$ is given by \eqref{New H}.  
\end{lemma}
Similarly, one can deduce the parabolic Poincar\'{e} inequality as in \cite[Lemma 3.3]{KO2024}.
\begin{lemma}
Let $u$ be a weak solution to \eqref{general multiphase}. Then there exists a constant $c=c(\textnormal{\texttt{data}})$ such that

    \begin{align*}
        \fiint_{Q_{r, \varrho}(z_0)}\left|\frac{u-(u)_{Q_{r, \varrho}(z_0)}}{r}\right|^{\theta \upgamma}\, dz \leq c \fiint_{Q_{r, \varrho}(z_0)}|\nabla u|^{\theta \upgamma } dz+ c\left(\frac{\varrho}{r^2}\fiint_{Q_{r, \varrho}(z_0)}\tilde{H}(z, |\nabla u|)\, dz\right)^{\theta \upgamma }
    \end{align*}
    for any $Q_{r, \varrho}(z_0)\subset \Omega_T$ defined in \cref{Notation} \ref{general cylinder} with $\upgamma \in (1, p_m]$ and $\theta \in \left(\frac{1}{\upgamma}, 1\right],$ where $\tilde{H}(z, \cdot)$ is given by \eqref{new H tilde}.    
\end{lemma}

\vspace{.5cm}
\subsection{Intrinsic geometry} In this section, we describe the intrinsic geometry governed by equation \eqref{general multiphase}. We will start by defining the intrinsic cylinders. 
\subsubsection{Intrinsic cylinders} In this context, we will have $m+2$ types of intrinsic cylinders, namely,
\begin{enumerate}[label=(\roman*),series=theoremconditions]
    \item $\mcp^{\lambda}_{\varrho}(z_0)$ is used to denote $p$-intrinsic cylinders at $z_0=(x_0, t_0),$

\begin{align*}
\mcp^{\lambda}_{\varrho}(z_0):=B_{\lambda^{-1+\mu}\varrho}(x_0)\times \left(t_0-\frac{\lambda^{2\mu}\varrho^2}{\lambda^p}, t_0+\frac{\lambda^{2\mu}\varrho^2}{\lambda^p}\right)
=:B^\la_{\varrho}(x_0)\times I^{\la, p}_{\varrho}(t_0).
\end{align*}
\item For all $i \in \{1, \cdots, m\},$ $\mcp^{i, \lambda}_{\varrho}(z_0)$ is used to denote $(p, p_i)$-intrinsic cylinders at $z_0=(x_0, t_0),$

\begin{align*}
\mcp^{i,\lambda}_{\varrho}(z_0):=B_{\lambda^{-1+\mu}\varrho}(x_0)\times \left(t_0-\frac{\lambda^{2\mu}\varrho^2}{\lambda^p+a_i(z_0)\la^{p_i}}, t_0+\frac{\lambda^{2\mu}\varrho^2}{\lambda^p+a_i(z_0)\la^{p_i}}\right)
=:B^\la_{\varrho}(x_0)\times I^{\la, (p, p_i)}_{\varrho}(t_0).
\end{align*}
\item $\mathcal{M}^{\la}_{\rhob}(z_0)$ is used to denote $(p,p_1, \cdots, p_m)$-intrinsic cylinders at $z_0=(x_0, t_0),$

\begin{align*}
    \mathcal{M}^{\la}_{\rhob}(z_0):&=B_{\lambda^{-1+\mu}\varrho}(x_0) \times \left(t_0-\frac{\lambda^{2\mu}\varrho^2}{\lambda^p+\sum_{i=1}^m a_i(z_0)\la^{p_i}}, t_0+\frac{\lambda^{2\mu}\varrho^2}{\lambda^p+\sum_{i=1}^m a_i(z_0)\la^{p_i}}\right)\\
    &=:B^\la_{\varrho}(x_0)\times I^{\la, (p, p_1, \cdots, p_m)}_{\varrho}(t_0).
\end{align*}
\end{enumerate}
\subsubsection{Separation of phases} Here we will describe the separation of the phases. First, we set 
\begin{align*}
    K:=2+ \sum_{i=1}^m \frac{\left(2[a_i]_{\alpha_i}\right)^{\frac{n+2}{\alpha_i}}}{|B_1|}\iint_{\Omega_T}H(z, |\nabla u|)\, dz.
\end{align*}
We shall start with the $p$-phase. We define the $p$-phase as:
\begin{align}\label{p case for general mp}
\begin{cases}  
{\bf p1.}\,\, K\la^p \geq a_i(z_0)\la^{p_i}\,\,\, \text{for every}\,\,\, i\in \{1, \cdots, m\}.\\
{\bf p2.} \fiint_{\mcp^{\la}_{\rho}(z_0)} H(z, |\nabla u|)\,dz < \la^p,\,\,\, \text{holds for every}\,\,\, \rho \in (\rhoa, \rhob].\\
{\bf p3.} \fiint_{\mcp^{\la}_{\rhoa}(z_0)} H(z, |\nabla u|)\,dz = \la^p.
\end{cases}  
\end{align}
Next, we define the $(p, p_i)$-phase as:
\begin{align}\label{ppi case for general mp}
\begin{cases}  
{\bf ppi1.}\,\, K\la^p < a_i(z_0)\la^{p_i},\,\, K\la^p\geq a_j(z_0)\la^{p_j}\,\,\, \text{for all}\,\,\, i, j\in \{1, \cdots, m\}\,\,\, \text{and}\,\,\, i\neq j.\\
{\bf ppi2.} \fiint_{\mcp^{i, \la}_{\rho}(z_0)} H(z, |\nabla u|)\,dz < \la^p+a_i(z_0)\la^{p_i},\,\,\, \text{holds for all}\,\,\, \rho \in (\rhoa, \rhob].\\
{\bf ppi3.} \fiint_{\mcp^{i, \la}_{\rhoa}(z_0)} H(z, |\nabla u|)\,dz = \la^p+a_i(z_0)\la^{p_i}.\\
{\bf ppi4.} \frac{a_i(z_0)}{2} \leq a_i(z)\leq 2a_i(z_0)\,\,\, \text{holds for every}\,\,\, z\in \mcp^{i, \la}_{\rhoa}(z_0).
\end{cases} 
\end{align}
Finally, we define the $(p, p_1, \cdots, p_m)$-phase as:
\begin{align}\label{ppm case for general mp}
\begin{cases}  
{\bf pm1.}\,\, K\la^p < a_i(z_0)\la^{p_i},\,\, \text{for every}\,\,\, i \in \{1, \cdots, m\}.\\
{\bf pm2.} \fiint_{\mathcal{M}^{\la}_{\rho}(z_0)} H(z, |\nabla u|)\,dz < \la^p+\sum_{i=1}^m a_i(z_0)\la^{p_i},\,\,\, \text{holds for all}\,\,\, \rho \in (\rhoa, \rhob].\\
{\bf pm3.} \fiint_{\mathcal{M}^{\la}_{\rhoa}(z_0)} H(z, |\nabla u|)\,dz = \la^p+\sum_{i=1}^m a_i(z_0)\la^{p_i}.\\
{\bf pm4.} \frac{a_i(z_0)}{2} \leq a_i(z)\leq 2a_i(z_0)\,\,\, \text{holds for every}\,\,\, z\in \mathcal{M}^{\la}_{\rhoa}(z_0)\,\,\, \text{and for every}\,\,\, i\in \{1, \cdots, m\}.
\end{cases} 
\end{align}
\subsection{Parabolic Sobolev-Poincar\'{e} inequalities} In this section, we state the parabolic Sobolev-Poincar\'{e} inequalities for each of the phases. The proofs can be completed applying the classical H\"{o}lder's inequality and Young's inequality as in \cref{sec3}.
\subsubsection{$p$-phase} We state here the parabolic Sobolev-Poincar\'{e} inequalities for the $p$-phase, analogous to \cref{p_intrinsic poincare 2}. The first step in proving these inequalities involves obtaining an analogue of \cref{p_intrinsic poincare 1}. Due to the similarity and length of the proofs, we omit them here.
\begin{lemma} \label{LEMMA_GEN_MULTIPHSE_1}
Let $u$ be a weak solution to \eqref{general multiphase} and let the assumption \eqref{p case for general mp} be in force. Then for any $\theta\in \left(\max\left\{\frac{p_m-1}{p}, \frac{1}{p}\right\}, 1\right]$ and $\varepsilon \in (0,1),$ there exists a constant $c=c(\textnormal{\texttt{data}})$ such that the following estimates
\begin{enumerate}[label=(\roman*),series=theoremconditions]
\item \label{1_lem_general5.1} 
\begin{align*}
\fiint_{\mcp^\la_{\rhob}(z_0)}\left|\frac{u-(u)_{\mcp^{\la}_{\rhob}(z_0)}}{\la^{-1+\mu}\rhob}\right|^{\theta p} dz \leq c \fiint_{\mcp^{\la}_{\rhob}(z_0)}\left|\nabla u\right|^{\theta p}dz +\varepsilon \la^{\theta p},
\end{align*}
\item \label{2_lem_general5.1} 
\begin{align*}
    \fiint_{\mcp^{\la}_{\rhob}(z_0)}&\left|\frac{u-(u)_{\mcp^{\la}_{\rhob}(z_0)}}{\la^{-1+\mu}\rhob}\right|^{\theta p_i} dz \leq c \fiint_{\mcp^{\la}_{\rhob}(z_0)}\left|\nabla u\right|^{\theta p_i}dz+ c\la^{(p_i-p)\theta}\left(\fiint_{\mcp^{\la}_{\rhob}(z_0)}a_i^{\theta}(z)|\nabla u|^{\theta p_i}dz\right)\\
    &+c\underset{i\neq j}{\sum_{{j=1}}^m} \la^{(2-p)\theta p_i}\la^{(p-p_j)\frac{\theta p_i}{p_j}}\left(\fiint_{\mcp^{\la}_{\rhob}(z_0)}\inf_{z\in \mcp^{\la}_{\rhob}(z_0)}a^{\theta}_j(z)|\nabla u|^{\theta p_j}\,dz\right)^{\frac{p_i(p_j-1)}{p_j}}+ \varepsilon \la^{\theta p_i},
\end{align*}
\end{enumerate}
for all $i \in \{1, \cdots, m\},$ hold whenever $\mcp^{\la}_{\rhob}(z_0)\subset \Omega_T.$
\end{lemma}
\subsubsection{$(p, p_i)$-phase for all $i\in\{1,\cdots, m\}$} Here we state the parabolic Sobolev-Poincar\'{e} for $(p, p_i)_{i=1}^m$-phases, analogous to \cref{pq_intrinsic poincare} and \cref{ps_intrinsic poincare}. The approach to proving these inequalities is the following: for a fixed $i \in \{1, \cdots, m\},$ we use bounds on the $i$-th coefficient, specifically,  $\frac{a_i(z_0)}{2}\leq a_i(z)\leq 2a_i(z_0),$ along with the H\"{o}lder regularity of $a_j(\cdot)$ for $i\neq j.$ 
\begin{lemma}\label{LEMMA_GEN_MULTIPHSE_2}
Let $u$ be a weak solution to \eqref{general multiphase} and let the assumption \eqref{ppi case for general mp} be in force. Then for any $\theta\in \left(\max\left\{\frac{p_m-1}{p}, \frac{1}{p}\right\}, 1\right]$ and $\varepsilon \in (0,1),$ there exists a constant $c=c(\textnormal{\texttt{data}})$ such that the following estimate
\begin{align*}
\fiint_{\mcp^{i, \la}_{\rhob}(z_0)}\left|\frac{u-(u)_{\mcp^{i, \la}_{\rhob}(z_0)}}{\la^{-1+\mu}\rhob}\right|^{\theta p_i} dz \leq c \fiint_{\mcp^{i, \la}_{\rhob}(z_0)}\left|\nabla u\right|^{\theta p_i}dz +\varepsilon \la^{\theta p_i},
\end{align*}
for all $i \in \{1, \cdots, m\},$ holds whenever $\mcp^{i, \la}_{\rhob}(z_0)\subset \Omega_T.$
\end{lemma}
\subsubsection{$(p, p_1,\cdots p_m)$-phase} Now we state the parabolic Solbole-Poincar\'{e} for $(p, p_1, \cdots, p_m)$-phase. In this case, we can use the bound $\frac{a_i(z_0)}{2}\leq a_i(z)\leq 2a_i(z_0)$ for every $i \in \{1, \cdots, m\}.$
\begin{lemma}\label{LEMMA_GEN_MULTIPHSE_3}
Let $u$ be a weak solution to \eqref{general multiphase} and the assumption \eqref{ppm case for general mp} is in force. Then for any $\theta\in \left(\max\left\{\frac{p_m-1}{p}, \frac{1}{p}\right\}, 1\right]$ and $\varepsilon \in (0,1),$ there exists a constant $c=c(\textnormal{\texttt{data}})$ such that the following estimate
\begin{align*}
\fiint_{\mathcal{M}^{\la}_{\rhob}(z_0)}\left|\frac{u-(u)_{\mathcal{M}^{\la}_{\rhob}(z_0)}}{\la^{-1+\mu}\rhob}\right|^{\theta p_i} dz \leq c \fiint_{\mathcal{M}^{\la}_{\rhob}(z_0)}\left|\nabla u\right|^{\theta p_i}dz +\varepsilon \la^{\theta p_i},
\end{align*}
for all $i \in \{1, \cdots, m\},$ holds whenever $\mathcal{M}^{\la}_{\rhob}(z_0)\subset \Omega_T.$
\end{lemma}
\subsubsection{Reverse H\"{o}lder inequalities} Now we briefly sketch the road map to achieve reverse H\"{o}lder inequalities for each phase, namely, $p, (p, p_i)_{i=1}^m, (p, p_1, \cdots, p_m)$-phases. We need to estimate the terms in the \textit{Caccioppoli inequality} (\cref{caccioppoli_multiphase}). For the $p$-phase, we will have
\begin{enumerate} [label=(\roman*),series=theoremconditions]
\item \label{section 5_1}
\begin{align*}
\sup_{I^{\la}_{2\varrho}(t_0)}\fint_{B^{\la}_{2\varrho}(x_0)}\left|\frac{u-(u)_{\mcp^{\la}_{2\varrho}(z_0)}}{\la^{-1+\mu}2\varrho}\right|^2\, dx \leq c\la^2,
\end{align*}
\item \label{section 5_2}
\begin{align*}
\fiint_{\mcp^{\la}_{2\varrho}(z_0)}H\left(z, \left|\frac{u-(u)_{\mcp^{\la}_{2\varrho}(z_0)}}{\la^{-1+\mu}2\varrho}\right|\right)\, dz \leq c \la^{(1-\theta )p}\fiint_{\mcp^{\la}_{2\varrho}(z_0)}H(z, |\nabla u|)\,dz+ \varepsilon \la^p.
\end{align*} 
\end{enumerate}
The above two estimates can be obtained by using \cref{g_n}, \cref{iter_lemma} and \cref{LEMMA_GEN_MULTIPHSE_1}. Once we have \ref{section 5_1} and \ref{section 5_2}, the reverse H\"{o}lder inequality for the $p$-phase follows.

For the $(p, p_i)_{i=1}^m$-phases, similarly we can have \ref{section 5_1} and 
\begin{align*}
\fiint_{\mcp^{i,\la}_{2\varrho}(z_0)}H\left(z, \left|\frac{u-(u)_{\mcp^{i,\la}_{2\varrho}(z_0)}}{\la^{-1+\mu}2\varrho}\right|\right)\, dz \leq c (\la+a_i(z_0)\la^{p_i})^{(1-\theta )}\fiint_{\mcp^{i, \la}_{2\varrho}(z_0)}H(z, |\nabla u|)\,dz+ \varepsilon (\la^p+a_i(z_0)\la^{p_i})    
\end{align*}
for each $i \in \{1, \cdots, m\},$
which is again an application of \cref{g_n} and \cref{LEMMA_GEN_MULTIPHSE_2}. These estimates give the reverse H\"{o}lder inequality for $(p, p_i)_{i=1}^m$-phase for each $i \in \{1,\cdots, m\}.$

For $(p, p_1\cdots, p_m)$-phase, we will have \ref{section 5_1} and
\begin{align*}
\fiint_{\mathcal{M}^{\la}_{2\varrho}(z_0)}H\left(z, \left|\frac{u-(u)_{\mathcal{M}^{\la}_{2\varrho}(z_0)}}{\la^{-1+\mu}2\varrho}\right|\right)\, dz \leq c \La^{(1-\theta )}\fiint_{\mcp^{i, \la}_{2\varrho}(z_0)}H(z, |\nabla u|)\,dz+ \varepsilon \La 
\end{align*} 
where $\La:=\la^p+\sum_{i=1}^m\la^{p_i}.$ Applying the above estimate and \ref{section 5_1}, we can deduce reverse H\"{o}lder inequality for $(p, p_1, \cdots, p_m)$-phase.
\subsection{Covering lemma} In this section, we briefly describe the completion of the higher integrability result for \eqref{general multiphase}. The first step is to show the existence of intrinsic cylinders on the superlevel set
\begin{align*}
    E_{\La, \varrho}:=\left\{z \in \mcp_\uprho(z_0)\,\,\Big|\,\,H\left(z, |\nabla u|\right)>\Lambda\right\},
\end{align*}
where $H(z, \cdot)$ is defined in \eqref{New H} and $\La=\la^p+\sum_{i=1}^m \la^{p_i}.$ Analogous to \eqref{EQUATION4.10}-\eqref{EQUATION4.21} in \cref{STA}, we may have
\begin{align*}
    &K \la^p_{\tz}\geq a_i(\tz)\la^{p_i}_{\tz},\\
    &K\la^p_{\tz}<a_i(\tz)\la^{p_i}_{\tz},\\
    &a_i(\tz)\geq 2[a_i]_{\alpha_i}\max\left\{\left(\la^{-1+\mu}_{\tz}\uprho_{\tz}\right)^{\alpha_i}, \left(\frac{\la^{\mu}_{\tz}\uprho_{\tz}}{\sqrt{\la^p_{\tz}+a_i(\tz)\la^{p_i}_{\tz}}}\right)^{\alpha_i}\right\},\\
    & a_i(\tz)< 2[a_i]_{\alpha_i}\max\left\{\left(\la^{-1+\mu}_{\tz}\uprho_{\tz}\right)^{\alpha_i}, \left(\frac{\la^{\mu}_{\tz}\uprho_{\tz}}{\sqrt{\la^p_{\tz}+a_i(\tz)\la^{p_i}_{\tz}}}\right)^{\alpha_i}\right\},\\
    & a_i(\tz)\geq 2[a_i]_{\alpha_i}\max\left\{\left(\la^{-1+\mu}_{\tz}\uprho_{\tz}\right)^{\alpha_i}, \left(\frac{\la^{\mu}_{\tz}\uprho_{\tz}}{\sqrt{\la^p_{\tz}+\sum_{i=1}^m a_i(\tz)\la^{p_i}_{\tz}}}\right)^{\alpha_i}\right\},\\
    &a_i(\tz)< 2[a_i]_{\alpha_i}\max\left\{\left(\la^{-1+\mu}_{\tz}\uprho_{\tz}\right)^{\alpha_i}, \left(\frac{\la^{\mu}_{\tz}\uprho_{\tz}}{\sqrt{\la^p_{\tz}+\sum_{i=1}^m a_i(\tz)\la^{p_i}_{\tz}}}\right)^{\alpha_i}\right\},
\end{align*}
for $i=1, 2, \cdots, m.$ Now following the similar analysis and considering several cases carefully as in \cref{STA}, one can complete the proof of existence of cylinders on the set $E_{\La, \uprho}.$

The next step involves carrying out a Vitali-type covering argument, considering $(m+2)^2$ distinct cylinders as described in \cref{vitali_cover}.

To complete the proof of gradient integrability, one needs to prove a result analogous to \cref{LEM6.1} on each of the $(m+2)$ cylinders. The rest of the proof follows from the proof of \cref{main_thm}. The final result of the gradient higher integrability for \eqref{general multiphase} can be stated as:
\begin{theorem}
Let $u$ be a weak solution of \eqref{general multiphase} and $\mu$ be the constant defined in \cref{Notation} \ref{not4}. Then there exist positive constants $\varepsilon_0 = \varepsilon_0(\textnormal{\texttt{data}})\in (0,1)$ and $c=c(\textnormal{\texttt{data}})\geq 1$ such that the following holds:
	\begin{align*}
		\fiint_{\mcp_{r}(z_0)} H(z, |\nabla u|)^{1+\varepsilon}\, dz \leq c\left(\fiint_{\mcp_{2r}(z_0)}H(z, |\nabla u|)\,dz+1\right)^{1+\frac{\varepsilon p_m}{(n+2)\mu-n}},  
	\end{align*}
	for every $\mcp_{2r}(z_0)\subset \Omega \times (0, T)$ and any $\varepsilon \in (0, \varepsilon_0)$ where $\mcp_{(\cdot)}(z_0)$ is defined in \cref{Notation} \ref{not5.5}, $H(z, \cdot)$ is defined in \eqref{New H} and $\textnormal{\texttt{data}}:=\left(n, N, p_1,\cdots, p_m, \alpha_i, [a_i]_{\alpha_i}, ||a_i||_{L^{\infty}}(\Omega_T), ||H(z, |\nabla u|)||_{L^1(\Omega_T)}\right).$   
\end{theorem}
\begin{appendices}
  \crefalias{section}{appsec}
  \section{Appendix}\label{appendix}
In this section, we prove \cref{ps_intrinsic poincare} and \cref{pqs_intrinsic poincare}. Since the proofs are similar to \cref{pq_intrinsic poincare}, we only provide the sketch.
\begin{proof}[Proof of \cref{ps_intrinsic poincare}] We only present the proof of \ref{2_lem7.10}, that is,
\begin{align}\label{appendix_1st_equation}
    \fiint_{\mathcal{S}^{\la}_{\rhob}(z_0)}\left|\frac{u-(u)_{\mathcal{S}^{\la}_{\rhob}(z_0)}}{\la^{-1+\mu}\rhob}\right|^{\theta q}\,dz \leq c\fiint_{\mathcal{S}^{\la}_{\rhob}(z_0)}|\nabla u|^{\theta q}\, dz+\varepsilon \la^{\theta q}.
\end{align}
The proof of \ref{1_lem7.10} and \ref{3_lem7.10} can be completed similarly. The proof of \eqref{appendix_1st_equation} is based on the bounds of $b(z)$ in the  $(p,s)$-intrinsic cylinders, i.e., $\frac{b(z_0)}{2}\leq b(z)\leq 2b(z_0)$ for any $z \in \mathcal{S}^{\la}_{\rhob}(z_0)$ and the estimate obtained in {\bf Step 1} of \cref{p_intrinsic poincare 1}, namely,
\begin{align}\label{appendix_2nd_equation}
    \fiint_{\mathcal{S}^{\la}_{\rhob}(z_0)}a(z)|\nabla u|^{q-1}\leq c\left(\fiint_{\mathcal{S}^{\la}_{\rhob}(z_0)}|\nabla u|^{\theta p}\, dz\right)^{\frac{p-1}{\theta p}}+K\la^{p-q}\fiint_{\mathcal{S}^{\la}_{\rhob}(z_0)}|\nabla u|^{q-1}\, dz.
\end{align}
First, we use \cref{parabolic poincare for double phase} with $m=q$ and $(p,s)$-intrinsic cylinders.
\begin{align*}
    \fiint_{\mathcal{S}^{\la}_{\rhob}(z_0)}\left|\frac{u-(u)_{\mathcal{S}^{\la}_{\rhob}}}{\la^{-1+\mu}\rhob}\right|^{\theta q} \leq c \fiint_{\mathcal{S}^{\la}_{\rhob}(z_0)}|\nabla u|^{\theta q} dz + c \underbrace{\left(\frac{\la^2}{\la^p+b(z_0)\la^s} \fiint_{\mathcal{S}^{\la}_{\rhob}(z_0)}\tilde{H}(z, \nabla u)\right)^{\theta q}}_{I}.
\end{align*}
We estimate $I$ from the above inequality by using \eqref{appendix_2nd_equation} and {\bf ps4} of \cref{assmps1d}.

\begin{align}\label{EQUATION_app3.31}
    I &\apprle \la^{(2-p)\theta q} \left(\fiint_{\mathcal{S}^{\la}_{\rhob}(z_0)}|\nabla u|^{\theta q} dz\right)^{p-1}+ \la^{(2-s)\theta q}\left(\fiint_{\mathcal{S}^{\la}_{\rhob}(z_0)}|\nabla u|^{\theta q} dz\right)^{s-1}\nonumber\\
    &\left(\frac{\la^2}{\la^p+b(z_0)\la^s}\right)^{\theta q}\la^{(p-q)\theta q}\left(\fiint_{\mathcal{S}^{\la}_{\rhob}(z_0)}|\nabla u|^{q-1}\, dz\right)^{\theta q}+ \left(\frac{\la^2}{\la^p+b(z_0)\la^s}\right)^{\theta q}\left(\fiint_{\mathcal{S}^{\la}_{\rhob}(z_0)}|\nabla u|^{\theta p}\, dz\right)^{\frac{(p-1)q}{p}}.
\end{align}
\noindent \textbf{Case $p \geq 2$ :} In this case, the first term of \eqref{EQUATION_app3.31} can be estimated as

\begin{align*}
 \la^{(2-p)\theta q} \left(\fiint_{\mathcal{S}^{\la}_{\rhob}(z_0)}|\nabla u|^{\theta q} dz\right)^{p-1}&=\la^{(2-p)\theta q} \left(\fiint_{\mathcal{S}^{\la}_{\rhob}(z_0)}|\nabla u|^{\theta q} dz\right)^{p-2} \left(\fiint_{\mathcal{S}^{\la}_{\rhob}(z_0)}|\nabla u|^{\theta q} dz\right)\\
 &\leq c \la^{(2-p)\theta q}\left(\fiint_{\mathcal{S}^{\la}_{\rhob}(z_0)}|\nabla u|^{s} dz\right)^{\frac{\theta(p-2)q}{s}}\left(\fiint_{\mathcal{S}^{\la}_{\rhob}(z_0)}|\nabla u|^{\theta q} dz\right)\\
 &\leq  c\la^{(2-p)\theta q} \la^{q\theta (p-2)}\left(\fiint_{\mathcal{S}^{\la}_{\rhob}(z_0)}|\nabla u|^{\theta q} dz\right).
\end{align*}
In the last inequality above, we used
\begin{align}\label{s-bound}
    \fiint_{\mathcal{S}^{\la}_{\rhob}(z_0)}|\nabla u|^s\,dz\apprle \la^s.
\end{align}
Similarly, the second term of \eqref{EQUATION_app3.31} can be estimated as

\begin{align*}
  \la^{(2-s)\theta q}\left(\fiint_{\mathcal{S}^{\la}_{\rhob}(z_0)}|\nabla u|^{\theta q} dz\right)^{s-1}&=\la^{(2-s)\theta q}\left(\fiint_{\mathcal{S}^{\la}_{\rhob}(z_0)}|\nabla u|^{\theta q} dz\right)^{s-2}\left(\fiint_{\mathcal{S}^{\la}_{\rhob}(z_0)}|\nabla u|^{\theta q} dz\right)\\
  &\leq c \la^{(2-s)\theta q}\left(\fiint_{\mathcal{S}^{\la}_{\rhob}(z_0)}|\nabla u|^{s} dz\right)^{\frac{\theta(s-2)q}{s}}\left(\fiint_{\mathcal{S}^{\la}_{\rhob}(z_0)}|\nabla u|^{\theta q} dz\right)\\
  &\overset{\eqref{s-bound}}{\leq} c \la^{(2-q)\theta q}\la^{(q-2)\theta q}\left(\fiint_{\mathcal{S}^{\la}_{\rhob}(z_0)}|\nabla u|^{\theta q} dz\right).
\end{align*}
Now we estimate the third term in \eqref{EQUATION_app3.31}.
\begin{align*}  
&\left(\frac{\la^2}{\la^p+b(z_0)\la^s}\right)^{\theta q}\la^{(p-q)\theta q}\left(\fiint_{\mathcal{S}^{\la}_{\rhob}(z_0)}|\nabla u|^{q-1}\, dz\right)^{\theta q}\\
&\leq \left(\frac{\la^2}{\la^p+b(z_0)\la^s}\right)^{\theta q}\la^{(p-q)\theta q} \left(\fiint_{\mathcal{S}^{\la}_{\rhob}(z_0)}|\nabla u|^{s}\,dz\right)^{\frac{\theta(q-2)q}{s}}\left(\fiint_{\mathcal{S}^{\la}_{\rhob}(z_0)}|\nabla u|^{\theta q}\,dz\right)\\
&\overset{\eqref{s-bound}}{\apprle} \left(\frac{\la^2}{\la^p+b(z_0)\la^s}\right)^{\theta q}\la^{(p-q)\theta q} \la^{\theta q(q-2)}\left(\fiint_{\mathcal{S}^{\la}_{\rhob}(z_0)}|\nabla u|^{\theta q}\,dz\right) \apprle \left(\fiint_{\mathcal{S}^{\la}_{\rhob}(z_0)}|\nabla u|^{\theta q}\,dz\right). 
\end{align*}
It remains to estimate the fourth term of \eqref{EQUATION_app3.31}. Indeed,
\begin{align*}
&\left(\frac{\la^2}{\la^p+b(z_0)\la^s}\right)^{\theta q}\left(\fiint_{\mathcal{S}^{\la}_{\rhob}(z_0)}|\nabla u|^{\theta p}\, dz\right)^{\frac{(p-1)q}{p}}\\
& \leq \left(\frac{\la^2}{\la^p+b(z_0)\la^s}\right)^{\theta q} \left(\fiint_{\mathcal{S}^{\la}_{\rhob}(z_0)}|\nabla u|^p\,dz\right)^{ \frac{\theta q(p-2)}{p}}\left(\fiint_{\mathcal{S}^{\la}_{\rhob}(z_0)}|\nabla u|^{\theta p}\right)^{\frac{q}{p}}\\
&\leq \left(\frac{\la^2}{\la^p+b(z_0)\la^s}\right)^{\theta q}(\la^p+b(z_0)\la^s)^{\frac{\theta q(p-2)}{p}}\left(\fiint_{\mathcal{S}^{\la}_{\rhob}(z_0)}|\nabla u|^{\theta q}\right)\\
&=\left(\frac{\la^2}{(\la^p+b(z_0)\la^s)^{\frac{2}{p}}}\right)^{\theta q}\left(\fiint_{\mathcal{S}^{\la}_{\rhob}(z_0)}|\nabla u|^{\theta q}\, dz\right)\leq \left(\fiint_{\mathcal{S}^{\la}_{\rhob}(z_0)}|\nabla u|^{\theta q}\, dz\right).
\end{align*}
Therefore, combining all the estimates above, we finally get
\begin{align*}
    I \apprle \left(\fiint_{\mcq^{\la}_{\rhob}(z_0)}|\nabla u|^{\theta q}\, dz\right).
\end{align*}
\noindent \textbf{Case $p<2$ :} We note that, using H\"{o}lder's inequality and $\la^p < \la^p +b(z_0)\la^s,$ \eqref{EQUATION_app3.31} can be expressed as
\begin{align*}
&\la^{(2-p)\theta q} \left(\fiint_{\mathcal{S}^{\la}_{\rhob}(z_0)}|\nabla u|^{\theta q} dz\right)^{p-1}+ \la^{(2-s)\theta q}\left(\fiint_{\mathcal{S}^{\la}_{\rhob}(z_0)}|\nabla u|^{\theta q} dz\right)^{s-1}\nonumber\\
    &\la^{(2-p)\theta q}\la^{(p-q)\theta q}\left(\fiint_{\mathcal{S}^{\la}_{\rhob}(z_0)}|\nabla u|^{\theta q}\, dz\right)^{q-1}+ \la^{(2-p)\theta q}\left(\fiint_{\mathcal{S}^{\la}_{\rhob}(z_0)}|\nabla u|^{\theta q}\, dz\right)^{p-1}.    
\end{align*}
Now, we use Young's inequality as in \cref{p_intrinsic poincare 2} to get \ref{2_lem7.10}. In the scenario, when $q\geq 2$ or $s\geq 2,$ we use the estimate of the case $p\geq 2.$ This completes the proof of \cref{appendix_1st_equation}.   
\end{proof}
\begin{proof}[Proof of \cref{pqs_intrinsic poincare}]
 In this case we have the advantage of having the bounds on $a(z)$ and $b(z)$ in the intrinsic cylinders, that is, $\frac{a(z_0)}{2}\leq a(z)\leq a(z_0)$ and $\frac{b(z_0)}{2}\leq b(z)\leq b(z_0)$ for every $z \in \mathcal{G}^{\la}_{\rhob}(z_0).$   Using these bounds and following similar calculations as in \cref{pq_intrinsic poincare}, we can complete the proof.
\end{proof}
\end{appendices}

\subsection*{Acknowledgement} This research is supported by the Alexander von Humboldt postdoctoral fellowship. The author also thanks Karthik Adimurthi from TIFR CAM, India, for introducing him to the subject. Part of this work has been done at TIFR CAM, India, where the author was a postdoctoral fellow. 


\vspace{.3cm}
\noindent\textbf{Data availability.} This manuscript has no associated data.


\begin{thebibliography}{10}

\bibitem{adi}
K. ~Adimurthi, S. S.~Byun and J. Oh.
\newblock Interior and boundary higher integrability of very weak solutions for quasilinear parabolic equations with variable exponents.
\newblock{\em Nonlinear Anal.}, 194(2020), 111370, 54 pp.

\bibitem{BBO21}
S. Baasandorj, S.S. Byun, and J. Oh.
\newblock Gradient estimates for multi-phase problems.
\newblock{\em Calc. Var. Partial Differential Equations} 60 (2021), no. 3, Paper No. 104, 48 pp.


\bibitem{MR3348922}
P.~Baroni, M.~Colombo, and G.~Mingione.
\newblock Harnack inequalities for double phase functionals.
\newblock {\em Nonlinear Anal.}, 121:206--222, 2015.

\bibitem{MR2779582}
V.~B\"{o}gelein and F.~Duzaar.
\newblock Higher integrability for parabolic systems with non-standard growth
  and degenerate diffusions.
\newblock {\em Publ. Mat.}, 55(1):201--250, 2011.

\bibitem{BL14}
V.~B\"{o}gelein and Q.~Li.
\newblock Very weak solutions of degenerate parabolic systems with non-standard $p(x,t)$-growth.
\newblock{\em Nonlinear Anal.}, 98(2014), 190–225.





\bibitem{MR3985549}
I.~Chlebicka, P.~Gwiazda, and A.~Zatorska-Goldstein.
\newblock Parabolic equation in time and space dependent anisotropic
  {M}usielak-{O}rlicz spaces in absence of {L}avrentiev's phenomenon.
\newblock {\em Ann. Inst. H. Poincar\'{e} C Anal. Non Lin\'{e}aire},
  36(5):1431--1465, 2019.

\bibitem{MR3294408}
M.~Colombo and G.~Mingione.
\newblock Regularity for double phase variational problems.
\newblock {\em Arch. Ration. Mech. Anal.}, 215(2):443--496, 2015.

\bibitem{MR3447716}
M.~Colombo and G.~Mingione.
\newblock Calder\'{o}n-{Z}ygmund estimates and non-uniformly elliptic
  operators.
\newblock {\em J. Funct. Anal.}, 270(4):1416--1478, 2016.

\bibitem{CM15}
M. ~Colombo, and G. ~Mingione.
\newblock Bounded Minimisers of Double Phase Variational Integrals.
\newblock {\em Arch. Rational. Mech. Anal.} 218, 219–273 (2015). https://doi.org/10.1007/s00205-015-0859-9.


\bibitem{DF21}
C.~De~Filippis.
\newblock Optimal gradient estimates for multi-phase integrals.
\newblock {Math. Eng.}, 4 (2022), no. 5, Paper No. 043, 36 pp.

\bibitem{MR3985927}
C.~De~Filippis and G.~Mingione.
\newblock A borderline case of {C}alder\'{o}n-{Z}ygmund estimates for
  nonuniformly elliptic problems.
\newblock {\em Algebra i Analiz}, 31(3):82--115, 2019.

\bibitem{FO19}
C,~De Filippis and J. Oh.
\newblock Regularity for multi-phase variational problems
\newblock {\em J. Differential Equations}, 267 (2019), no. 3, 1631–1670.

\bibitem{MR1230384}
E.~DiBenedetto.
\newblock {\em Degenerate parabolic equations}.
\newblock Universitext. Springer-Verlag, New York, 1993.

\bibitem{MR2076158}
L.~Esposito, F.~Leonetti, and G.~Mingione.
\newblock Sharp regularity for functionals with {$(p,q)$} growth.
\newblock {\em J. Differential Equations}, 204(1):5--55, 2004.

\bibitem{FRZZ22}
Y. Fang, V. D. Rădulescu, C. Zhang and X. Zhang.
\newblock Gradient estimates for multi-phase problems in Campanato spaces.
\newblock {Indiana Univ. Math. J.} 71(2022), no.3, 1079–1099.

\bibitem{FD84}
A.~Friedman and E.~DiBenedetto.
\newblock Regularity of solutions of nonlinear degenerate parabolic systems.
\newblock{\em Journal für die reine und angewandte Mathematik}, vol. 1984, no. 349, 1984, pp. 83-128. https://doi.org/10.1515/crll.1984.349.83.

\bibitem{FD85}
A.~Friedman and E.~DiBenedetto.
\newblock H\"{o}lder estimates for nonlinear degenerate parabolic Systems.
\newblock{\em  J. Reine Angew. Math.}, 357 (1985), 1-22.

\bibitem{FD185}
\newblock Addendum to: H\"{o}lder estimates for nonlinear degenerate parabolic systems.
\newblock{\em J. Reine Angew. Math.}, 363 (1985), 217-220.

\bibitem{GH}
F. W.~Gehring. 
\newblock The $L^p$-integrability of the partial derivatives of A quasiconformal mapping.
\newblock {\em Acta Math.} 130, 265–277 (1973). https://doi.org/10.1007/BF02392268.



\bibitem{MR4302665}
P.~H\"{a}st\"{o} and J.~Ok.
\newblock Higher integrability for parabolic systems with {O}rlicz growth.
\newblock {\em J. Differential Equations}, 300:925--948, 2021.



\bibitem{KO2024}
B.~Kim and J.~Oh.  
\newblock Higher integrability for weak solutions to parabolic multi-phase equations.
\newblock {\em J. Differential Equations, 409}, (2024), 223–298.

\bibitem{KOS2024}
B.~Kim, J.~Oh and A.~Sen.
\newblock Parabolic Lipschitz truncation for multi-phase problems.
\newblock{\em preprint.}, 2024.

\bibitem{KKS}
W.~Kim, J.~Kinnunen, and L.~Särkiö.
\newblock Lipschitz truncation method for the parabolic double-phase system and applications.
\newblock {\em To appear, J. Funct. Anal.}, 2024.

\bibitem{KKM}
W.~Kim, J.~ Kinnunen and K.~ Moring.
\newblock Gradient Higher Integrability for Degenerate Parabolic Double-Phase Systems. 
\newblock {\em Arch. Rational. Mech. Anal. 247, 79}, 2023,
\newblock {https://doi.org/10.1007/s00205-023-01918-0}.

\bibitem{KS2024}
W.~Kim, L.~ Särkiö.
\newblock Gradient higher integrability for singular parabolic double-phase systems.
\newblock {\em Nonlinear. Differ. Equ. Appl. 31, 40}, 2024,
\newblock {https://doi.org/10.1007/s00030-024-00928-5}.

\bibitem{MR1749438}
J.~Kinnunen and J.~L. Lewis.
\newblock Higher integrability for parabolic systems of {$p$}-{L}aplacian type.
\newblock {\em Duke Math. J.}, 102(2):253--271, 2000.

\bibitem{KLveryweak}
J.~Kinnunen and J.~L. Lewis.
\newblock Very weak solutions of parabolic systems of {$p$}-Laplacian type.
\newblock{\em Ark. Mat.}, 40(1): 105-132 (April 2002). DOI: 10.1007/BF02384505.

\bibitem{Li17}
Q. ~Li.
\newblock Very weak solutions of subquadratic parabolic systems with non-standard $p(x,t)$-growth.
\newblock{\em Nonlinear Anal.}, 156(2017), 17–41.



\bibitem{ME}
N. G.~Meyers and A.~Elcart.
\newblock Some results on regularity for solutions of non-linear elliptic systems and quasi-regular functions.
\newblock {\em Duke Math. J.}, 42(1): 121-136 (March 1975). DOI: 10.1215/S0012-7094-75-04211-8





\bibitem{MR3532237}
T.~Singer.
\newblock Existence of weak solutions of parabolic systems with {$p,
  q$}-growth.
\newblock {\em Manuscripta Math.}, 151(1-2):87--112, 2016.





	
\end{thebibliography}
\end{document}